\documentclass[12pt]{amsart}
\usepackage{times}
\usepackage{anysize}
\marginsize{2.7cm}{2.44cm}{2.0cm}{3.0cm}
\usepackage[small,it]{caption}
 \usepackage[latin1]{inputenc}
 \usepackage[dvips]{graphicx}
 \usepackage{wrapfig}
 \usepackage{amsmath}
 \usepackage{leftidx}
 \usepackage{amsthm}
 \usepackage{amsfonts}
 \usepackage{amssymb}
 \usepackage{layout
} \usepackage{verbatim}
 \usepackage{alltt}
\usepackage{yfonts}

\usepackage[all]{xy}
\usepackage{setspace}
\newtheorem*{ithm}{Theorem}
\newtheorem*{thma}{Theorem A}
\newtheorem*{thmb}{Theorem B}
\newtheorem*{thmc}{Theorem C}
\newtheorem*{thmd}{Theorem D}

\newtheorem*{claim}{Claim}

\newcommand{\LL}{\Lambda}
\newcommand{\TT}{\mathbb{T}}
\newcommand{\QQ}{\mathbb{Q}}
\newcommand{\FF}{\mathcal{F}}

\newcommand{\lra}{\longrightarrow}
\newcommand{\ZZ}{\mathbb{Z}}
\newcommand{\PP}{\mathcal{P}}
\newcommand{\MM}{\mathcal{M}}

\newcommand{\NN}{\mathcal{N}}
\newcommand{\ra}{\rightarrow}

\newcommand{\be}{\begin{equation}}
\newcommand{\ee}{\end{equation}}

\newcommand{\XX}{\mathcal{X}}

\newcommand{\al}{\mathcal{L}}

\newcommand{\oo}{\mathcal{O}}
\newcommand{\RR}{\mathcal{R}}

\newcommand{\mm}{\hbox{\frakfamily m}}

\newcommand{\FFc}{\mathcal{F}_{\textup{\lowercase{can}}}}

\newcommand{\BK}{\textup{BK}}

\numberwithin{equation}{section}
\newtheorem{thm}{Theorem}[section]
\newtheorem{lemma}[thm]{Lemma}
\newenvironment{define}{\par\medskip\noindent\refstepcounter{thm}
\bgroup{\hspace*{-0.15 cm}\bf{Definition}
\thethm.}\bgroup}{\egroup \egroup\par\medskip}
\newenvironment{question}{\par\medskip\noindent\refstepcounter{thm}
\bgroup{\hspace*{-0.15 cm}\bf{Question}
\thethm.}\bgroup}{\egroup \egroup\par\medskip}
\newtheorem{prop}[thm]{Proposition}
\newtheorem{cor}[thm]{Corollary}
\newtheorem{conj}[thm]{Conjecture}
\newenvironment{rem}{\par\medskip\noindent\refstepcounter{thm}
\bgroup{\hspace*{-0.15 cm}\bf{Remark} \thethm.}\bgroup}{\egroup
\egroup\par\medskip} \parskip 2pt

\newenvironment{example}{\par\medskip\noindent\refstepcounter{thm}
\bgroup{\hspace*{-0.15 cm}\bf{Example}
\thethm.}\bgroup}{\egroup \egroup\par\medskip}

\newcounter{Athm}[section]

\renewcommand{\theAthm} {\arabic{Athm}}

\begin{document}
\title{{D}\lowercase{eformations} \lowercase{of} {K}\lowercase{olyvagin systems}}

\author{K\^az\i m B\"uy\"ukboduk}

\email{kazim@math.stanford.edu}
\address{K\^az\i m B\"uy\"ukboduk \hfill\break\indent Ko\c{c} University Mathematics, 34450 Sariyer, Istanbul, Turkey
\hfill\break\indent
\,\,\,\,\,\,\,\,\,\,\,\,\,\,\,\,\,\,\,\,\,\,\,\,\,\,\,\,\,\,\,\,\,\,\,\,\,\,\,\,\,\,\,\,\,\,\,\,\,\,\,\,\,\,\,\,\,\,\,\,\,\,\,\,\,\,\,\,\,\,\,\,\,\,\,\,\,\,\,\,\,\,\,\,\,\,\,\,\,\,\,\,\,\,\,\,\,and
\hfill\break
\indent MPIM Vivatsgasse 7, 53111 Bonn, Deutschland\vspace{1em}}


\begin{abstract}
Ochiai has previously proved that the Beilinson-Kato Euler systems for modular forms interpolate in nearly-ordinary $p$-adic families (Howard has obtained a similar result for Heegner points), based on which he was able to prove a half of the two-variable main conjectures. The principal goal of this article is to generalize Ochiai's work in the level of Kolyvagin systems so as to prove that Kolyvagin systems associated to Beilinson-Kato elements interpolate in the full deformation space (in particular, beyond the nearly-ordinary locus) and use what we call \emph{universal Kolyvagin systems} to attempt a main conjecture over the eigencurve. Along the way, we utilize these objects in order to define a quasicoherent sheaf on the eigencurve that behaves like a $p$-adic $L$-function (in a certain sense of the word, in $3$-variables).

\end{abstract}

\maketitle
\tableofcontents
\section{Introduction}
\label{sec:intro}
Fix forever a prime $p>2$. {Classical} Iwasawa theory concerns the variation of arithmetic invariants of number fields in a $\ZZ_p$-tower. Mazur (resp., Greenberg) extended this study to abelian varieties (resp., to motives) along a $\ZZ_p$-extension. All these may be considered within the framework of Mazur's general theory of Galois deformations and hand in hand with this perspective, Greenberg in \cite[Conjecture 4.1]{greenbergdeform} formulated a ``main conjecture'' for a $p$-adic deformation of a motive. 

The results of this article are closely related to Greenberg's main conjecture: Given a `big' Galois representation (attached to a $p$-{adic deformation of a motive}), we prove in a wide variety of cases that an associated `big' Kolyvagin system exists as well (Theorem A below). As a rather standard application, we deduce that the relevant \emph{strict} Selmer groups are controlled by the big Kolyvagin systems we prove to exist (Theorem C). When the motive in question is that attached to an elliptic  modular form, we show that the \emph{universal} Kolyvagin system we prove to exist interpolates the Beilinson-Kato Kolyvagin systems associated to the `modular points' of the $p$-adic deformation space (Theorem B)  in a suitable sense. Furthermore, we use its leading term in order to define what we call the \emph{sheaf of universal $p$-adic $L$-function} (Definition~\ref{def:universalpadicLfunction} and Theorem~D). See Section~\ref{intro:univpadicL} for further discussion on (potential) applications of all these towards a main conjecture over the eigencurve. 

Before we explain our results in detail,  we provide a quick overview of Mazur's theory of Galois deformations in order to motivate for our main technical result; see \cite{mazurdeform, lenstradeform, gouveadeform} for details.
\subsection{Deformations of Galois representations}
Let $\Phi$ be a finite extension of $\QQ_p$ and $\oo$ be the ring of integers of $\Phi$. Let $\varpi \in \oo$ be a uniformizer, and let $\texttt{k}=\oo/{\varpi}$ be its residue field. Consider the following category $\mathbf{C}$:
\begin{itemize}
\item Objects of $\mathbf{C}$ are complete, local, Noetherian commutative $\oo$-algebras $A$ with residue field $\texttt{k}_A=A/\mm_A$ isomorphic to $\texttt{k}$, where $\mm_A$ denotes the maximal ideal of $A$.
\item A morphism $f:A\ra B$ in $\mathbf{C}$ is a local $\oo$-algebra morphism.
\end{itemize}

Let $\Sigma$ be a finite set of places of $\QQ$ that contains $p$ and $\infty$. Let $G_{\QQ,\Sigma}$ denote the Galois group of the maximal extension $\QQ_\Sigma$ of $\QQ$ unramified outside $\Sigma$. Fix an \emph{absolutely irreducible}, continuous Galois representation $\overline{\rho}: G_{\QQ,\Sigma}\ra \textup{GL}_n(\texttt{k})$ and let $\overline{T}$ be the representation space (so that $\overline{T}$ is an $n$-dimensional $\texttt{k}$-vector space on which $G_{\QQ,\Sigma}$ acts continuously).

Let ${D}_{\overline{\rho}}: \mathbf{C}\lra \texttt{Sets}$ be the functor defined as follows. For every object $A$ of $\mathbf{C}$, ${D}_{\overline{\rho}}(A)$ is the set of continuous homomorphisms $\rho_A:G_{\QQ,\Sigma}\ra \textup{GL}_n(A)$
 that satisfy $\rho_A\otimes_A k\cong \overline{\rho}$, taken modulo conjugation by the elements of $\textup{GL}_n(A)$. For every morphism $f:A\ra B$ in $\mathbf{C}$, ${D}_{\overline{\rho}}(f)(\rho_A)$ is the $\textup{GL}_n(B)$-conjugacy class of $\rho_A\otimes_A B$.

 \begin{ithm}[Mazur]
 The functor ${D}_{\overline{\rho}}$ is representable.
 \end{ithm}
 In other words, there is a ring $R(\overline{\rho})\in \textup{Ob}(\mathbf{C})$ and a continuous representation
 $$\pmb{\rho}: G_\QQ\lra \textup{GL}_n(R(\overline{\rho}))$$
 such that for every $A \in \textup{Ob}(\mathbf{C})$ and any continuous representation $\rho_A:G_{\QQ,\Sigma}\ra \textup{GL}_n(A)$, there is a unique morphism $f_A: R(\overline{\rho})\ra A$ for which we have $\pmb{\rho}\otimes_{R(\overline{\rho})}A\cong \rho_A.$ The ring $R(\overline{\rho})$ is called the universal deformation ring and $\pmb{\rho}$ the universal deformation of $\overline{\rho}$.

Let $\textup{ad}\,\overline{\rho}$ be the adjoint representation. We say that the deformation problem for $\overline{\rho}$ is \emph{unobstructed} if the following hypothesis holds true:\\

(\textbf{H.nOb}) $H^2(G_{\QQ,\Sigma},\textup{ad}\,\overline{\rho})=0$.\\

 When (\textbf{H.nOb}) holds true, Mazur proved that  $R(\overline{\rho})\cong \oo[[X_1,\cdots,X_d]]$, where $d$ is the dimension of the ${\texttt{k}}$-vector space $H^1(G_{\QQ,\Sigma},\textup{ad}\,\overline{\rho})$.

\subsubsection{$p$-ordinary families}
One might also consider a subclass of deformations of a given $\overline{\rho}$, rather than the full deformation space $R(\overline{\rho})$. The following paragraph illustrates a particular case which has been much studied by many authors, see Section~\ref{subsec:example2} for a perspective offered through the general theory we develop here. 

Suppose $\overline{\rho}: G_{\QQ,\Sigma}\lra \textup{GL}_2(\texttt{k})$ is $p$-ordinary and a $p$-distinguished, in the sense that the restriction of $\overline{\rho}$ to a decomposition group at $p$ is reducible and non-scalar. Assume further that $\overline{\rho}$ is odd, i.e., $\det(\rho)(c)=-1$ where $c$ is any complex conjugation. Then Serre's conjecture~\cite{serreconj} (as proved in~\cite{kharewint,kisinserre}) implies that $\overline{\rho}$ arises from an ordinary  newform. Hida associates in \cite{hidafamily,hida} such $f$ a family of ordinary modular forms and a Galois representation $\mathcal{T}$ attached to the family, with coefficients in the \emph{universal ordinary Hecke algebra} $\frak{H}$.  Thanks to the ``$R=T$'' theorems proved in~\cite{wiles, taylorwiles} (and their refinements) it follows that $\frak{H}$ is the universal ordinary deformation ring of $\overline{\rho}$, parametrizing all ordinary deformations of $\overline{\rho}$. Ochiai in~\cite{ochiaideform} (resp., Howard in~\cite{how2}) has studied the Iwasawa theory of this family of Galois representations by interpolating Kato's Euler system (resp., Heegner points) for each member of the family to a `big' Euler system for the whole $p$-ordinary family.

\subsection{The results} One of the main goals of the current article is to generalize the works of Ochiai and Howard so as to treat the full deformation ring (e.g., not necessarily its $p$-ordinary locus) of a  mod $p$ Galois representation. Our approach is altogether different from theirs: Instead of interpolating Euler systems (as in \cite{ochiaideform,how2}), we instead deform Kolyvagin systems. We remark that a Kolyvagin system has exactly the same use as an Euler system, when they are used to bound Selmer groups. We recall the definition of a Kolyvagin system in Section~\ref{subsec:kolsysforbig} below. See also \cite[\S3.2]{mr02} for the relation between Euler systems and Kolyvagin systems over a DVR or over the cyclotomic Iwasawa algebra $\LL=\ZZ_p[[\Gamma]]$, where $\Gamma=\textup{Gal}(\QQ_\infty/\QQ)$ is the Galois group of the (unique) $\ZZ_p$-extension $\QQ_\infty/\QQ$.


 \subsubsection{The Setup}
 \label{subsubsec:setup} In this paper we will study the deformation problem for Kolyvagin systems to one of the following choices of rings:
\begin{itemize}
\item[(i)]$\frak{R}=\RR[[\Gamma]]$, where $\RR$ is a dimension-2 Gorenstein $\oo$-algebra with a regular sequence $\{\varpi,X\}$ such that $\RR/X$ is a finitely generated torsion-free $\oo$-module.
\item[(ii)] ${R}=\oo[[X_1,X_2,X_3]]$ (which we shall arise as an \emph{unobstructed} universal deformation ring of a two dimensional mod $\varpi$ Galois representation $\overline{\rho}$ in examples.)
\end{itemize}

Let $\mm$ (resp., $\mathcal{M}$) be the maximal ideal of $\frak{R}$ (resp., of $R$) and $\texttt{k}=\frak{R}/\mm$ (resp., $R/\mathcal{M}$ which we also denote by $\texttt{k}$ as there will be no danger of confusion) be the residue field. When the coefficient ring we are interested in is the ring $\frak{R}$ as in (i) above, we let $\mathcal{T}$ be a free $\RR$-module of finite rank which is endowed with a continuous $G_\QQ$-action, unramified outside a finite set of primes. Set $\frak{T}=\mathcal{T}\otimes_{\ZZ_p}\LL$, where we allow $G_\QQ$ act on both factors. When the coefficient ring we are interested in is ${R}$ as in (ii), we let $\TT$ be a free ${R}$-module of finite rank endowed with a continuous $G_\QQ$-action unramified outside a finite number of primes. Let $\Sigma$ be a finite consisting of primes at which $\TT$ is ramified, $p$ and $\infty$.  Let $\QQ_{\Sigma}$ denote the maximal extension of $\QQ$ unramified outside $\Sigma$. We set $\overline{T}=\frak{T}/\mm$ (resp., $\overline{T}=\TT/\MM$) and define $\chi(\frak{T})=\dim_{\texttt{k}}\overline{T}^{-}$ (resp., $\chi(\TT)$), where $\overline{T}^-$ is the $(-1)$-eigensubspace of $\overline{T}$ under the action of a fixed complex conjugation.

Let $\pmb{\mu}_{p^\infty}$ be the $p$-power roots of unity and for any $\oo[[G_{\QQ}]]$-module $M$, we let $M^*=\textup{Hom}\left(M,\pmb{\mu}_{p^\infty}\right)$ denote its Cartier dual.

 The following hypotheses will play a role in what follows:\\
(\textbf{H1}) $\overline{T}$ is an absolutely irreducible $G_\QQ$-module.\\
(\textbf{H2}) There is a $\tau \in G_\QQ$ such that $\tau$ acts trivially on $\pmb{\mu}_{p^\infty}$ and the $R$-module $\TT/(\tau-1)\TT$ \par(resp., the $\frak{R}$-module $\frak{T}/(\tau-1)\frak{T}$) is free of rank one.\\
(\textbf{H3}) $H^0(\QQ,\overline{T})=H^0(\QQ,\overline{T}^*)=0$.\\
(\textbf{H4}) Either $\textup{Hom}_{\mathbb{F}_p[[G_\QQ]]}(\overline{T},\overline{T}^*)=0$, or $p>4$.\\
(\textbf{H.Tam}) For all bad primes $\ell \neq p$,
 \begin{itemize} \item[(i)] $H^0(\QQ_\ell,\overline{T})=0$.
\item[(ii)] $H^0(I_\ell,A)$ is $p$-divisible.  \end{itemize}
(\textbf{H.nA}) $H^0(\QQ_p,\overline{T}^*)=0$.
\begin{rem}
The hypotheses (\textbf{H1})-(\textbf{H4}) are also present in \cite{mr02}. (\textbf{H.Tam}) will be used in Section~\ref{subsec:cartesianprop} to check that the unramified local conditions are \emph{cartesian}\footnote{Most importantly in the proof that the map $\beta$ that appears in Lemma~\ref{lem:alphainj} is injective. Although one may possibly verify this fact under less restrictive hypothesis, the assumption (\textbf{H.Tam}) is simple to state, easy to check and thus allows one to produce many interesting examples where Theorem A applies; c.f., Proposition~\ref{prop:tamimpliestam} below.} in a sense to be made precise.
\end{rem}
\begin{rem} 
When $\mathcal{T}$ is the self-dual Galois representation attached to a (twisted) Hida family with coefficients in the universal ordinary Hecke algebra $\RR$ (as studied in \cite{how2}), the hypothesis (\textbf{H.Tam})(ii) asks that there is a single member $f$ of the twisted Hida family such that the local Tamagawa number $c_{\ell}(f)$, defined as in \cite[\S I.4.2.2]{FoPR}, is prime to $p$. As explained in \cite[\S3]{kbbheegner}, this in turn implies that the Tamagawa number $c_\ell(g)$ is prime to $p$ for every member $g$ of the twisted family. 

See Section~\ref{subsec:example1} (particularly, Proposition~\ref{prop:tamimpliestam} and Remark~\ref{rem:nonanamolous}) for a discussion of the content of the hypotheses (\textbf{H.Tam}) and (\textbf{H.nA}) when $\overline{T}$ is the mod $p$ Galois representation attached to (the central critical twist of) a cuspidal elliptic modular newform $f$. We note here only the fact that these hypotheses simultaneously hold true for infinitely many primes $p$  in that setting, so our results below are not vacuous.
\end{rem}
\subsubsection{``Big'' Kolyvagin systems}
For a $G_{\QQ,\Sigma}$-representation $\TT$ (resp., $\frak{T}$) as in Section \ref{subsubsec:setup}, we let $\overline{\mathbf{KS}}(\TT,\FFc,\PP)$ (respectively, $\overline{\mathbf{KS}}(\frak{T},\FFc,\PP)$) denote the $R$-module (respectively, the $\frak{R}$-module) of \emph{big} Kolyvagin systems for the \emph{canonical Selmer structure} $\FFc$ on $\TT$ (respectively, for $\frak{T}$). See Sections \ref{sec:selmerstr} and \ref{subsec:kolsysforbig} for precise definitions of these objects.
\begin{thma}[Theorem~\ref{thm:KSmain} below]
\label{thm:mainintro}
Suppose $\chi(\frak{T})=\chi(\TT)=1$. Assume that the hypotheses $($\textbf{\textup{H1)} - $($\textbf{\textup {H4}}}$)$, $($\textbf{\textup{H.Tam}}$)$ and $($\textbf{\textup{H.nA}}$)$ hold true. Then,
\begin{itemize}
\item[(i)] the $R$-module $\overline{\mathbf{KS}}(\TT,\FFc,\PP)$ is free of rank one, generated by a Kolyvagin system $\pmb{\kappa}$ whose image $\bar{\pmb{\kappa}} \in {\mathbf{KS}}(\overline{T},\FFc,\PP)$ is non-zero,\\
\item[(ii)]the $\frak{R}$-module $\overline{\mathbf{KS}}(\frak{T},\FFc,\PP)$ is free of rank one. When the ring $\RR$ is regular, the module $\overline{\mathbf{KS}}(\frak{T},\FFc,\PP)$ is generated by $\pmb{\kappa}$ whose image $\bar{\pmb{\kappa}} \in {\mathbf{KS}}(\overline{T},\FFc,\PP)$ is non-zero.
\end{itemize}
\end{thma}

In Theorem A, the Galois modules $\TT$ (resp., $\frak{T}$) should of course be thought of as a family of Galois representations and the conclusion of Theorem~A as an assertion that the Kolyvagin systems for each individual member of the family $\TT$ (resp., $\frak{T}$) do interpolate so as to give rise to a \emph{big} Kolyvagin system.
\begin{rem}
\label{rem:MRdarmonconj}
The \emph{rigidity principle} that the Kolyvagin systems exhibit as in Theorem~A has been crucially used in his recent work~\cite{kbbssCM1} on the Iwasawa theory of abelian varieties for supersingular primes. More precisely, the author appeals to Theorem A in the verification that certain bounds in the signed main conjectures for CM elliptic curves over general totally real fields are indeed sharp, in situations where we do not have direct access to analytic class number formulas (which were traditionally employed so as to promote various inequalities in the statements of the main conjectures to equalities). Note also that a variant of Theorem A plays a further important role in \cite{kbbssCM1, kbbCMabvar}, providing evidence for the Rubin-Stark conjectures.  We finally remark that the rigidity phenomenon for Kolyvagin systems has also played a pivotal role in the proof of Darmon's conjecture by Mazur and Rubin in \cite{mrdarmon}.
\end{rem}
\begin{rem}
\label{rem:LambdaadicKS}
The author has proved (as part of his Ph.D. thesis) a primal version of Theorem A when $R$ is the cyclotomic Iwasawa algebra $\LL$ and $\TT=T\otimes\LL$ is the cyclotomic deformation of $T$. Barry Mazur had asked us then to work out a generalization of this result to other types of deformation rings (and our long due response is Theorem A). Note that Theorem~A is not only much more general than \cite[Theorem 3.23]{kbb} in its scope, it also allows us to gain insight for the Iwasawa theory on the eigencurve (as we outline in Section~\ref{intro:univpadicL}). We further remark that the proof of Theorem~A is technically much more demanding. One reason is the fact that the behavior of unramified cohomology for an arbitrary family of Galois representations at bad primes $\ell\neq p$ is much more complicated, as compared with the situation when the family is given as twists of $T$ by powers of the cyclotomic character. (Note that this matter also makes its appearance in the statement of \cite[Theorem 3.13]{pottharst}.) \end{rem}

\begin{rem} We expect that the arguments used in the proof of Theorem A should generalize without much effort to handle a general regular ring (not necessarily of relative dimension $3$ over $\oo$, as $R$ above is). Note on the other hand that for explicit arithmetic applications of our theorem, we would like that the big Kolyvagin system for $\TT$ (resp., for $\frak{T}$) of Theorem~A interpolates Kolyvagin systems which are explicitly related to $L$-values. At the time when this article was written up, this was only possible when the residual representation $\overline{T}$ is two dimensional and $\chi(\overline{T})=1$, in which case the unobstructed universal deformation ring is $R$. Note that when $\overline{T}$ is two dimensional and $\chi(\overline{T})=1$, we know that $\overline{T}$ is modular. A density theorem due to B\"ockle shows that the big Kolyvagin system for the universal deformation $\TT$ indeed interpolates the Beilinson-Kato Kolyvagin systems for elliptic modular forms whose associated Galois representations are congruent mod $p$ to $\overline{T}$. See Theorems~\ref{thm:interpolation}, \ref{thm:univKSvsOchiaiKS} and \ref{thm:firstquestionhasaffrimativeanswerifconjholdstrue} below for details.
\end{rem}
 \subsubsection{Universal Kolyvagin systems and Beilinson-Kato elements}
 \label{subsub:universalKS}  Let $E_{/\QQ}$ be an elliptic curve, $\overline{T}=E[p]$ be the $p$-torsion
subgroup of $E(\bar{\QQ})$ and  $\overline{\rho}=\overline{\rho}_E$ the mod $p$ Galois representation on
$\overline{T}$. Let $\Sigma$ be the set of primes that consists of primes at which $E$ has bad reduction, $p$ and $\infty$. Define also $R=R(\overline{\rho})$ to be the universal deformation ring of
$\overline{\rho}$. The universal deformation representation $\TT$ is then a free $R$-module of
rank two.  Until the end of this introduction, we concentrate on this particular representation $\TT$ and give applications of Theorem A. See Sections~\ref{subsec:example1}-\ref{subsub:generalex} below for the most general form of the results we record in \S\ref{subsub:universalKS}-\ref{intro:univpadicL} (where we treat cuspidal elliptic newforms of weight $>2$ as well). 
\begin{ithm}[Flach, \cite{flach}]
\label{example:flach}
Suppose that
\begin{itemize}
\item[(\textbf{F1})] $\overline{\rho}_E$ is surjective,
\item[(\textbf{F2})] $H^0(\QQ_\ell,\overline{T}\otimes \overline{T})=0$ for all $\ell \in \Sigma$,
\item[(\textbf{F3})] $p$ does not divide $\Omega^{-1}L(\textup{Sym}^2(E),2)$, where $\Omega=\Omega(\textup{Sym}^2(E),2)$ is the transcendental period.
\end{itemize}
Then $E$ satisfies \textup{(\textbf{H.nOb})} and $R\cong\ZZ_p[[X_1,X_2,X_3]]$.
\end{ithm}
\emph{In addition to the assumptions \textup{(\textbf{F1}) - (\textbf{F3})}, we suppose throughout the introduction  that $\TT$ satisfies the hypotheses
\textup{(\textbf{H1}) - (\textbf{H4}), (\textbf{H.Tam})(i)} and \textup{(\textbf{H.nA}).}} See \S\ref{subsec:example1} for the content of these assumptions in a variety of cases of interest.

\begin{rem}
\label{example:weston}
Weston has proved the following result, which is an important generalization of Flach's theorem to modular forms of higher weight. Let $f$ be an elliptic newform of level $N$, weight $k>2$ and character $\psi$. Let $K$ be the number field generated by the Fourier coefficients of $f$ and $\oo_K$ be its ring of integers. For a prime $\wp$ of $K$ above $p$, let $\texttt{k}=\oo_K/\wp$ and $\oo=W(\texttt{k})$, the Witt vectors of $\texttt{k}$.  Let
$$\overline{\rho}=\overline{\rho}_{f,\wp}:G_{\QQ,\Sigma}\lra \textup{GL}_2(\texttt{k})$$
be the Galois representation attached to $f$ by Deligne.  Weston proved in~\cite{westonunobstructed} that  the deformation problem for $\overline{\rho}$ is unobstructed and $R(\overline{\rho})\cong \oo[[X_1,X_2,X_3]]$, for almost all choices of a prime $\wp$ of $K$. Using Weston's theorem, all our results proved in weight two generalizes \emph{verbatim} to higher weights.
\end{rem}

We next state a consequence of Theorem A, which may be thought of an extension of a result of Ochiai (on the interpolation of Beilinson-Kato \emph{Euler system} in $p$-ordinary familes), but beyond the ordinary locus. Let us first set our notation. Let $g=\sum a_nq^n$ be a newform of weight $\kappa\geq 2$ and let
$\oo_g$ be the finite flat normal extension of $\ZZ_p$ which the
$a_n$'s generate. Let
$$\rho_f:G_{\QQ}\lra\textup{GL}_2(\oo_g)$$
be Deligne's Galois representation attached to $g$ and $T_g$ be
the free $\oo_g$-module of rank two on which $G_{\QQ}$ acts via
$\rho_f$. Suppose that $\overline{\rho}_f\cong \overline{\rho}$, so that by the
universality of $R$ there is a ring homomorphism $\psi_g:R\ra \oo_g$
which induces an isomorphism $\TT\otimes_{\psi_g}\oo_g\cong T_g$
and a map
$$\overline{\mathbf{KS}}(\TT,\FFc,\PP)\stackrel{\psi_g}{\lra}
\overline{\mathbf{KS}}(T_g,\FFc,\PP).$$
For each $g$ as above, the Beilinson-Kato Euler system of \cite{ka1} also gives rise to a Kolyvagin
system $\pmb{\kappa}^{\textup{BK},g} \in
\overline{\mathbf{KS}}(T_g,\FFc,\PP)$ (c.f., \cite[\S3.2]{mr02}). 

We call any generator $\pmb{\kappa}$ of the cyclic $R$-module $\overline{\mathbf{KS}}(\TT,\FFc,\PP)$ a \emph{universal Kolyvagin system}. Universal Kolyvagin systems interpolate the Beilinson-Kato Kolyvagin systems in the following sense:
\begin{thmb}[Theorem~\ref{thm:interpolation}]  Let $\pmb{\kappa}$  be a universal Kolyvagin system. For every newform $f$ as above, we have ${\kappa}^{\textup{Kato},g}=\lambda_g\cdot\psi_g(\pmb{\kappa})$
for some $\lambda_g \in \oo_g$.
\end{thmb}
\begin{rem}
\label{ref:remglobalfunctiononR}
It would be very interesting to know whether the interpolation factors $\lambda_g$ above \emph{interpolate} to give rise to a global function on $\textup{Spec}(R)$. In Section~\ref{subsec:example2}, we are able to answer to this question affirmatively over Hida's nearly-ordinary locus (denoted by $\textup{Spec}(\frak{R}) \subset \textup{Spec}(R)$ in the main text below). We also elaborate on this question over  affinoid subdomains of the eigencurve in Section~\ref{subsub:generalex} (see also Section~\ref{intro:univpadicL}) and verify that it has a positive answer in that case as well, assuming the truth of Conjecture~\ref{conj:bigColeman}\footnote{As far as the author understands, much of this conjecture has been settled by David Hansen.}.

The question raised in this remark is relevant in the study of $p$-adic variation of the Iwasawa invariants associated to the members of the family $\TT$, much in the spirit of \cite{emertonpollackweston} and \cite{ochiaideform}. In order to make use of Theorems A, B above and C below in this vain, it seems that one would need a positive answer to this question. See Remark~\ref{rem:lamdaareunits} below for a continuation of the discussion in the realm of Hida families, dwelling on the main results of \cite{skinnerurbanmainconj,skinnersplitcyclo}. This theme is the essence of Theorem~\ref{thm:univKSvsOchiaiKS}, Corollaries~\ref{cor:nearlyordkappa1andkitagawapadicLfunc} and \ref{cor:twovarmainconjMTTmainconj} (for Hida families); and Theorem~\ref{thm:firstquestionhasaffrimativeanswerifconjholdstrue}, Remark~\ref{rem:globalcolemantrivialization} (for families of Galois representations carried by the eigencurve) in the main body of the text. 
\end{rem}

\subsubsection{Universal Kolyvagin system and the Greenberg-Kato main conjecture}
\label{subset:greenbergmainconj}
We still work in the setting of Section~\ref{subsub:universalKS}. The next result we present (more specifically, its final part) is the fundamental application of the universal Kolyvagin
system whose existence is guaranteed by Theorem A. Let $\FFc^*$ denote the
dual Selmer structure on $\TT^*$ (in the sense of Definition~\ref{def:dualselmerstructure}). For any abelian group $N$, let $N^{\vee}$ denote the Pontryagin dual. If
$M$ is a finitely generated torsion $R$-module, set
$$\textup{char}(M)=\prod_{\frak{p}}\frak{p}^{\textup{length}(M_{\frak{p}})}$$
where the product is over height one primes of $R$.

Given a Kolyvagin system $\pmb{\kappa}\in\overline{\mathbf{KS}}(\TT,\FFc,\PP)$, let $\kappa_1\in H^1_{\FFc}(\QQ,\TT)$ denote its \emph{leading term}; see Definition~\ref{def:leadingtermofKS} for a precise definition of this notion.
\begin{thmc}Let $\pmb{\kappa}$ be a universal Kolyvagin system for $\TT$. Then, 
\begin{itemize}
\item[(i)] $\kappa_1$ is not $R$-torsion, 
\item[(ii)] The $R$-module $H^1_{\FFc^*}(\QQ,\TT^*)$ is cotorsion and $H^1_{\FFc}(\QQ,\TT)$ is free of rank one.
\item[(iii)] $\textup{char}\left(H^1_{\FFc^*}(\QQ,\TT^*)^{\vee}\right)=
\textup{char}\left(H^1_{\FFc}(\QQ,\TT)/R\cdot\kappa_1\right).$
\end{itemize}
\end{thmc}

See Proposition~\ref{prop:BKleadingtermnonzero} for (i), Theorems~\ref{thm:weakleo2} and \ref{thm:h1free} for (ii), Theorem~\ref{thm:sharpening} for (iii). We actually prove more in the main text: For a general Galois representation $\TT$ over a general regular ring $R$, we in fact prove in Theorems~\ref{thm:weakleo1} and \ref{thm:weakleo2} that the statements in (i) and (ii) are equivalent to each other. In other words, the weak Leopoldt conjecture for $\TT$ is \emph{equivalent to} (i.e., not only implied by) the non-vanishing the leading term of a deformed Kolyvagin system.

It is also worth taking a note of the systematic use of Nekov\'a\v{r}'s descent procedure (that he has developed in the context of his Selmer complexes) in the proof of Theorem~C(iii) (which is Theorem~\ref{thm:sharpening} in the main text below). This is one of the fundamental additions to the tools utilized in prior works~\cite{mr02,kbb} on the general theory of Kolyvagin systems.

In view of Theorem~B, the statement of (iii) is closely related to Greenberg's main conjecture~\cite{greenbergdeform} on the Iwasawa theory of $p$-adic deformations of motives. The reader is invited to compare our statement especially with Kato's formulation of the \emph{main conjecture without $p$-adic $L$-functions} in \cite{katolectturesonapproach}. See also Corollaries~\ref{cor:nearlyordkappa1andkitagawapadicLfunc} and \ref{cor:twovarmainconjMTTmainconj} for applications of this statement to the main conjecture for a nearly-ordinary family (which are originally due to Ochiai). More relevant to the main purpose of this article, see Section~\ref{subsub:generalex} where we construct utilizing the universal Kolyvagin system (and relying on the works of \cite{kpx,rliu}) what we call the \emph{quasicoherent sheaf of universal $p$-adic $L$-function}. We hope that this construction would be useful in deducing a form of main conjecture (\emph{with $p$-adic $L$-function}) over affinoid subdomains of the eigencurve.

\begin{rem}
\label{rem:lamdaareunits}
This paragraph is a continuation of the discussion in Remark~\ref{ref:remglobalfunctiononR}. Under the additional hypothesis that $E$ has split multiplicative reduction at $p$ (and a further mild assumption on the ramification of $\overline{T}$), one may prove the existence of a member $g$ of the Hida family for $\overline{T}$ for which the work of Skinner and Urban~\cite{skinnerurbanmainconj} applies. This, along with Theorem C(iii), shows that the \emph{interpolation factors} $\lambda_g$ (which appeared in the statement of Theorem~B) are all units for every form $g$ that belongs to the relevant Hida family; in particular for $g=f_E$ itself. 
This allows one to prove that the divisibility one obtains using the Beilinson-Kato elements in the statement of Kato's main conjecture (c.f., \cite[Theorem 12.5]{ka1}) may be turned into an equality and leads to a proof of \cite[Theorem A]{skinnersplitcyclo}. 
\end{rem}
\subsubsection{Sheaf of universal $p$-adic $L$-function and the Iwasawa theory of eigencurve}
\label{intro:univpadicL}
Suppose that we are still in the setting of Section~\ref{subsub:universalKS}. Greenberg's main conjecture predicts the existence of a several-variable $p$-adic $L$-function attached to a family $\mathcal{T}$, assuming that the representation $\mathcal{T}$ is \emph{Panchishkin-ordinary} at $p$. This $p$-adic $L$-function should in return control a `big' Selmer group associated to $\mathcal{T}$; see Conjecture 4.1 in \cite{greenbergdeform}. The existence of this $p$-adic $L$-function remains highly conjectural. 

The situation in our case with the family $\TT$ is all the way more complicated: In this case $\TT$ is no longer Panchishkin-ordinary and one does not even have a definition of a Selmer group for $\TT$. Likewise, it is not altogether clear what the \emph{three-variable} $p$-adic $L$-function should look like in this picture (or if it should exist in the first place as an element of $R$ or its analytification $R^\dagger=\Gamma(\frak{X},\oo_{\textup{Spf}(R[1/p])})$, where $\frak{X}$ denotes the (Berthelot) generic fiber of $\textup{Spf}(R[1/p])$).

 Despite the rather gloomy look of things (off the nearly-ordinary deformation subspace, that is) we portray above, there has been significant progress in the last few years over the \emph{finite-slope} locus: See \cite{bellaicheL} for the construction of a \emph{$2$-variable} $p$-adic $L$-functions on the Coleman-Mazur eigencurve $\mathcal{C}(\overline{\rho})$ and \cite{pottharst,bellaicheS,kpx,rliu} for the global triangulation on $\mathcal{C}(\overline{\rho})$ that yields to the construction of a coherent Selmer sheaf over $\mathcal{C}(\overline{\rho})$. We now explain how our construction of universal Kolyvagin systems relate to these developments and provide us with a quasicoherent sheaf on the eigencurve which in many ways resemble a \emph{universal $p$-adic $L$-function}. We give a definition of this object in this introduction (Definition~\ref{def:universalpadicLfunction}) and outline its basic properties. We refer the reader to Section~\ref{subsub:generalex} for further details. This part of our article was heavily influenced by our personal communication with J. Pottharst. We thank him heartily for his insightful e-mail and detailed response to our questions. 
  
 Let $\LL_E=\oo_E[[\Gamma]]$ be the cyclotomic Iwasawa algebra, $\frak{I}_E$ be Berthelot's analytic generic fiber of $\textup{Spf}\,\LL$ and $\LL^\dagger_E=\Gamma(\frak{I}_E,\oo_{\textup{Spf}\,\LL_E})$. Note that our $\LL^\dagger_E$  is denoted by $\LL_\infty$ in \cite{pottharstcyclo}. 
 
Let $\lambda \in \oo(\mathcal{C})^\times$ be the $U_p$-eigenvalue function and $\kappa \in \oo(\mathcal{C})$ the weight function. Let $\TT_{\mathcal{C}}$ denote the pullback of the universal deformation, which is locally free coherent sheaf on $\mathcal{C}$ of rank $2$ which equipped with a continuous $\oo(\mathcal{C})$-linear Galois action. 
 Let $\mathcal{D}_{\textup{rig}}^{\dagger}(\TT_{\mathcal{C}})$ denote the sheafification of the $(\varphi,\Gamma)$-module functor on the weak $G$-topology of $\mathcal{C}$ given as in \cite[Definition 1.2.4]{rliu}. The work of Liu and Kedlaya-Pottharst-Xiao equips us with a coherent subsheaf $\mathcal{F}$  of the sheaf $\mathcal{D}_{\textup{rig}}^{\dagger}(\TT_{\mathcal{C}})$ which is locally free of rank one away from \emph{exceptional points}, and saturated away from the \emph{unsaturated points} of $\mathcal{C}$ (see Definition~\ref{define:singularlocus} where these notions are introduced) and restricts for every $x$ belonging to this locus to a triangulation of the $(\varphi,\Gamma)$-module $D_{\textup{rig}}^{\dagger}(V_x)$ associated to $V_x$.

 
 Let $\kappa_1$ denote the initial term of \emph{any} universal Kolyvagin system and $ H^1_{\psi}(*)$ is the Iwasawa cohomology sheaf given as in \cite{kpx} (where $\psi$ is a left inverse of the Frobenius operator $\varphi$ in the context of $(\varphi,\Gamma)$-modules).
 
 \begin{define}
 \label{def:universalpadicLfunction}
The quasicoherent sheaf of \emph{universal $p$-adic $L$-function} is the invertible sheaf
$$\Xi_{\overline{\rho}}:=\textup{im}\left(R\cdot\textup{loc}_p(\kappa_1) \lra H^1_{{\psi}}(\mathcal{D}_{\textup{rig}}^\dagger(\TT_{\mathcal{C}}))/H^1_{\psi}(\mathcal{F})\right)\,.$$
Here the arrow is obtained via Sen's theory (through taking Tate twists) and using \cite[Theorem 1.9]{pottharst} along with the identification of \cite[Corollary 4.4.11]{kpx}.
\end{define} 
The following interpolation property partially justifies why  $\Xi_{\overline{\rho}}$ deserves to be called the sheaf of a ``$p$-adic $L$-function''. See Theorem~\ref{thm:firstquestionhasaffrimativeanswerifconjholdstrue} and Remark~\ref{rem:globalcolemantrivialization} for a strengthening of this interpolation property.

\begin{thmd}[Theorem~\ref{prop:singularkatointerpolation}]
  Let $E$ be a finite extension of $\QQ_p$ and $x=(\psi^\dagger,\lambda(x)) \in  \mathcal{C}_{\textup{cl-fs}}$ be an $E$-valued saturated point. Let $f_{x}$ denote the corresponding (classical) eigenform with a distinguished $U_p$-eigenvalue $\lambda(x)$. For some $\frak{t}_{x} \in \LL_E$\,, the following equality of invertible ideals of $\LL_E^\dagger$ holds true:
  $$\frak{t}_{x}\cdot\al_{D_{x}/F_{x}}\circ\psi^\dagger_\LL\left(\Xi_{\overline{\rho}}\right)=\Gamma_{\kappa(x)-1}\cdot\delta^{-1}_x\al_{p,\lambda(x)^c}(f_x^c)\cdot \LL^\dagger_E\,,$$
  \end{thmd}
Here $\psi^\dagger_{\LL} : R^\dagger \rightarrow \LL_E^\dagger$ the cyclotomic deformation of the point $\psi^\dagger$ (given as in Definition~\ref{def:cyclodeformpsi}), $\al_{D_{x}/F_{x}}$ the Pottharst-Perrin-Riou ``big logarithm'' map introduced in Definition~\ref{def:pottharstscoleman}, $\al_{p,\lambda(x)^c}(f_{\psi^\dagger}^c)$ is the $p$-adic $L$-function attached to the complex conjugate of $f_{x}$ for the $p$-stabilization determined by $\lambda(x)^c$ and finally, the $\Gamma$ and $\delta$-factors are given as in Definition~\ref{define:gammafactors}. 

As we have already pointed out in Remark~\ref{ref:remglobalfunctiononR},  whether the factors $\frak{t}_{x}$ interpolate over affinoid subdomains of $\mathcal{C}$ or not is a question that needs to be explored in order to gain understanding of the variation of arithmetic data on the eigencurve. See Question~\ref{question:txinterpolate1} for precise formulation of this question and Theorem~\ref{thm:firstquestionhasaffrimativeanswerifconjholdstrue} for a result in this direction. 

Theorem~\ref{thm:firstquestionhasaffrimativeanswerifconjholdstrue} together with the forthcoming work of Hansen shows that the image of the sections $\Xi_{\overline{\rho}}(A)$ under what we call the Coleman-trivialization (introduced as part of Conjecture~\ref{conj:bigColeman}) lies in the module generated by Bella\"iche's two-variable $p$-adic $L$-function. Along with Theorem~C (\emph{Greenberg's main conjecture without $p$-adic $L$-function}), this fact is expected to be one of the main ingredients in deducing a divisibility in the main conjecture (\emph{with $p$-adic $L$-function}) in this context.  What is crucial is that the leading term $\kappa_1$ (which is used to construct the sheaf $\Xi_{\overline{\rho}}$ as well as to control the strict Selmer group as in Theorem~C) of the universal Kolyvagin system is integral, so that the Euler-Kolyvagin system machinery applies. The only missing ingredient that prevents us to assemble a proof of the divisibility alluded to above is a well-behaved descent/control mechanism for Selmer complexes in this context. We hope to address this matter in a future work. 

What is significant about that the modules  $\Xi_{\overline{\rho}}(A)$ is that (which, as we indicated above, are related in a precise manner to Bella\"iche's two-variable $p$-adic $L$-function by Hansen's work) they readily patch together over $\mathcal{C}$, and $\Xi_{\overline{\rho}}$ indeed behaves like a sheaf of $p$-adic $L$-functions on the eigencurve.

  \subsubsection*{\bf Acknowledgements.}
The author wishes to thank Jan Nekov\'a\v{r} and Tadashi Ochiai for very useful conversations on various technical points; to Shu Sasaki for his comments on an earlier version of this article and for letting the author know of Kisin's recent (unpublished) $\textup{big } R=\textup{big } T$ theorem which played a motivating role to prepare this article as well. He is indebted to Barry Mazur for encouraging the author to extend his thesis work to this generality. Much of the work in this article was carried out during the two summers when the author was hosted by Pontificia Universidad Cat\'olica de Chile, for whose hospitality he is grateful. He acknowledges partial support through a EC-IRG and a T\"UB\.ITAK grant.

\subsection{Notations}
\label{subsec:notationhypo} For any field $K$, fix a separable closure $\bar{K}$ of $K$ and set $G_K=\textup{Gal}(\bar{K}/K)$. Let $F$ be a number field and $\lambda$ be a non-archimedean place of $F$. Fix a decomposition subgroup $G_\lambda<G_F$ and let $I_\lambda<G_\lambda$ denote the inertia subgroup. Often we will identify $G_\lambda$ by $G_{F_\lambda}$. For a finite set $\Sigma$ of places of $K$, define $K_\Sigma$ to be the maximal extension of $K$ unramified outside $\Sigma$.

Let $p$ be an odd prime and let $\QQ_\infty/\QQ$ be the cyclotomic $\ZZ_p$-extension. Let $\pmb{\mu}_{p^n}$ denote the $p^n$-th roots of unity and set $\pmb{\mu}_{p^\infty}=\varinjlim \pmb{\mu}_{p^n}$. Set $\Gamma=\textup{Gal}(\QQ_\infty/\QQ)$ and fix a topological generator $\gamma$ of  $\Gamma$.  Let $\LL=\ZZ_p[[\Gamma]]$ be the cyclotomic Iwasawa algebra.

For a ring $S$, an $S$-module $M$ and an ideal $I$ of $S$,  let $M[I]$ denote the submodule of $M$ consisting of elements that are killed by $I$.

For the ring $R=\oo[[X_1,X_2,X_3]]$, we set $R_{u,v,w}:=R/(X_1^u,X_2^v,X_3^w)$ and $R_{r,u,v,w}:=R/(\varpi^r,X_1^u,X_2^v,X_3^w)$. We define the quotient modules $\TT_{u,v,w}:=\TT\otimes_{R}R_{u,v,w}$ and $\TT_{r,u,v,w}:=\TT\otimes_{R}R_{r,u,v,w}$.

Similarly for the ring $\frak{R}=\RR[[\Gamma]]$ as above, define the rings  $\frak{R}_{u,v}=\frak{R}/(X^u,(\gamma-1)^v)$ and  $\frak{R}_{r,u,v}=\frak{R}/(\varpi^r,X^u,(\gamma-1)^v)$. Define also the quotient modules $\frak{T}_{u,v}:=\frak{T}\otimes_{\frak{R}}\frak{R}_{u,v}$ and $\frak{T}_{r,u,v}:=\frak{T}\otimes_{\frak{R}}\frak{R}_{r,u,v}$.

Finally, we define the $p$-divisible goups $\mathcal{A}_{u,v,w}:=\TT_{u,v,w}\otimes\Phi/\oo$ and $\frak{A}_{u,v}:=\frak{T}_{u,v}\otimes\Phi/\oo$.

Let $\RR_0=\RR/(\varpi,X)$ be the dimension-zero Gorenstein artinian ring, where $\RR$ is as above. As explained in~\cite[Proposition 1.4]{flttil},
\be\label{eqn:gorzeromodm}\RR_0[\mm_{\RR}] \hbox{ is a one-dimensional } \texttt{k}\hbox{-vector space}\ee
 where $\mm_{\RR}$ denotes the maximal ideal of $\RR$ and $\texttt{k}=\RR/\mm_{\RR}$. Define also $\RR_1=\RR/X$. Using the fact $\{\varpi,X\}$ is a regular sequence in $\RR$, we see that $\RR_1$ is a dimension-1 Gorenstein domain. Set $\tilde{\Phi}=\textup{Frac}({\RR_1})$. As $\RR_1$ is finitely generated and free  as an $\oo$-module, it follows that $\tilde{\Phi}$ is a finite extension of $\Phi$. Let $\frak{O}$ be the integral closure of $\RR_1$ in $\tilde{\Phi}$. Then $\frak{O}$ is a discrete valuation ring and $\frak{O}/\RR_1$ has finite cardinality. Let $\mm_\frak{O}$ be the maximal ideal of $\frak{O}$ and $\pi_\frak{O}$ be a uniformizer of $\frak{O}$. Define $T_{\frak{O}}:=\frak{T}_{1,1}\otimes_{\RR_1}\frak{O}$ (\emph{deformation of \,$\overline{T}$ to $\frak{O}$}) and $A=T_{\frak{O}}\otimes\QQ_p/\ZZ_p$.  As $\frak{O}/\RR_1$ is of finite order, it follows that $A\cong\frak{A}_{1,1}$.

\section{Local Conditions and Selmer groups}
\label{sec:selmerstr}
We recall a definition from~\cite[\S2]{mr02}. Let $M$ be any $\oo[[G_{\QQ}]]$-module.
\begin{define}
\label{selmer structure}
 A \emph{Selmer structure} $\FF$ on $M$ is a collection of the following data:
\begin{itemize}
\item A finite set $\Sigma(\FF)$ of places of $\QQ$, including $\infty$, $p$, and all primes where $M$ is ramified.
\item For every $\ell \in \Sigma(\FF)$, a local condition on $M$ (which we now view as a $\oo[[G_{\ell}]]$-module), i.e., a choice of an $\oo$-submodule $H^1_{\FF}(\QQ_{\ell},M) \subset H^1(\QQ_{\ell},M).$
 \end{itemize}
\end{define}
\begin{define}
\label{selmer_group}
For a Selmer structure $\FF$ on $M$, define the Selmer group $H_{\FF}^1(\QQ,M)$ to be
$$H_{\FF}^1(\QQ,M)=\ker\left(H^1(\QQ_{\Sigma(\FF)}/\QQ,M)\lra \prod_{\ell\in \Sigma(\FF)}\frac{H^1(\QQ_\ell,M)}{H_{\FF}^1(\QQ_\ell,M)}\right).$$
\end{define}
\begin{define}\label{selmer triple}
A \emph{Selmer triple} is a triple $(M,\FF,\PP)$ where $\FF$ is a Selmer structure on $M$ and $\PP$ is a set of rational primes, disjoint from $\Sigma(\FF)$.
\end{define}

\begin{define}
\label{def:dualselmerstructure}
Let $\FF$ be a Selmer structure on $M$. For each prime $\ell \in \Sigma(\FF)$, define $H^1_{\FF^*}(\QQ_\ell,M^*):=H^1_\FF(\QQ_\ell,M)^\perp$ as the orthogonal complement of $H^1_\FF(\QQ_\ell,M)$ under the local Tate pairing. The Selmer structure $\FF^*$ on $M^*$ defined in this manner is called the \emph{dual Selmer structure}.
\end{define}

Define the Selmer structure $\FFc$ (the \emph{canonical Selmer structure}) on $\TT_{u,v,w}$ as follows:
\begin{itemize}
\item  $\Sigma(\FFc)=\Sigma:=\{\ell: \TT \hbox{ is ramified at } \ell\} \cup \{p,\infty\}$.

\item $H ^{1} _{\FFc}(\QQ _{\ell}, \TT_{u,v,w}):= \left \{
\begin{array}{ccl}
	 H ^{1} (\QQ _{p} ,\TT_{u,v,w})& ,& \hbox{if } \ell = p,    \\
          H ^{1}_{f}(\QQ _{\ell},\TT_{u,v,w})& ,& \hbox{if } \ell \in \Sigma(\FFc)-\{p,\infty\}.
\end{array}
\right.$
\end{itemize}
Here
$$H ^{1}_{f} (\QQ _{\ell} , \TT_{u,v,w}):= \ker\left\{ H^{1} (\QQ _{\ell} , \TT_{u,v,w}) \lra H^{1}  (I_\ell , \TT_{u,v,w} \otimes_{\oo} \Phi)\right \}.$$

We denote the Selmer structure on the quotients $\TT _{r,u,v,w}$ obtained by \emph{propagating} $\FFc$ on $\TT_{u,v,w}$ to $\TT _{r,u,v,w}$ also by $\FFc$. See~\cite[Example 1.1.2]{mr02} for a definition of the \emph{propagation} of local conditions.

We define the Selmer structure $\FFc$ on $\frak{T}_{u,v}$ (and its propagation to its quotients $\frak{T} _{r,u,v}$) in a similar way:
\begin{itemize}
\item  $\Sigma(\FFc)=\{\ell: \frak{T} \hbox{ is ramified at } \ell\} \cup \{p,\infty\}$.

\item $H ^{1} _{\FFc}(\QQ _{\ell}, \frak{T}_{u,v}):= \left \{
\begin{array}{ccl}
	 H ^{1} (\QQ _{p} ,\frak{T}_{u,v})& ,& \hbox{if } \ell=p,    \\
          H ^{1}_{f}(\QQ _{\ell},\frak{T}_{u,v})& ,& \hbox{if } \ell \in \Sigma(\FFc)-\{p,\infty\}
\end{array}
\right.,$\\
\end{itemize}
We also define a Selmer structure $\FFc$ on $\frak{T}/\mm$ as follows:
\begin{itemize}
\item Set $H^1_{\FFc}(\QQ_p,\frak{T}/\mm)=H^1(\QQ_p,\frak{T}/\mm)$.
\item Propagate the local conditions at $\ell\neq p$ given by $\FFc$ on $\frak{T}_{1,1}$ to $\frak{T}/\mm$.
\end{itemize}
Note in particular that the Selmer structure $\FFc$ on $\frak{T}/\mm$ will not always be the propagation of the canonical Selmer structure on $\frak{T}_{1,1}$.
\subsection{Local conditions at $\ell \neq p$}
In this section we compare various alterations of the local conditions at $\ell \neq p$.
Define
$$H^1_{\textup{ur}}(\QQ_\ell,M):=\ker\left(H^1(\QQ_\ell,M)\lra H^1(I_\ell,M)\right)$$
for any $M$ on which $G_{\QQ_\ell}$ acts.
Using the exact sequence
$$0\lra \TT_{u,v,w} \lra \TT_{u,v,w}\otimes_{\oo}\Phi \lra \mathcal{A}_{u,v,w}\lra0$$
define also
$$H^1_{f}(\QQ_\ell,\mathcal{A}_{u,v,w})=\textup{im}\left(H^1_\textup{ur}(\QQ_\ell, \TT_{u,v,w}\otimes_{\oo}\Phi)\lra  \mathcal{A}_{u,v,w}\right).$$
Define finally
$$H^1_{f}(\QQ_\ell,\TT_{r,u,v,w})=\ker\left(H^1(\QQ_\ell,\TT_{r,u,v,w}) \lra  \frac{H^1(\QQ_\ell,\mathcal{A}_{u,v,w})}{H^1_{f}(\QQ_\ell,\mathcal{A}_{u,v,w})}\right),$$
where the map is induced from the injection $\TT_{r,u,v,w}\hookrightarrow \mathcal{A}_{u,v,w}$. Lemma 1.3.8(i) of \cite{r00} shows that
\be\label{eqn:caneqfin} H^1_{\FFc}(\QQ_\ell,\TT_{r,u,v,w})=H^1_{f}(\QQ_\ell,\TT_{r,u,v,w}).\ee

Similarly one defines $H^1_f(\QQ_\ell,\frak{A}_{u,v})$ and $H^1_f(\QQ_\ell,\frak{T}_{r,u,v})$, and verifies using \cite[Lemma 1.3.8(i)]{r00} that
\be\label{eqn:caneqfinfrak}
H^1_{\FFc}(\QQ_\ell,\frak{T}_{r,u,v})=H^1_f(\QQ_\ell,\frak{T}_{r,u,v}).
\ee
\subsection{Local conditions at $p$}
\begin{prop}
\label{prop:loccondatp}
Assuming $($\textbf{\textup{H.nA}}$)$,
\begin{itemize}
\item[(i)] $H^1_{\FFc}(\QQ_p,\TT_{r,u,v,w})=H^1(\QQ_p,\TT_{r,u,v,w})$,
\item[(ii)] $H^1_{\FFc}(\QQ_p,\frak{T}_{r,u,v})=H^1(\QQ_p,\frak{T}_{r,u,v})$.
\end{itemize}
\end{prop}
\begin{proof}
We need to check that
$$\textup{coker}\left(H^1(\QQ_p,\TT_{u,v,w})\lra H^1(\QQ_p,\TT_{r,u,v,w})\right)=0.$$
Note that
$$\textup{coker}\left(H^1(\QQ_p,\TT_{u,v,w})\lra H^1(\QQ_p,\TT_{r,u,v,w})\right)=H^2(\QQ_p,\TT_{u,v,w})[\varpi^r],$$
so it suffices to check that $H^2(\QQ_p,\TT_{u,v,w})=0$.

Now by local duality, $(\textbf{\textup{H.nA}})$ implies that $H^2(\QQ_p,\overline{T})=0$. Using the fact that the cohomological dimension of $G_{p}$ is two, it follows that
$$H^2(\QQ_p,\TT)/\mathcal{M}\cdot H^2(\QQ_p,\TT)=0,$$
where $\mathcal{M}=(\varpi,X_1,X_2,X_3)$ is the maximal ideal of the ring $R$. By Nakayama's Lemma, we therefore see that $H^2(\QQ_p,\TT)=0$. Using again the fact that the cohomological dimension of $G_{p}$ is two, we conclude that
$$0=H^2(\QQ_p,\TT)\twoheadrightarrow H^2(\QQ_p,\TT_{u,v,w})$$
and the proof of (i) follows.

The proof of (ii) is similar but more delicate as the ring $\frak{R}$ is not necessarily regular. As above, we first check that
\be\label{eqn:vanishfrak1}H^2(\QQ_p,\frak{T})=0.\ee
Considering the $G_p$-cohomology induced from the exact sequence
$$0\lra \frak{T}\stackrel{\gamma-1}{\lra}\frak{T}\lra\mathcal{T}\lra0$$
and using Nakayama's lemma, (\ref{eqn:vanishfrak1}) is reduced to verifying that $H^2(\QQ_p,\mathcal{T})=0$. Similarly, using the exact sequences
$$0\lra\mathcal{T}\stackrel{X}{\lra}\mathcal{T}\lra\mathcal{T}/X\lra0$$
$$0\lra\mathcal{T}/X\stackrel{\varpi}{\lra}\mathcal{T}/X\lra\frak{T}_{1,1}\lra0$$
in turn, we reduce to checking that
\be\label{eqn:vanishfrak2} H^2(\QQ_p, \frak{T}_{1,1,1})=0\ee
The assertion (\ref{eqn:vanishfrak2}) is proved below. We first show that (ii) follows from (\ref{eqn:vanishfrak2}).

As above,
$$\textup{coker}\left(H^1(\QQ_p,\frak{T}_{u,v})\lra H^1(\QQ_p,\frak{T}_{r,u,v})\right)=H^2(\QQ_p,\frak{T}_{u,v})[\varpi^r].$$
By (\ref{eqn:vanishfrak2}) and the fact that the cohomological dimension of $G_p$ is two, it follows that
$$0=H^2(\QQ_p,\frak{T})\twoheadrightarrow H^2(\QQ_p,\frak{T}_{u,v}).$$
This proves that  $H^2(\QQ_p,\frak{T}_{u,v})=0$ and it follows that
$$H^1_{\FFc}(\QQ_p,\frak{T}_{r,u,v})=\textup{im}\left(H^1(\QQ_p,\frak{T}_{u,v}\lra H^1(\QQ_p,\frak{T}_{r,u,v})\right)= H^1(\QQ_p,\frak{T}_{r,u,v}),$$
as desired. 
\end{proof}
\begin{claim}
\label{claim:vanishfrak1}
Assuming $(${\upshape{\textbf {H.nA}}}$)$, we have $H^2(\QQ_p, \frak{T}_{1,1,1})=0$.
\end{claim}
\begin{proof}
The property (\ref{eqn:gorzeromodm}) shows that $\frak{T}_{1,1,1}[\mm]\cong \overline{T}$, hence that
$\frak{T}_{1,1,1}^*/\mm\cong \overline{T}^*.$ 
Since we assumed \textbf{H.nA}, it thus follows that
$$H^0(\QQ_p,\frak{T}_{1,1,1}^*/\mm)=0.$$
The module $\frak{T}_{1,1,1}^*$ is free of of finite rank over the Gorenstein artinian ring $\RR_0$, hence by \cite[Lemma 2.1.4]{mr02} we conclude that $H^0(\QQ_p,\frak{T}_{1,1,1}^*)=0$ as well. Claim now follows by local duality.
\end{proof}
\subsection{Kolyvagin primes and transverse conditions}
Let $\tau \in G_\QQ$ be as in the statement of  the hypothesis (\textbf{H.2}).
\begin{define}
\label{def:kolyprimes}For $\bar{\frak{n}}=(r,u,v,w) \in (\ZZ_{>0})^4$, define
\begin{itemize}
\item[(i)]  $H_{\bar{\frak{n}}}=\ker\left(G_\QQ\ra \textup{Aut}(\TT_{r,u,v,w})\oplus\textup{Aut}(\pmb{\mu_{p^{r}}})\right)$,
\item[(ii)]  $L_{\bar{\frak{n}}}=\bar{\QQ}^{H_{\bar{\frak{n}}}}$,
\item[(iii)] $\mathcal{P}_{\bar{\frak{n}}}=\{\hbox{primes }\ell: \textup{Fr}_\ell \hbox{ is conjugate to } \tau \hbox{ in } \textup{Gal}(L_{\bar{\frak{n}}}/\QQ)\}$.
\end{itemize}
The collection $\PP_{\bar{\frak{n}}}$ is called the collection of \emph{Kolyvagin primes} for $\TT_{r,u,v,w}$. Set $\PP=\PP_{(1,1,1,1)}$ and define $\NN_{\bar{\frak{n}}}$ to be the set of square free products of primes in $\PP_{\bar{\frak{n}}}$.
\end{define}
We similarly define for $\bar{\frak{s}}=(r,u,v)$ the sollection of Kolyvagin primes $\PP_{\bar{\frak{s}}}$ for $\frak{T}_{r,u,v}$ and the set $\NN_{\bar{\frak{s}}}$ of square free products of primes in $\PP_{\bar{\frak{s}}}$.
\begin{define}
\label{def:partialorder}
The partial order $\prec$ on the collection of quadruples $(r,u,v,w) \in (\ZZ_{>0})^4$ is defined by setting
$$\bar{\frak{n}}=(r,u,v,w) \prec(r^ \prime,u^ \prime,v^ \prime,w^ \prime)= \bar{\frak{n}}^\prime$$
if $r\leq r^\prime$, $u\leq u^\prime$, $v\leq v^\prime$ and $w\leq w^\prime$.

We denote the partial order defined on triples of positive integers in an identical manner also by $\prec$.
\end{define}
 To ease notation, set  $\TT_{\bar{\frak{n}}}:=\TT_{r,u,v,w}$ and $R_{\bar{\frak{n}}}:=R_{r,u,v,w}$  for $\bar{\frak{n}}=(r,u,v,w)$. Define similarly $\frak{T}_{\bar{\frak{n}}}:=\frak{T}_{r,u,v}$ and $\frak{R}_{\bar{\frak{n}}}:=\frak{R}_{r,u,v}$.
\begin{rem}
\label{rem:ellcong1modp}
Suppose $\ell$ is a Kolyvagin prime in $\PP_{\bar{\frak{n}}}$ (resp., in $\PP_{\bar{\frak{s}}}$), where $\bar{\frak{n}}$ (resp., $\bar{\frak{s}}$)  are as above. Then as $\tau$ acts trivially on $\pmb{\mu}_{p^r}$ and $\textup{Fr}_\ell$ is conjugate to $\tau$ in $ \textup{Gal}(L_{\bar{\frak{n}}}/\QQ)$, it follows that $\textup{Fr}_\ell$ acts trivially on $\pmb{\mu}_{p^r}$ and hence that $\ell\equiv 1 \mod p^r.$ In particular
$$|\mathbb{F}_{\ell}^\times|\cdot \TT_{r,u,v,w}=(\ell-1)\TT_{r,u,v,w}=0$$
$$(\textup{resp., } |\mathbb{F}_{\ell}^\times|\cdot \frak{T}_{r,u,v}=0).$$
\end{rem}
Throughout this section, fix a Kolyvagin prime $\ell \in \PP_{\bar{\frak{n}}}$ (or in $\PP_{\bar{\frak{s}}}$, whenever we talk about quotients of $\frak{T}$).

\begin{define}
\label{def:transverse}Let $T$ be one of $\TT_{r,u,v,w}$, $\frak{T}_{r,u,v}$ or $\frak{T}/\mm$.
\begin{itemize}
\item[(i)]The submodule of $H^1(\QQ_\ell,T)$ given by
$$H^1_{\textup{tr}}(\QQ_\ell,T)=\ker\left(H^1(\QQ_\ell,T)\lra H^1(\QQ_\ell(\pmb{\mu_\ell}),T)\right)$$
is called the \emph{transverse submodule}.
\item[(ii)] The \emph{singular quotient} $H^1_s(\QQ_\ell,T)$ is defined by the exactness of the sequence
\be\label{eqn:singseq} 0\lra H^1_f(\QQ_\ell,T)\lra H^1(\QQ_\ell,T)\lra H^1_s(\QQ_\ell,T)\lra 0 \ee
\end{itemize}
\end{define}

\begin{define}
\label{def:modifiedSelmer}
Let $T$ be one of $\TT_{r,u,v,w}$, $\frak{T}_{r,u,v}$ or $\frak{T}/\mm$ and suppose $n\in \NN_{\bar{\frak{n}}}$ (or $n\in \NN_{\bar{\frak{s}}}$ if we are talking about quotients of $\frak{T}$). The \emph{modified Selmer structure} $\FFc(n)$ on $T$ is defined with the following data:
\begin{itemize}
\item $\Sigma(\FFc(n))=\Sigma(\FFc)\cup\{\hbox{primes }\ell: \ell \mid  n\}$.
\item If $\ell\nmid n$ then $H^1_{\FFc(n)}(\QQ_\ell,T)=H^1_{\FFc}(\QQ_\ell,T)$.
\item If $\ell \mid n$ then $H^1_{\FFc(n)}(\QQ_\ell,T)=H^1_{\textup{tr}}(\QQ_\ell,T)$.
\end{itemize}
\end{define}
\begin{rem}
Proposition 1.3.2 of \cite{mr02} shows that $\FFc(n)^*=\FFc^*(n)$.
\end{rem}
\begin{lemma}
\label{lem:transverseproperties}
Let $T$ be one of the rings $\TT_{r,u,v,w}$, $\frak{T}_{r,u,v}$ or $\frak{T}/\mm$. Then the \emph{transverse subgroup} $H^1_{\textup{tr}}(\QQ_\ell,T) \subset H^1(\QQ_\ell,T)$ projects isomorphically onto $H^1_s(\QQ_\ell,T)$. In other words $($\ref{eqn:singseq}$)$ above has a functorial splitting.
\end{lemma}

\begin{proof}
This is \cite[Lemma 1.2.4]{mr02} which is proved for a general artinian coefficient ring.
\end{proof}
\begin{prop} Let $\bar{\frak{n}}=(r,u,v,w)$ and $\bar{\frak{s}}=(r,u,v)$ be as above.
\label{prop:transversefurtherprop} $\,$
\begin{itemize}
\item[(i)] There are canonical functorial isomorphisms $$H^1_f(\QQ_\ell,\TT_{\bar{\frak{n}}})\cong \TT_{\bar{\frak{n}}}\big{/}(\textup{Fr}_\ell-1)\TT_{\bar{\frak{n}}},$$
 $$H^1_s(\QQ_\ell,\TT_{\bar{\frak{n}}})\cong \left(\TT_{\bar{\frak{n}}}\right)^{\textup{Fr}_\ell=1}\,\,.$$
\item[(ii)] There is a canonical isomorphism (called the finite-singular comparison isomorphism)
$$\phi_\ell^{fs}: H^1_f(\QQ_\ell,\TT_{\bar{\frak{n}}})\lra H^1_s(\QQ_\ell,\TT_{\bar{\frak{n}}})\otimes \mathbb{F}_\ell^\times\,\,.$$
\item[(iii)] The $R_{\bar{\frak{n}}}$ modules $H^1_f(\QQ_\ell,\TT_{\bar{\frak{n}}})$, $H^1_s(\QQ_\ell,\TT_{\bar{\frak{n}}})$ and $H^1_{\textup{tr}}(\QQ_\ell,\TT_{\bar{\frak{n}}})$ are free of rank one.
\end{itemize}
The analogous statements hold true when $\TT_{\bar{\frak{n}}}$ is replaced by $\frak{T}_{\bar{\frak{s}}}$ or $\frak{T}/\mm$ \textup{(}and the ring $R_{\bar{\frak{n}}}$ by $\frak{R}_{\bar{\frak{s}}}$ or $\texttt{k}=\frak{R}/\mm$\textup{)}.
\end{prop}
\begin{proof}
(i) is \cite[Lemma 1.2.1]{mr02}. The finite-singular comparison isomorphism is defined in \cite[Definition 1.2.2]{mr02} and (ii) is \cite[Lemma 1.2.3]{mr02}. (iii) follows from (i), (ii) and Lemma~\ref{lem:transverseproperties}.

Note that all the results quoted from \cite{mr02} apply in our setting thanks to Remark~\ref{rem:ellcong1modp}.
\end{proof}
\section{Universal Kolyvagin systems}
\label{subsec:kolsysforbig}
Assume in this section that $\chi(\TT)=\chi(\frak{T})=1$ as well as the truth of the hypotheses  (\textbf{H1}) - (\textbf{H4}), (\textbf{H.Tam}), and (\textbf{H.nA}).

\subsection{Kolyvagin systems over Artinian rings}
Throughout this section fix $\bar{\frak{n}}=(r,u,v,w)$ and $\bar{\frak{i}}\in (\ZZ_{>0})^4$ (resp., $\bar{\frak{s}}=(r,u,v)$ and $\bar{\frak{j}}\in (\ZZ_{>0})^3$) such that $\bar{\frak{n}}\prec \bar{\frak{i}}$ (resp., $\bar{\frak{s}}\prec \bar{\frak{j}}$). Let $T$ be one of $\TT_{\bar{\frak{n}}}$, $\frak{T}_{\bar{\frak{s}}}$ or $\frak{T}/\mm$, and let $S$ be the corresponding quotient ring $R_{\bar{\frak{n}}}$, $\frak{R}_{\bar{\frak{s}}}$ or $\texttt{k}$. Let $\PP$ denote the collection of Kolyvagin primes $\PP_{\bar{\frak{i}}}$ (resp., $\PP_{\bar{\frak{j}}}$) and $\NN$ denote the set of square free products of primes in $\PP$. 

Many of the definitions and arguments in this section follow closely \cite{mr02} and \cite{kbb}.
\begin{define}
\label{simplicial sheaf}$\,$
\begin{enumerate}
\item[(i)] If $X$ is a graph and $\textup{Mod}_{R}$ is the category of $R$-modules, a \emph{simplicial sheaf} $\mathcal{S}$ on $X$ with values in $\textup{Mod}_R$ is a rule assigning
\begin{itemize}
\item an $R$-module $\mathcal{S}(v)$ for every vertex $v$ of $X$,
\item an $R$-module $\mathcal{S}(e)$ for every edge $e$ of $X$,
\item an $R$-module homomorphism $\psi_{v}^e:\,\mathcal{S}(v)\ra\mathcal{S}(e)$ whenever the vertex $v$ is an endpoint of the edge $e$.
\end{itemize}
\item[(ii)] A \emph{global section} of $\mathcal{S}$ is a collection $\{\kappa_v \in \mathcal{S}(v): v \hbox{ is a vertex of } X\}$ such that, for every edge $e=\{v,v^{\prime}\}$ of $X$, we have  $\psi_{v}^e(\kappa_v)=\psi_{v^{\prime}}^e(\kappa_{v^{\prime}})$ in $\mathcal{S}(e)$. We write $\Gamma(\mathcal{S})$ for the $R$-module of global sections of $\mathcal{S}$.
\end{enumerate}
\end{define}

\begin{define} (Mazur-Rubin) 
\label{selmer sheaf}
For the Selmer triple $(T,\FFc,\PP)$, we define a graph $\XX=\XX(\PP)$ by taking the set of vertices of $\XX$ to be $\NN$, and the edges to be $\{n,n\ell\}$ whenever $n,n\ell \in \NN$ (with $\ell$ prime).
\begin{enumerate}
\item[(i)] The \emph{Selmer sheaf} $\mathcal{H}$ is the simplicial sheaf on $\XX$ given as follows. Set ${\displaystyle G_n:=\otimes_{\ell | n}\, \mathbb{F}_{\ell}^\times}$. We take
 \begin{itemize}
\item $\mathcal{H}(n) := H^1_{\FFc(n)}(\QQ,T) \otimes G_n$ for $n\in \NN$,
\item  if $e$ is the edge $\{n,n\ell\}$ then $\mathcal{H}(e):= H^1_{s}(\QQ_{\ell},T)\otimes G_{n\ell}$.
\end{itemize}
 We define the vertex-to-edge maps to be
\begin{itemize}
\item $\psi_{n\ell}^{e}:\,H^1_{\FFc(n\ell)}(\QQ,T) \otimes G_{n\ell} \ra H^1_{s}(\QQ_{\ell},T)\otimes G_{n\ell}$ is localization followed by the projection to the singular cohomology $H^1_s(\QQ_\ell,T)$.
\item $\psi_{n}^{e}:\,H^1_{\FFc(n)}(\QQ,T) \otimes G_{n} \ra H^1_{s}(\QQ_{\ell},T)\otimes G_{n\ell}$ is the composition of localization at $\ell$ with the finite-singular comparison map $\phi^{fs}_{\ell}$.
\end{itemize}
\item[(ii)]  A \emph{Kolyvagin system} for the triple $(T,\FFc, \PP)$ is simply a global section of the Selmer sheaf $\mathcal{H}$.
\end{enumerate}
 \end{define}
Let $\Gamma(\mathcal{H})$ denote the $S$-module of global sections of $\mathcal{H}$. 

\begin{define}
\label{def:KSglobalsections} We set $\textbf{KS}(T,\FFc,\PP):=\Gamma(\mathcal{H})$ and call it the \emph{Kolyvagin systems for the Selmer structure $\FFc$ on $T$}. More explicitly, an element $\pmb{\kappa}\in\textbf{KS}(T,\FFc,\PP)$ is a collection $\{\kappa_n\}$ of cohomology classes indexed by $n\in\NN$ such that for every $n,n\ell \in \NN$ we have:
\begin{itemize}
\item  $\kappa_n\in H^1_{\FFc(n)}(\QQ,T)\otimes G_n$,
\item $\phi_{\ell}^{fs}\left(\textup{loc}_\ell(\kappa_n)\right)=\textup{loc}_\ell^s(\kappa_{n\ell})$.
\end{itemize}
Here, $\textup{loc}_\ell^s$ stands for the composite
$$H^1(\QQ,T)\stackrel{\textup{loc}_\ell}{\lra}H^1(\QQ_\ell,T)\lra H^1_s(\QQ_\ell,T).$$
\end{define}
\subsection{Kolyvagin systems over \emph{big} rings}
We are now ready to state the main technical results of this article. We shall give a proof of these assertions in Sections \ref{sec:coreverticesandKS}-\ref{sec:coreverticesexist} and exhibit various arithmetic consequences of them in Section~\ref{sec:applications}.
\begin{define}
\label{def:KSbigrings}
The $R$-module
$$\overline{\mathbf{KS}}(\TT,\FFc,\PP):=\varprojlim_{\bar{\frak{n}}} \left(\varinjlim_{\bar{\frak{i}}} \mathbf{KS}(\TT_{\bar{\frak{n}}},\FFc, \PP_{\bar{\frak{i}}})\right)$$
is called the \emph{module of universal Kolyvagin systems}. Likewise, we define the $\frak{R}$-module $\overline{\mathbf{KS}}(\frak{T},\FFc,\PP):=\varprojlim_{\bar{\frak{s}}} \varinjlim_{\bar{\frak{j}}} \mathbf{KS}(\frak{T}_{\bar{\frak{s}}},\FFc, \PP_{\bar{\frak{j}}})$.
\end{define}

\begin{thm}
\label{thm:KSmain}
Under the running hypotheses the following hold.
\begin{itemize}
\item[(i)] The $R$-module $\overline{\mathbf{KS}}(\TT,\FFc,\PP)$ is free of rank one, generated by a Kolyvagin system $\pmb{\kappa}$ whose image $\bar{\pmb{\kappa}} \in {\mathbf{KS}}(\overline{T},\FFc,\PP)$ is non-zero.\\
\item[(ii)]The $\frak{R}$-module $\overline{\mathbf{KS}}(\frak{T},\FFc,\PP)$ is free of rank one. When the ring $\RR$ is regular, the module $\overline{\mathbf{KS}}(\frak{T},\FFc,\PP)$ is generated by $\pmb{\kappa}$ whose image $\bar{\pmb{\kappa}} \in {\mathbf{KS}}(\overline{T},\FFc,\PP)$ is non-zero.
\end{itemize}
\end{thm}

The proof of this theorem will be given in Section~\ref{sec:coreverticesandKS}.

\begin{define}
\label{def:leadingtermofKS}
Given a Kolyvagin system 
$$\pmb{\kappa}=\{\kappa_m(\bar{\frak{n}})\}_{\bar{\frak{n}}, m\in \NN_{\bar{\frak{n}}}}\in \overline{\mathbf{KS}}(\TT,\FFc,\PP)=\varprojlim_{\bar{\frak{n}}}\mathbf{KS}(\TT_{\bar{\frak{n}}},\FFc, \PP_{\bar{\frak{n}}})$$ 
(where the last equality is by Lemma~\ref{lem:j}),  we define the leading term of $\pmb{\kappa}$ to be the element
$$\kappa_1=\{\kappa_1(\bar{\frak{n}})\}_{\bar{\frak{n}}} \in \varprojlim_{\bar{\frak{n}}} H^1(\QQ_{\Sigma}/\QQ,\TT_{\bar{\frak{n}}})=H^1(\QQ_{\Sigma}/\QQ,\TT).$$
We also define the leading term of a Kolyvagin system for $\frak{T}$ in a similar manner.
\end{define}

\section{Applications of universal Kolyvagin systems}
 \label{sec:applications}
 
  \subsection{Weak Leopoldt Conjecture and Greenberg's main conjecture} 
 \label{subsec:weakleo} The notation and the hypotheses of Section \ref{subsec:kolsysforbig} are in effect throughout. 
We also assume until the end of this section that the ring $\RR$ is regular (therefore, so is the ring $\frak{R}$).

\begin{rem}
 \label{rem:regular}
  When $\RR$ is Hida's universal ordinary Hecke algebra (which we have a closer look at in Section \ref{subsec:example2} below), the ring $\mathcal{R}$ is often regular as explained in \cite[Lemma 2.7]{fouochiai}.
 \end{rem} 
 
  Our goal in this section is to prove Theorems~\ref{thm:weakleo1}, \ref{thm:weakleo2} and \ref{thm:sharpening} below. 

 \begin{thm}
 \label{thm:weakleo1}
 If the ${R}$-module $H^1_{\FFc^*}(\QQ,{\TT}^*)^{\vee}$ is torsion, then there is a Kolyvagin system $\pmb{\kappa} \in \overline{\mathbf{KS}}({\TT},\FFc,\PP)$ whose leading term $\kappa_1$ is non-zero. 
 
 The identical statement holds true when ${R}$ is replaced by $\frak{R}$ and ${\TT}$ by $\frak{T}$.
 \end{thm}
 \begin{thm}
 \label{thm:weakleo2}
 Suppose that there is a Kolyvagin system $\pmb{\kappa} \in \overline{\mathbf{KS}}({\TT},\FFc,\PP)$ whose leading term $\kappa_1$ is non-vanishing. 
 Then the ${R}$-module $H^1_{\FFc^*}(\QQ,{\TT}^*)^{\vee}$ is $R$-torsion.
 
 Similar statement holds true when ${R}$ is replaced by $\frak{R}$ and ${\TT}$ by $\frak{T}$.
  \end{thm}
   
  Suppose $S$ is a regular ring and $M$ is a finitely generated torsion $S$-module. We define
 $$\textup{char}_S(M):=\prod_{\frak{p}}\frak{p}^{\textup{length}_{R_{\frak{p}}}(M_{\frak{p}})}\,,$$
 \emph{the characteristic ideal} of $M$, where the product is over all height 1 primes of $S$. Observe that $\textup{char}_S(M)$ is a principal ideal. In case $M$ is not torsion, we set $\textup{char}_S(M)=0$. When the ring $S$ is understood, we simply write $\textup{char}(M)$ in place of $\textup{char}_S(M)$.
    
    \begin{thm}
\label{thm:sharpening} For $\kappa_1 \in H^1_{\FFc}(\QQ,\TT)$ as above, we have 
$$\textup{char}\left(H^1_{\FFc^*}(\QQ,\TT^*)^{\vee}\right) = \textup{char}\left(H^1_{\FFc}(\QQ,\TT)/R\cdot\kappa_1\right).$$
\end{thm}

  \begin{rem}
  \label{rem:nonzeroimpliesnontorsionimpliesnotinp}
    In this remark, we record the following simple observation: Assuming (\textbf{H3}), the $R$-module $H^1(\QQ_{\Sigma}/\QQ,\TT)$ is $R$-torsion-free. Indeed, we first note 
  \be\label{mrgeneralvanishing}H^0(\QQ,X)=0\ee
  for every subquotient $X$ of $\TT$, by \cite[Lemma 2.1.4]{mr02}.
  Let $H^1(\QQ_{\Sigma}/\QQ,\TT)_{\textup{tor}}$ denote the $R$-torsion submodule of $H^1(\QQ_{\Sigma}/\QQ,\TT)$. We also let $F$ denote the field of fractions of $R$ and set $\mathbb{V}=\TT\otimes F$, $\mathbb{W}=\mathbb{V}/\TT$. Observe that 
  $$H^1(\QQ_{\Sigma}/\QQ,\TT)_{\textup{tor}}=\ker\left(H^1(\QQ_{\Sigma}/\QQ,\TT)\lra H^1(\QQ_{\Sigma}/\QQ,\mathbb{V})\right)\cong H^0(\QQ,\mathbb{W}),$$
  and thus we are reduced to verify the vanishing of $H^0(\QQ,\mathbb{W})$. Suppose 
  $$0\neq t\otimes 1/g \in H^0(\QQ,\mathbb{W})$$ 
  (with $t\in \TT$ and $g\in R$). 
  We thus obtain a non-trivial element $\overline{T} \in H^0(\QQ,\TT/g\TT)$ and contradict \eqref{mrgeneralvanishing}. This argument applies verbatim for the regular ring $\frak{R}$ and the representation $\frak{T}$.
  
  Under the hypothesis (\textbf{H3}), the following conditions on an element $c \in H^1(\QQ_{\Sigma}/\QQ,Y)$ are therefore equivalent for $Y=\TT$ or $\frak{T}$ (and correspondingly, $X=R$ or $\frak{R}$):
  
    \begin{enumerate}
  \item $c\neq 0$.
  \item $c$ is not $X$-torsion.
  \item $c \notin \frak{p}H^1(\QQ_{\Sigma}/\QQ,Y)$ for infinitely many  height one primes $\frak{p} \subset X$.
  \item $c \notin \frak{p}H^1(\QQ_{\Sigma}/\QQ,Y)$ for some  height one prime $\frak{p} \subset X$.
  \end{enumerate}
  The only non-obvious step is to verify the implication $(2)\implies(3)$ and this may verified out following the proof of \cite[Lemma 2.1.7]{how2}. 
    \end{rem}
  
 Before we give the proofs of Theorems \ref{thm:weakleo1} and \ref{thm:weakleo2}, we prove two general preparatory lemmas from commutative algebra. Let $\mathcal{H}_0$ be a local integral domain which is a complete, regular $\oo$-algebra that has relative dimension $one$ over $\oo$. Set  $\mathcal{H}=\mathcal{H}_0[[X]]$ and suppose $\mathbf{T}$ is a finite $\mathcal{H}$-module endowed with a continuous action of $\textup{Gal}(\QQ_\Sigma/\QQ)$. 
 
 In applications, $\mathcal{H}$ will either be a certain quotient of $R=\oo[[X_1,X_2,X_3]]$ (and $\mathbf{T}=\TT\otimes_R \mathcal{H}$) or else $\mathcal{H}_0$ will be the ring $\RR$. In case of the latter, we will set $\mathcal{H}=\frak{R}:=\RR[[\Gamma]]$ and $\mathbf{T}$ will be the representation $\frak{T}=\mathcal{T}\otimes\LL$ that was introduced in the previous section. Note that in either case, $\mathbf{T}$ satisfies the hypotheses (\textbf{H1})-(\textbf{H4}), (\textbf{H.Tam}) and (\textbf{H.nA}). Suppose further until the end that $\chi(\overline{T})=1$.


 \begin{lemma}
 \label{lem:finiteimpliestorsion}
 Suppose $M$ is an $\mathcal{H}$-module. Assume for a height one prime $\wp$ of $\mathcal{H}_0$ and an integer $N$, the quotient $M/(\wp,X+p^N)M$ is of finite order. Then $M$ is a finitely generated $\mathcal{H}$-torsion module.
 \end{lemma}
 \begin{proof}
We first give a proof assuming that $p \notin \wp$. To ease notation, write $X_N=X+p^N$. One can find an integer $s$ so that
 \be\label{eqn:killM}p^s\cdot\left(M/(\wp,X_N)M\right)=0.\ee
 By Nakayama's Lemma $M$ is finitely generated as an $\mathcal{H}$-module, say by $m_1,\dots, m_r \in M$. It follows from (\ref{eqn:killM}) that
 $$p^sm_i=\sum_{j=1}^r a_{j}^{(i)}m_j,$$
 where $a_{j}^{(i)} \in (\wp,X_N)$. Setting $A=[a_j^{(i)}]$ and $B=A-p^s\cdot I_{r\times r}$, we conclude by (\ref{eqn:killM}) that
 \begin{align*}
 &\sum_{j=1}^{i-1} a_{j}^{(i)}m_j + (a_i^{(i)}-p^s)m_i+ \sum_{j=i+1}^{r} a_{j}^{(i)}m_j=0\\
&\implies (A-p^s\cdot I_{r\times r})\left[ \begin{array}{c}m_1\\
\vdots\\
m_r
\end{array}\right]=B\cdot\left[ \begin{array}{c}m_1\\
\vdots\\
m_r
\end{array}\right]=0\\
&\implies \textup{adj}(B)\cdot B\left[ \begin{array}{c}m_1\\
\vdots\\
m_r
\end{array}\right]=0\implies \det(B)\cdot M=0.
 \end{align*}
To conclude with the proof of the lemma, we check that the element $\det(B) \in \mathcal{H}$ is non-zero. Observe that
\begin{align*}\det(B)=\det(A-p^s\cdot I_{r\times r})&\equiv (-1)^{r}p^{sr} \mod (\wp,X_N)\\
&\not\equiv 0 \mod (\wp,X_N),
\end{align*}
as the ring $\mathcal{H}/(\wp,X_N)\cong \mathcal{H}_0/\wp$ is an integral domain of characteristic zero, as we have assumed $p\notin\wp$.

In case $p\in\wp$, note that the given condition on $M$ translates into  the statement that the module $M/(\wp,X)M$ has finite cardinality. If $M$ were not a torsion module, $M/(\wp,X)M$ would contain a submodule isomorphic to $\mathcal{H}_0/\wp$. The latter ring, however, has infinite cardinality as it at least contains a ring isomorphic to a formal power series ring in one variable over $\mathbb{F}_p$. 
  \end{proof}
\begin{lemma}
 \label{lem:torsionimpliestorsion}
 If the $R$-module $H^1_{\FFc^*}(\QQ,{\TT}^*)^{\vee}$ is torsion, then for all but finitely many height one primes $\wp\subset R$, the $R/\wp$-module $H^1_{\FFc^*}(\QQ,{\TT}^*)^{\vee}/\wp H^1_{\FFc^*}(\QQ,{\TT}^*)^{\vee}$ is torsion.
 \end{lemma}
 \begin{proof}
  Let $M$ be any finitely generated torsion $R$-module, with generators $m_1,\cdots,m_r$. Since $M$ is torsion, it follows that for all but finitely many height one primes $\wp$ of $R$, we have $M_\wp=0$. This in particular means for every $1\leq i\leq r$, there is $s_i\in R-\wp$ such that $s_im_i=0$. Set $s=s_1\cdots s_r$ and let $\bar{s}\in R/\wp$ denote the homomorphic image of $s$. Note that $\bar{s}\neq 0$ and that $\bar{s}\cdot M/\wp M=0$. This shows that the $R/\wp$-module $M/\wp M$ is torsion for almost all height one primes $\wp$. This is exactly the assertion of the Lemma with $M=H^1_{\FFc^*}(\QQ,{\TT}^*)^{\vee}$.
  
  \end{proof}
  \begin{proof}[Proof of Theorem~\ref{thm:weakleo1}]

We first give a proof for the statement concerning the ring $\frak{R}$ and later use this result to deduce the statement for the ring $R$. 

For any ideal $I$ of $\frak{R}$ and any subquotient $M$ of $\frak{T}$, we have by \cite[3.5.2]{mr02} and the proof of \cite[Lemma 3.5.3]{mr02} (both of which apply thanks to our running hypothesis (\textbf{H3})) that
  $$H^1(\QQ_{\Sigma(\FFc)}/\QQ,{M}^*[I]) \stackrel{\sim}{\lra} H^1(\QQ_{\Sigma(\FFc)}/\QQ,{M}^*)[I]$$
  and hence also an injection
  $$H^1_{\FFc^*}(\QQ,{M}^*[I]) \hookrightarrow H^1_{\FFc^*}(\QQ,{M}^*)[I],$$
where $\FFc^*$ on $M^*$ is induced from the Selmer structure $\FFc$ on $\frak{T}^*$ by propagation. Passing to Pontryagin duals, we thus obtain a surjection
\be\label{eqn:surjpontr}
H^1_{\FFc^*}(\QQ,{M}^*)^\vee\otimes \frak{R}/I \twoheadrightarrow H^1_{\FFc^*}(\QQ,{M}^*[I])^\vee.
\ee
By the assumption that $H^1_{\FFc^*}(\QQ,\frak{T}^*)^{\vee}$ is $\frak{R}$-torsion, one may choose  by \cite[Theorem 6.5]{matsumura} a specialization 
$$\frak{s}_\wp: \RR \lra S_\wp$$
into the ring of integers $S_\wp$ of a finite extension $\Phi_\wp$ of $\QQ_p$, whose kernel $\wp$ is a height one prime $\wp\subset \RR$ and satisfies with the following properties:
\begin{itemize}
\item $S_\wp$ is integral closure of the integral domain $\oo_\wp:=\RR/\wp$ in $\textup{Frac}(\oo_\wp)=\Phi_\wp$,
\item $\wp \notin \textup{Supp}_{\frak{R}}(H^1_{\FFc^*}(\QQ,\frak{T}^*)^{\vee})$, where $\wp$ here denotes by slight abuse the height one prime which is the kernel of the induced map
$$\frak{R}=\RR[[\Gamma]]\stackrel{\frak{s}_\wp}{\lra} S_\wp[[\Gamma]].$$ 
\end{itemize}
We denote the induced ring homomorphism
$\oo_\wp \hookrightarrow S_\wp$ also by $\frak{s}_\wp$.

For $\wp$ chosen as above, it follows that the module
$H^1_{\FFc^*}(\QQ,\frak{T}^*)^{\vee}/\wp$ is $\oo_\wp[[\Gamma]]$-torsion. By (\ref{eqn:surjpontr}) this implies that the module
$$H^1_{\FFc^*}(\QQ,\frak{T}^*[\wp])^\vee \cong H^1_{\FFc^*}(\QQ,(\mathcal{T}/\wp\mathcal{T}\otimes\LL)^*)^\vee$$
is $\oo_p[[\Gamma]]$-torsion as well. It is therefore possible (using Hensel's Lemma and \cite[Theorem 6.5]{matsumura}) to choose an $N>>0$ such that
\begin{itemize}
\item $\oo_\wp[[\Gamma]]/(\gamma-1+p^N) \cong \oo_\wp$,
\item  $\gamma-1+p^N \notin \textup{Supp}_{\oo_\wp[[\Gamma]]}\left(H^1_{\FFc^*}(\QQ,(\mathcal{T}/\wp\mathcal{T}\otimes\LL)^*)^\vee\right)$.
\end{itemize}
For $N$ chosen as above, we therefore have that the module
$$H^1_{\FFc^*}(\QQ,(\mathcal{T}/\wp\mathcal{T}\otimes\LL)^*)^\vee/(\gamma-1+p^N)$$
is $\oo_{\wp}$-torsion. Setting $T(\wp,N):=\mathcal{T}/\wp\mathcal{T}\otimes\LL/(\gamma-1+p^N)$ and applying (\ref{eqn:surjpontr}) again, we conclude that the module
$$H^1_{\FFc^*}(\QQ,(\mathcal{T}/\wp\mathcal{T}\otimes\LL/(\gamma-1+p^N))^*)^\vee\cong H^1_{\FFc^*}(\QQ,T(\wp,N)^*)^\vee$$
is $\oo_{\wp}$-torsion, hence finite. 

When we do not vary $N$, we write $T(\wp)$ in place of $T(\wp,N)$ to ease notation. Let $T=T(\wp)\otimes_{\frak{s}_\wp}S_\wp$ and define the Selmer structure $\FF_T$ by setting
\begin{itemize}
\item $\Sigma(\FF_T)=\Sigma(\FFc)=:\Sigma$,
\item $H^1_{\FF_T}(\QQ_p,T)=H^1(\QQ_p,T)$,
\item $H^1_{\FF_T}(\QQ_\ell,T)=\ker\left(H^1(\QQ_\ell,T)\lra H^1(I_\ell,T\otimes_{\ZZ_p}\QQ_p)\right)$, for $\ell\neq p$.
\end{itemize}
Note that $\FF_T$ is exactly what Mazur and Rubin call the \emph{canonical Selmer structure} on $T$. Let $\iota$ denote the injection $T(\wp) \hookrightarrow T$. Then $\iota$ induces maps
$$H^1(\QQ_{\Sigma}/\QQ,T^*) \lra H^1(\QQ_{\Sigma}/\QQ,T(\wp)^*)$$
$$H^1(\QQ_\ell,T(\wp)) \lra H^1(\QQ_\ell,T)$$
$$H^1(\QQ_\ell,T^*) \lra H^1(\QQ_\ell,T(\wp)^*)$$
for every prime $\ell$.
It is easy to see that the image of $H^1_{\FFc}(\QQ_\ell,T(\wp))$ lands in $H^1_{\FF_T}(\QQ_\ell,T)$ for every $\ell$ (and by local duality, the image of $H^1_{\FF_T^*}(\QQ_\ell,T^*)$ therefore lands in $H^1_{\FFc^*}(\QQ_\ell,T(\wp)^*)$). We thence obtain a map
\be\label{eqn:wantfinitekernelcokernel} H^1_{\FF_T^*}(\QQ,T^*)\lra H^1_{\FFc^*}(\QQ,T(\wp,N)^*).\ee
In Lemma~\ref{lemma:reallywanted} below we check that the kernel and the cokernel of this map is finite for all sufficiently large $N$. This shows that $H^1_{\FF_T^*}(\QQ,T^*)$ is of finite order for $N>>0$, as we have already verified above that $H^1_{\FFc^*}(\QQ,T(\wp,N)^*)$ is finite.

Let $\pmb{\kappa} \in \overline{\mathbf{KS}}(\frak{T},\FFc,\PP)$ be a generator so that its image $\bar{\pmb{\kappa}}\in {\mathbf{KS}}(\overline{T},\FFc,\PP)$ is non-zero by Theorem~\ref{thm:KSmain}. Hence, the image $\pmb{\kappa}^{(\wp)}$ of $\pmb{\kappa}$ in $\overline{\mathbf{KS}}({T},\FF_{T},\PP)$ is non-zero as well. Corollary 5.2.13 of \cite{mr02} applies thanks to our running hypotheses and it follows that the $\kappa^{(\wp)}_1\neq0$ and hence $\kappa_1\neq 0$. 

As for the assertion for $\TT$, we first use Lemma~\ref{lem:torsionimpliestorsion} to find a height one prime of the form $\wp=(X_3+p^M)R$ that verifies the conclusion of Lemma~\ref{lem:torsionimpliestorsion}. For the chosen height one prime $\wp$ set $\mathcal{H}=R/\wp\cong \oo[[X_1,X_2]]$ and $\mathbf{T}=\TT\otimes_{R}\mathcal{H}$. Observe (using \eqref{eqn:surjpontr}) that 
\be\label{eqn:reducedimensionbyone}H^1_{\FFc^*}(\QQ,\mathbf{T}^*)\cong H^1_{\FFc^*}(\QQ,\TT^*)[\wp]\ee
and upon taking Pontryagin duals
$$H^1_{\FFc^*}(\QQ,\mathbf{T}^*)^\vee\cong H^1_{\FFc^*}(\QQ,\TT^*)^\vee/\wp H^1_{\FFc^*}(\QQ,\TT^*)^\vee.$$
Let $\pmb{\kappa} \in \overline{\mathbf{KS}}({\TT},\FFc,\PP)$ be any generator and let 
$$\varphi: \overline{\mathbf{KS}}({\TT},\FFc,\PP) \lra \overline{\mathbf{KS}}(\mathbf{T},\FFc,\PP)$$
denote the map induced from $R\ra \mathcal{H}$. The commutativity of the diagram
\be\label{diagram:primitive}\xymatrix{\overline{\textbf{KS}}(\TT,\FFc,\PP)\ar[rr]^{\varphi}\ar@{->>}[dr]&&\overline{\textbf{KS}}(\mathbf{T},\FFc,\PP)\ar[dl]\\
&{\textbf{KS}}(\overline{T},\FFc,\PP)&
} \ee
shows that $\varphi(\pmb{\kappa})$ generates the free $\mathcal{H}$-module $\overline{\mathbf{KS}}(\mathbf{T},\FFc,\PP)$ of rank one. Furthermore, by our choice of the prime ideal $\wp$ and (\ref{eqn:reducedimensionbyone}), the $\mathcal{H}$-module $H^1_{\FFc^*}(\QQ,\mathbf{T}^*)^\vee$ is torsion. Using the proof above for the ring $\mathcal{H}$ in place of $\frak{R}$, we conclude that 
$$\varphi({\kappa}_1)=\varphi(\pmb{\kappa})_1\neq 0,$$ 
in particular that $\kappa_1\neq 0$.

This completes the proof of Theorem~\ref{thm:weakleo1}\,, modulo Lemma~\ref{lemma:reallywanted} below.
  \end{proof}




  \begin{lemma}
  \label{lemma:reallywanted}
Let $T$ and $T(\wp,N)$ be as in the proof of Theorem~\ref{thm:weakleo1}. When the positive integer $N$ is sufficiently large, both the kernel and the cokernel of the map
$$ H^1_{\FF_T^*}(\QQ,T^*)\lra H^1_{\FFc^*}(\QQ,T(\wp,N)^*)$$
are finite.
  \end{lemma}

  \begin{proof}
  We first verify that the kernels and the cokernels  of the maps
  \be\label{eqn:finitewant1} H^1(\QQ_{\Sigma}/\QQ,T^*) \lra H^1(\QQ_{\Sigma}/\QQ,T(\wp,N)^*)
\ee
\be\label{eqn:compatecanTtocan}H^1_{\FF_T^*}(\QQ_\ell,T^*)\lra H^1_{\FFc^*}(\QQ_\ell,T(\wp,N)^*)\ee
have finite order for every prime $\ell$. When $N$ is fixed, we will denote $T(\wp,N)$ simply by $T(\wp)$.

Observe that the kernel of the map (\ref{eqn:finitewant1}) lives in $H^0(\QQ_{\Sigma}/\QQ,(T/T(\wp))^*)$ and its cokernel in $H^1(\QQ_{\Sigma}/\QQ,(T/T(\wp))^*)$, which are both finite.

As for the map (\ref{eqn:compatecanTtocan}) when $\ell=p$, our running hypothesis (\textbf{H.nA}) along with the fact that the ideal $\wp$ is principal\footnote{This is the only point in the proof of Theorems~\ref{thm:weakleo1} and \ref{thm:weakleo2} where we use that the ring $\RR$ is regular in an essential way.} (being a height-one prime of the regular ring $\RR$) show that
$$H^1_{\FFc}(\QQ_p,T(\wp)):=\textup{im}\left(H^1(\QQ_p,\frak{T})\lra H^1(\QQ_p,T(\wp))\right)=H^1(\QQ_p,T(\wp)),$$
hence we have
$$ H^1_{\FF_T^*}(\QQ_\ell,T^*)=0=H^1_{\FFc^*}(\QQ_\ell,T(\wp)^*)\,;$$
so the kernel and cokernel of (\ref{eqn:compatecanTtocan}) are trivial. It remains to control the kernel and the cokernel of (\ref{eqn:compatecanTtocan}) when $\ell\neq p$. The kernel of
\be\label{eqn:finitewewantdual}
H^1_{\FFc}(\QQ_\ell,T(\wp))\lra H^1_{\FF_T}(\QQ_\ell,T)
\ee
 is controlled by
$$\ker\left(H^1(\QQ_\ell,T(\wp))\lra H^1(\QQ_\ell,T)\right)= \textup{im}\left(H^0(\QQ_\ell,T/T(\wp))\lra H^1(\QQ_\ell,T(\wp))\right)$$
which is visibly  finite.

We finally prove that the cokernel of (\ref{eqn:finitewewantdual}) is finite. Consider now the commutative diagram
$$\xymatrix{0 \ar[r]&H^1_{\textup{ur}}(\QQ_\ell,T(\wp))\ar[r]\ar[d]&H^1(\QQ,T(\wp))\ar[r]\ar[d]&H^1(I_\ell,T(\wp))^{\textup{Fr}_\ell=1}\ar[r]\ar[d]&0\\
0 \ar[r]&H^1_{\textup{ur}}(\QQ_\ell,T)\ar[r]&H^1(\QQ,T)\ar[r]&H^1(I_\ell,T)^{\textup{Fr}_\ell=1}\ar[r]&0
}$$
The cokernel of the vertical map in the middle is controlled by $H^2(\QQ_\ell,T/T(\wp))$ hence it is finite. Also, the kernel of the the rightmost is finite for a similar reason. This shows by snake lemma that the cokernel of the leftmost vertical map is also finite. As the index of $H^1_{\textup{ur}}(\QQ_\ell,T)$ in $H^1_f(\QQ_\ell,T)=H^1_{\FF_T}(\QQ_\ell,T)$ is finite as well, we therefore proved that
\be\label{eqn:wantstep1}\hbox{ the cokernel of the map }
H^1_{\textup{ur}}(\QQ_\ell,T(\wp))\lra H^1_{\FF_T}(\QQ_\ell,T)
\hbox{ is finite.}\ee
 Furthermore, it is not hard to see that the $\LL$-module
$$H^1_{\textup{ur}}(\QQ_\ell,\mathcal{T}/\wp\mathcal{T}\otimes\LL)=H^1(G_{\QQ_\ell}/I_\ell,(\mathcal{T}/\wp\mathcal{T})^{I_\ell}\otimes\LL)$$
 is $\LL$-torsion. (Note in the equality above we use the fact that $I_\ell$ acts trivially on $\LL$.) 
Choosing the positive integer $N>>0$ above so that $\gamma-1+p^N$ does not divide the characteristic ideal of this module, we obtain a finite quotient
$H^1(G_{\QQ_\ell}/I_\ell,(\mathcal{T}/\wp\mathcal{T})^{I_\ell}\otimes\LL)/(\gamma-1+p^N).$
Since the cohomological dimension of $G_{\QQ_\ell}/I_\ell$ is one, we have
\begin{align*} H^1(G_{\QQ_\ell}/I_\ell,(\mathcal{T}/\wp\mathcal{T})^{I_\ell}\otimes\LL)/(\gamma-1+p^N)&\stackrel{\sim}{\lra} H^1(G_{\QQ_\ell}/I_\ell,(\mathcal{T}/\wp\mathcal{T})^{I_\ell}\otimes\LL/(\gamma-1+p^N))\\&\cong H^1(G_{\QQ_\ell}/I_\ell,(\mathcal{T}/\wp\mathcal{T}\otimes\LL/(\gamma-1+p^N))^{I_\ell})\\
&\cong  H^1(G_{\QQ_\ell}/I_\ell,T(\wp,N)^{I_\ell})=H^1_{\textup{ur}}(\QQ_\ell,T(\wp,N)),
\end{align*}
where $T(\wp,N)=\mathcal{T}/\wp\mathcal{T}\otimes\LL/(\gamma-1+p^N)$ as above. In particular,
the index of $H^1_{\FFc}(\QQ,T(\wp,N))$ in $H^1_{\textup{ur}}(\QQ_\ell,T(\wp,N))$ is finite for $N>>0$. This, together with (\ref{eqn:wantstep1}) shows that the kernel and cokernel of the  map (\ref{eqn:finitewewantdual}), and by local duality, also the kernel and the cokernel of the map (\ref{eqn:compatecanTtocan}) are finite for $N>>0$. 

Using the fact that the kernels and cokernels of the maps (\ref{eqn:finitewant1}) and (\ref{eqn:compatecanTtocan}) are both finite the proof of the lemma follows at once.
  \end{proof}

  \begin{proof}[Proof of Theorem~\ref{thm:weakleo2}]
  As the proof of this Theorem in fact follows from a more general statement due to Ochiai~\cite{ochiaideform} (see the proof Theorem 2.4 and Remark 2.5 of loc.cit.), we only give a sketch of the proof and only in the situation concerning the ring $\frak{R}$ and the representation $\frak{T}$. We use the notation from the proof of Theorem~\ref{thm:weakleo1}.

  Let $\pmb{\kappa}\in  \overline{\mathbf{KS}}(\frak{T},\FFc,\PP)$ be a given generator. Since $\kappa_1$ is non-torsion, it follows that there is a height one prime $\wp$ of $\RR$ as in the proof of Theorem~\ref{thm:weakleo1} and a positive integer $N$ (chosen in way that the conclusion of Lemma~\ref{lemma:reallywanted} holds true) such that the image
  $$\textup{red}_{\wp,N}(\kappa_1) \in H^1_{\FFc}(\QQ,T(\wp,N))$$
  of $\kappa_1$ is non-zero. Fix such $\wp$ and $N$; define $T$ (and the Selmer structure $\FF_T$) as in the proof of Theorem~\ref{thm:weakleo1}. We let $\pmb{\kappa}^{(\wp)} \in \overline{\mathbf{KS}}(T,\FF_T,\PP)$ be the image of $\pmb{\kappa}$. By (\textbf{H3}), the map
  $$H^1(\QQ_{\Sigma}/\QQ,T(\wp,N))\lra H^1(\QQ_{\Sigma}/\QQ,T)$$
is injective. In particular, the image $\kappa_1^{(\wp)}$ of $\textup{red}_{\wp,N}(\kappa_1)$ inside $H^1_{\FF_T}(\QQ,T)$ is non-zero. Let $\pmb{\kappa}^{(\wp)} \in \overline{\mathbf{KS}}(T,\FF_T,\PP)$ be the image of $\pmb{\kappa}$. We therefore showed the existence of  a Kolyvagin system $\pmb{\kappa}^{(\wp)} \in \overline{\mathbf{KS}}(T,\FF_T,\PP)$ whose leading term verifies that $\kappa_1^{(\wp)}\neq0$. This shows that $H^1_{\FF_T^*}(\QQ,T^*)$ is finite.

By Lemma~\ref{lemma:reallywanted}, we have a map
$$H^1_{\FF_T^*}(\QQ,T^*) \lra H^1_{\FFc^*}(\QQ,T(\wp,N)^*)$$
with finite kernel and cokernel. Hence  $H^1_{\FFc^*}(\QQ,T(\wp)^*)$ is finite as well. We conclude by (\ref{eqn:surjpontr}) that
$$H^1_{\FFc^*}(\QQ,\frak{T}^*)^{\vee}/(\wp,\gamma-1+p^N) \cong H^1_{\FF_T^*}(\QQ,T^*)^{\vee}$$
is also finite. It follows from Lemma~\ref{lem:finiteimpliestorsion} that $H^1_{\FFc^*}(\QQ,\frak{T}^*)^{\vee}$ is $\frak{R}$-torsion, as desired.
  \end{proof}

  \begin{thm}
  \label{thm:h1free}
Suppose $H^1_{\FFc^*}(\QQ,\frak{T}^*)^{\vee}$ is $\frak{R}$-torsion. Under the running hypotheses of this section, the $\frak{R}$-module $H^1_{\FFc}(\QQ,\frak{T})$ is free of rank one.

 Similar statement holds true when $\frak{R}$ is replaced by $R$ and $\frak{T}$ by $\TT$.
  \end{thm}
  \begin{proof}
  To simplify the arguments we suppose in addition that the ring $\RR$ is the power series ring $\oo[[X]]$; the general case when $\RR$ is a general regular $\oo$-algebra of dimension two may be treated after minor alterations. As above, choose a positive integer $N>>0$ so that
  \begin{itemize}
  \item $\oo[[X]]/(X+p^N)\cong \oo$,
  \item $H^1_{\FFc^*}(\QQ,\frak{T}^*)^{\vee}/(X+p^N)$ is $\LL$-torsion.
  \end{itemize}
  By setting $T:=\mathcal{T}/(X+p^N)\mathcal{T}$, we conclude using (\ref{eqn:surjpontr}) that
 the module $H^1(\QQ, (T\otimes\LL)^*)^{\vee}$ is $\LL$-torsion. Similarly, choose a positive integer $M>>0$ such that
 $$H^1_{\FFc^*}(\QQ, (T\otimes\LL)^*)^{\vee}/(\gamma-1+p^M)\cong H^1_{\FFc^*}(\QQ,\dot{T})^{\vee}$$
is finite. Here, $\dot{T}$ is the free $\oo$-module $T\otimes\LL/(\gamma-1+p^M)$. By \cite[Corollary 5.2.6]{mr02}, it follows that $\textup{rank}_\oo(H^1_{\FFc}(\QQ,\dot{T}))=\chi(\overline{T})=1$. Furthermore, the $\oo$-module $H^1_{\FFc}(\QQ,\dot{T})$ is torsion-free as since we assumed (\textbf{H3}), hence we conclude that $H^1_{\FFc}(\QQ,\dot{T})$ is a free $\oo$-module of rank one.

Set $X_1=X+p^N$ and $X_2=\gamma-1+p^M$ for $M,N$ as above and define $\frak{R}_{u,v}=\frak{R}/(X_1^u,X_2^v)$,  $\frak{R}_{r,u,v}=\frak{R}/(\varpi^r, X_1^u,X_2^v)$, $\frak{T}_{u,v}=\frak{T}\otimes_{\frak{R}}\frak{R}_{u,v}$ and $\frak{T}_{r,u,v}=\frak{T}\otimes_{\frak{R}}\frak{R}_{r,u,v}$. Note that $\frak{T}_{1,1}=\dot{T}$. As $H^1_{\FFc}(\QQ,\frak{T}_{u,v})=\varprojlim_{r} H^1_{\FFc}(\QQ,\frak{T}_{r,u,v})$, it follows by the proof of Proposition~\ref{prop:upperbound} that
$$H^1_{\FFc}(\QQ,\frak{T}_{1,1})\stackrel{\sim}{\lra} H^1_{\FFc}(\QQ,\frak{T}_{u,v})[X_1^{u-1},X_2^{v-1}].$$
This shows that the module
\begin{align*}\textup{Hom}\,_{\frak{R}_{u,v}}(H^1_{\FFc}(\QQ,\frak{T}_{u,v}),{\frak{R}_{u,v}})/(X_1^{u-1},X_2^{v-1})&\cong \textup{Hom}\,_{\frak{R}_{u,v}}\left(H^1_{\FFc}(\QQ,\frak{T}_{u,v})[X_1^{u-1},X_2^{v-1}],{\frak{R}_{u,v}}\right)\\
&\cong \textup{Hom}_{\oo}\left(H^1_{\FFc}(\QQ,\frak{T}_{1,1}),\oo\right),
\end{align*} 
is cyclic, hence by Nakayama's Lemma (along with the fact that the $\oo$-module $H^1_{\FFc}(\QQ,\frak{T}_{1,1})=H^1_{\FFc}(\QQ,\dot{T})$ is free of rank one) we conclude that the module $\textup{Hom}_{\frak{R}_{u,v}}\left(H^1_{\FFc}(\QQ,\frak{T}_{u,v}),{\frak{R}_{u,v}}\right)$ is cyclic as well.

On the other hand, (\textbf{H3}) shows that the module $H^1_{\FFc}(\QQ,\frak{T}_{u,v})$ is $\oo$-torsion free and the proof of Proposition~\ref{prop:lowerbound} shows $\textup{rank}_{\oo}(H^1_{\FFc}(\QQ,\frak{T}_{u,v}))\geq uv$. This shows that the cyclic $\frak{R}_{u,v}$-module $\textup{Hom}_{\frak{R}_{u,v}}\left(H^1_{\FFc}(\QQ,\frak{T}_{u,v}),{\frak{R}_{u,v}}\right)$ is indeed free of rank one, hence the module $H^1_{\FFc}(\QQ,\frak{T}_{u,v})$ itself is free of rank one as an $\frak{R}_{u,v}$-module. Passing to limit we conclude with the proof of the Theorem when the coefficient ring is $\frak{R}$. In the situation when the coefficient ring is $R$, one easily reduces to the case discussed above using Lemma~\ref{lem:torsionimpliestorsion}.
  \end{proof}
 \begin{lemma}
 \label{lem:maindescentlemma}
 Suppose  $S$ is a regular ring and $\pi_t: S[[t]]\ra S$ the natural ring homomorphism induced by $t\mapsto 0$. $h \in S[[t]]$ and $M,N$ are torsion $S[[t]]$-modules,  $M^\prime,N^\prime$ are torsion $S$-modules such that 
 \begin{itemize}
 \item[(i)] $\textup{char}_{S[[t]]}(M)=\textup{char}_{S[[t]]}(N)\cdot h$\,,
 \item[(ii)] $M[t]$ is pseudo-null\,,
 \item[(iii)] $\textup{char}_S(M^\prime/\pi_t(M))=\textup{char}_S(N[t])$\,.
 \end{itemize}
 where $X[t]$ stands as usual for the submodule of an $S[[t]]$-module $X$ annihilated by $tS[[t]]$. Then
$$\textup{char}_S(M^\prime)=\textup{char}_S(N^\prime)\cdot\pi_t(h)\,.$$ 
 \end{lemma}
 \begin{proof}
 This follows from \cite[Proposition 1.1]{bandiniNY}.
 \end{proof}

Recall that $\pmb{\kappa} \in \overline{\mathbf{KS}}({\TT},\FFc,\PP)$ is a fixed generator and $\kappa_1 \in H^1_{\FFc}(\QQ,\TT)$ is its leading term.

 \begin{lemma}
 \label{lemma:preparationtosharpen}
 Suppose that $\kappa_1 \neq 0$. Then there exists positive integers $\alpha_1, \alpha_2,\alpha_3$ such that 
\begin{itemize}
\item $R/(X_1+p^{\alpha_1}, X_2+p^{\alpha_2}, X_3+p^{\alpha_3})\cong \oo$,
\item $H^1_{\FFc^*}(\QQ,\dot{T}^*)$ is finite, where $\dot{T}:=\TT/(X_1+p^{\alpha_1}, X_2+p^{\alpha_2}, X_3+p^{\alpha_3})$\,.
\end{itemize} 
 \end{lemma}
 \begin{proof}
 the $R$-module $H^1_{\FFc^*}(\QQ,\TT^*)^{\vee}$ is torsion by Theorem~\ref{thm:weakleo2} and in this case, one may find a triple $(\alpha_1,\alpha_2,\alpha_3)$ by proceeding as in the proof of Theorem~\ref{thm:h1free}, by iteratively using (\ref{eqn:surjpontr}).
 \end{proof}
 \begin{define}
 \label{def:goodquotients}
In the situation of Lemma~\ref{lemma:preparationtosharpen}, we set $Y_i=X_i+p^{\alpha_i}$. We also define the quotient rings $S_0=R/(Y_1,Y_2,Y_3)$, $S_1=R/(Y_2,Y_3)$, $S_2=R/Y_3$ and $S_3=R$ with natural surjections $\pi_i: R\twoheadrightarrow S_i$. Note that each $S_i$ is a power series ring over $S_{i-1}$ in one variable (which may be naturally identified with $Y_i$). Set $T_i=\TT\otimes_{R}S_i$\,. Observe that $T_3=\TT$ and $T_0=\dot{T}$.
 \end{define}
 \begin{define}
 \label{def:selmercomplexforTi}
 In the setting of Lemma~\ref{lemma:preparationtosharpen} and with the notation of Definition~\ref{def:goodquotients}, we define Nekov\'a\v{r}'s Selmer complex $C^{\bullet}_f(\QQ,X)$ for $X=T_i$ or $T_i^*$ (i=1, 2, 3) as the following complex of (co-)finite type $S_i$-modules:
$$C^{\bullet}_f(\QQ,X)=\textup{Cone}\left(C^{\bullet}_{\textup{cont}}(\QQ_\Sigma/\QQ,X)\oplus \bigoplus_{\ell\in\Sigma}C^{\bullet}_f(\QQ_\ell,X) \lra \bigoplus_{\ell\in\Sigma} C^{\bullet}_f(\QQ_\ell,X)\right)[-1]$$
where
 $$C^{\bullet}_f(\QQ,X)=\left\{\begin{array}{lc}C^{\bullet}_{\textup{cont}}(\QQ_p,U_p^+(X)) &\hbox{ if }\ell=p,\\\\
 C^{\bullet}_{\textup{cont}}(G_{\QQ_\ell}/I_\ell,X^{I_\ell}) & \hbox{ if } \ell\neq p, 
 \end{array}\right.$$
and  $U_p^+(T_i)=T_i$\,, $U_p^+(T_i^*)=0$.
Let $R\Gamma_f(\QQ,X)$ denote the corresponding object in the derived category and let $\widetilde{H}^j_f(\QQ,X)$ be the $i$th cohomology of $R\Gamma_f(\QQ,X)$ in degree $j$. 
 \end{define}
\begin{prop}
\label{prop:propertiesofselmercomplex}For $i=1,2,3: $
\begin{enumerate}\item The complex $R\Gamma_f(\QQ,T_i)$ may be represented by a perfect complex $($in the sense of \cite[Exp. I, Cor. 5.8.1]{sga6}$)$ of $S_i$-modules concentrated in degrees $1$ and $2$. 
\item There is a natural isomorphism $H^1_{\FFc^*}(\QQ,T_i^*)^\vee \cong \widetilde{H}^2_f(\QQ,T_i)\,.$ 
\item The following sequence is exact:
$$H^1_{\FFc}(\QQ,T_i)/Y_i\cdot H^1_{\FFc}(\QQ,T_i)\lra H^1_{\FFc}(\QQ,T_{i-1})\lra H^1_{\FFc^*}(\QQ,T_i^*)^\vee[Y_i]\lra 0$$
\end{enumerate}
\end{prop}
\begin{proof}
(1) follows from the fact our ring $S_i$ is regular, $T_i$ is a free $S_i$-module  and using a result of Serre and Auslander-Buchsbaum. Due to $p$-cohomological dimension considerations, it is easy to see that the cohomology of the complex $R\Gamma_f(\QQ,\TT)$ is concentrated in degrees $[0,3]$. By (\textbf{H3}) the cohomology $\widetilde{H}^0_f(\QQ,T_i)$ in degree zero vanishes. By Matlis duality we have $\widetilde{H}^3_f(\QQ,T_i)\cong\tilde{H}^0_{f}(\QQ,T_i^*)^\vee$ and the cohomology of $R\Gamma_f(\QQ,\TT)$ in degree $3$ vanishes thanks to (\textbf{H3}) as well. 

It follows from \cite[Lemma 9.6.3]{nek} and our running hypothesis (\textbf{H.nA}) that $\widetilde{H}^1_f(\QQ,T_i^*)\cong H^1_{\FFc^*}(\QQ,T_i^*)$. The isomorphism in (2) is then induced by Matlis duality (c.f., \cite[8.9.6.1]{nek}). We remark that since the residue field of $R$ is finite, Matlis duality functor coincides with the Pontryagin duality functor.

(3) follows from Nekov\'a\v{r}'s control theorem \cite[Proposition 8.10.1]{nek}, using the identification in (2) and the isomorphism $\widetilde{H}^1_f(\QQ,T_i)\cong H^1_{\FFc}(\QQ,T_i)$ (which follows from the exact sequence of \cite[Lemma 9.6.3]{nek}).
\end{proof}
\begin{proof}[Proof of Theorem~\ref{thm:sharpening}]
Under the assumptions of the current section, one may show that 
\be\label{eqn:boundselmer}\textup{char}\left(H^1_{\FFc}(\QQ,\TT)/R\cdot\kappa_1\right)=\textup{char}\left(H^1_{\FFc^*}(\QQ,\TT^*)^{\vee}\right)\cdot h\ee
 for some $h\in R$, following the arguments of \cite{ochiaideform}; see particularly the proof of Theorem 2.4 in loc.cit. We will make use Nekov\'a\v{r}'s descent formalism alluded to above in order to verify that that $h\in R^\times$ and therefore, deduce Theorem~\ref{thm:sharpening}.
  
Assume without loss of generality that $\kappa_1\neq 0$ (as otherwise, Theorem~\ref{thm:weakleo1} shows that both sides of the claimed equality are $0$ and Theorem~\ref{thm:sharpening} holds true for trivial reasons). 
Let ${\pmb{\dot{\kappa}}}$ be the image of the generator $\pmb{\kappa}$ under the map 
$$ \overline{\mathbf{KS}}({\TT},\FFc,\PP) \lra  \overline{\mathbf{KS}}(\dot{T},\FFc,\PP)$$ 
induced from $\pi_0:\,R\twoheadrightarrow S_0$. The diagram (\ref{diagram:primitive}) (with $T$ replaced by $\dot{T}$) shows that the Kolyvagin system ${\pmb{\dot{\kappa}}}$ is primitive. Let $\dot{\kappa}_1\in H^1_{\FFc}(\QQ,\dot{T})$ denote its leading term. It follows from \cite[Theorem 5.2.14]{mr02} that
\be\label{eqn:sharpbasecase}
\textup{char}\left(H^1_{\FFc^*}(\QQ,\dot{T}^*)^{\vee}\right) = \textup{char}\left(H^1_{\FFc}(\QQ,\dot{T})/S_0\cdot\dot{\kappa}_1\right)\,.
\ee
One may inductively verify that the hypotheses of Lemma~\ref{lem:maindescentlemma} hold true with $S=S_{i-1}$\,, $S[[t]]=S_i$\,, $M=H^1_{\FFc}(\QQ,T_i)/S_i\cdot \pi_i(\kappa_1)$ and $N=H^1_{\FFc^*}(\QQ,T_i^*)^\vee$ for every $i=1, 2, 3$ (basically, using Theorem~\ref{thm:h1free} and Proposition~\ref{prop:propertiesofselmercomplex}) and conclude that
$$\textup{char}_{S_i}\left(H^1_{\FFc}(\QQ,{T_i})/S_i\cdot\pi_i(\kappa_1)\right)=\textup{char}_{S_i}\left(H^1_{\FFc^*}(\QQ,{T_i}^*)^{\vee}\right) \cdot \pi_i(h)\,.$$
This shows (applied with $i=0$) along with (\ref{eqn:sharpbasecase}) that $\pi_0(h)\in S_0^\times$ and therefore that $h\in R^\times$, as desired.
\end{proof}

 \subsection{Modular Galois representations and the universal Kolyvagin system}
 \label{subsec:example1}
 Let $N$ be a positive integer which is prime to $p$ and suppose $p\geq 5$. Let $\omega$ denote the mod $p$ cyclotomic character (of $G_\QQ$), which we view both as a $p$-adic
and complex character by fixing embeddings $\overline{\QQ}
\hookrightarrow \overline{\QQ}_p$, $\overline{\QQ} \hookrightarrow
\mathbb{C}$, as well as a Dirichlet character modulo $Np$. Let
$$f=\sum_{n=1}^{\infty} a_nq^n \in S_k(\Gamma_0(Np),{\omega}^j)$$
be a normalized  cuspidal elliptic modular newform of even weight $k$, which is an
eigenform for the Hecke operators $T_\ell$ for $\ell\nmid Np$ and
$U_\ell$ for $\ell \mid Np$. Let $E/\QQ_p$ be a finite extension
that contains $a_n$ for all $n$ and let $\oo=\oo_E$ be its ring of integers and $\pi=\pi_E$ a fixed uniformizer. 
Let
$\rho_f: G_\QQ \ra \textup{GL}_2(E)$ be the Galois representation
attached to $f$ by Deligne.  Throughout this subsection,
we assume the following hypothesis holds true:

\emph{The semi-simple residual representation $\overline{\rho}_f$ associated to $\rho_f$ is absolutely
irreducible.}

As explained in \cite[12.2.2]{nek}, the Galois representation $\rho_f$ admits a self-dual twist ${\rho}$, which is Deligne's Galois representation associated to a twisted cusp form $\widetilde{f}$ of weight $k$ with trivial central character.  We assume in addition that the tame level $N(\widetilde{f})$ of $\widetilde{f}$ is square-free. In view of \cite[Lemma 12.3.10]{nek} this amounts to saying that the local automorphic representation $\pi(\widetilde{f})_\ell$ at primes $\ell \mid N(\widetilde{f})$ are \emph{twisted Steinberg} (in the sense of \cite[12.3.6.2]{nek}).

Let us also choose the prime $p$ sufficiently large so that Weston's theorem that we have referred to in Remark~\ref{example:weston} applies for $\widetilde{f}$ (and its residual representation $\overline{\rho}$). In Sections~\ref{subsec:example1} and \ref{subsec:example2}, we shall study the deformations of this mod $p$ representation $\overline{\rho}$. This seems to be the technically most involved case as compared to the case of non-critical twists; for example, the hypotheses (\textbf{H.Tam}) and (\textbf{H.nA}) are easier to verify in these cases. 
 
 \begin{define} \label{def:TTmodular} 
 Let $R$ denote the universal deformation ring of $\overline{\rho}$. Note that $R$ is isomorphic to  $\oo[[X_1,X_2,X_3]]$ by Weston's theorem. Let $\pmb{\rho}$ the universal deformation of $\overline{\rho}$ and let $\TT$ the deformation space (a free $R$-module of rank two) on which $G_\QQ$ acts by $\pmb{\rho}$.
  \end{define}
   It is easy to see that we have $\chi(\overline{T})=1$ for the residual representation of $\TT$.
 \begin{example} 
\label{example:ellipticcurve}
When the eigenform $f$ has weight $k=2$ and has trivial central character, then $f=\widetilde{f}$ corresponds to an elliptic curve $E_{/\QQ}$ without CM that has split-multiplicative reduction at all primes $\ell$ dividing its conductor $N_E$. In this case, $\overline{T}=E[p]$ is the $p$-torsion subgroup of $E(\overline{\QQ})$ and $\overline{\rho}_E: G_\QQ \ra \textup{Aut}(\overline{T})=\textup{GL}_2(\mathbb{F}_p)$
is the mod $p$ Galois representation attached to $E$. The Weil-pairing shows that $\chi(\overline{T})=1$. 

Assuming the hypotheses (\textbf{F1})-(\textbf{F3}), we conclude as in Example~\ref{example:flach} that the deformation problem in this situation is unobstructed. Let $R\cong  \ZZ_p[[X_1,X_2,X_3]]$ be the universal deformation ring, $\pmb{\rho}_E$ be the universal deformation of $\overline{\rho}_E$ and $\TT$ be the free rank-two $R$-module on which $G_\QQ$ acts by $\pmb{\rho}_E$. 
\end{example}
\subsubsection{The hypotheses of Section~\ref{subsubsec:setup}} 
\label{subsub:contentsofhypo}
We explain the contents of the hypotheses on the main technical results presented in the previous section, when applied to our current situation. 

For $\TT$ as in Definition~\ref{def:TTmodular}, observe that (\textbf{H1}), (\textbf{H3}) and (\textbf{H4}) are verified automatically. The hypothesis (\textbf{H2}) holds true also for all large enough $p$ thanks to \cite{serreimage} when $f$ is as in Example~\ref{example:ellipticcurve}, otherwise thanks to \cite{ribet85}. We assume that $\TT$ verifies (\textbf{H2}), as well as (\textbf{\textup{H.Tam}}) and (\textbf{\textup{H.nA}}). We discuss the last two hypotheses in Proposition~\ref{prop:tamimpliestam} and Remark~\ref{rem:nonanamolous} below. 

\begin{prop}
\label{prop:tamimpliestam}
For $\widetilde{f}$ and $\TT$ as in Definition~\ref{def:TTmodular}, assume that
\begin{itemize}
\item $p$ does not divide the Tamagawa number $c_\ell(\widetilde{f})$ at $\ell$ (defined as in \cite[I.4.2.2]{FoPR}),
\item $p$ does not divide $\ell-1$.
\end{itemize}
Then $($\textbf{\textup{H.Tam}}$)$
\,holds true for $\TT$.
\end{prop}
\begin{proof}
We will only provide the details for the case when $f=\widetilde{f}$ is of weight two; the general case may be handled in a similar manner. Let $E$ denote the elliptic curve attached to $f$ and  $T=T_p(E)$ be the $p$-adic Tate module of $E$. Under the running assumptions, there is a non-split exact sequence
  \be\label{eqn:Two} 0\lra\ZZ_p(1) \lra T \lra \ZZ_p \lra 0\ee
  of $\ZZ_p[[G_{\ell}]]$-modules. Let $\sigma=\partial (1)\in H^1(\QQ_\ell,\ZZ_p(1))$ where $\partial: \ZZ_p \ra H^1(\QQ_\ell,\ZZ_p(1))$ is the connecting homomorphism in the long exact sequence of the $G_{\QQ_\ell}$-cohomology of the sequence (\ref{eqn:Two}). Kummer theory gives an isomorphism
   \be\label{eqn:ord}\textup{ord}_\ell: H^1(\QQ_\ell,\ZZ_p(1))\stackrel{\sim}{\lra} \QQ_\ell^{\times,\wedge}\stackrel{\sim}{\lra} \ZZ_p.\ee
 According to~\cite{cartaneilenberg} pp. 290 and 292, $-\sigma$ is the extension class
of the sequence (\ref{eqn:Two}) inside
$\textup{Ext}^1_{\ZZ_p[G_{\QQ_\ell}]}(\ZZ_p,\ZZ_p(1))=H^1(\QQ_\ell,\ZZ_p(1))$. Hence
$\textup{ord}_\ell(\sigma)\neq 0$ as the sequence (\ref{eqn:Two})
is non-split and by \cite[Prop. 3.3]{kbbheegner} it further follows  that $\textup{ord}_\ell(\sigma) \in \ZZ_p^\times$ and therefore $\partial$ is surjective.

We have the following diagram below with exact rows and commutative squares
$$\xymatrix{&&H^0(\QQ_\ell,T)\ar[d]\ar[r]&\ZZ_p\ar@{->>}[d]\ar@{->>}[r]^(.30){\partial}&H^1(\QQ_\ell,\ZZ_p(1))\cong \ZZ_p\ar@{->>}[d] \\
0\ar[r]&H^0(\QQ_\ell,\pmb{\mu}_p)\ar[r]&H^0(\QQ_\ell,\overline{T})\ar[r]&\ZZ/p\ZZ\ar[r]^(.30){\bar{\partial}}&H^1(\QQ_\ell,\pmb{\mu}_p)\cong \ZZ/p\ZZ
}$$
 This shows that the map $\bar{\partial}$ is surjective as well and hence
$$H^0(\QQ_\ell,\overline{T})\cong H^0(\QQ_\ell,\pmb{\mu}_p)=0$$
as we assumed $p\nmid \ell-1$.
\end{proof}
\begin{rem}
\label{rem:nonanamolous}
In the situation of Example~\ref{example:ellipticcurve}, the hypothesis (\textbf{H.nA}) translates into the requirement that $p$ is \emph{non-anomalous} for $E$ (in the sense of \cite{mazur-anom}). Given an elliptic curve $E_{/{\QQ}}$\,, Mazur in loc.cit. explains that anomalous primes should be scarce. For example, a lemma due to R. Greenberg shows that if $E(\QQ)$ has a point of order $2$, then any prime $p>5$ at which $E$ has good reduction is non-anomalous. 

In the case of modular forms of higher weight, it is easy to see that the hypothesis (\textbf{H.nA}) holds true equally often.
\end{rem}
\subsubsection{Interpolation of Beilinson-Kato Kolyvagin systems}
\label{subsec:interpolateBK}
Suppose $\TT$ is as in Definition~\ref{def:TTmodular} for which the hypotheses (\textbf{H1})-(\textbf{H4}), (\textbf{H.Tam}) and (\textbf{H.nA}) simultaneously hold true.

Let $g$ be any elliptic newform of weight $\omega\geq 2$ and let
$$\rho_g : G_\QQ \lra \textup{GL}_2(\oo_g)$$
be the Galois representation attached to $g$ by Deligne with coefficients in the ring of integers $\oo_g$ of a finite extension $\Phi_g$ of of $\QQ_p$. Let $T_g$ be the free $\oo_g$ module of rank $2$ on which $G_\QQ$ acts via $\rho_g$. Let $\mm_g$ denote the maximal ideal of $\oo_g$ and let $\overline{\rho}_g$ the residual representation of $\overline{\rho}_g$ $\mod \mm_g$. Suppose that $\overline{\rho}_g\cong\overline{\rho}$ so that $\rho_g$ is a deformation of $\overline{\rho}$ to the ring $\oo_g$. We thus have a ring homomorphism $\varphi_g: R\ra \oo_g$ that induces and isomorphism $T_g\cong \TT\otimes_{\varphi_g} \oo_f$, and by functoriality a commutative diagram
\be\label{eqn:diagramKSdeform}
\xymatrix{\overline{\textbf{KS}}(\TT,\FFc,\PP)\ar[rr]^{\varphi_g}\ar@{->>}[dr]&&\overline{\textbf{KS}}(T_g,\FFc,\PP)\ar[dl]\\
&{\textbf{KS}}(\overline{T},\FFc,\PP)&
} \ee
Until the end of this section, we let $0\neq \pmb{\kappa} \in \overline{\textbf{KS}}(\TT,\FFc,\PP)$ denote a universal big Kolyvagin system and we let $\varphi_g(\pmb{\kappa})$ be its image in $\overline{\textbf{KS}}(T_g,\FFc,\PP)$. Let $\pmb{\kappa}_g^{\textup{BK}}\in \overline{\textbf{KS}}(T_g,\FFc,\PP)$ be the Kolyvagin system obtained from the Beilinson-Kato Euler system attached to the modular form $g$ (as in \cite[Theorem 3.2.4]{mr02}).

\begin{thm}[Interpolation]
\label{thm:interpolation}
There is a $\lambda_g \in \oo_g$ such that
$$\lambda_g\cdot \varphi_g(\pmb{\kappa})=\pmb{\kappa}_g^{\textup{BK}}.$$
\end{thm}
\begin{proof}[Proof of Theorem~\ref{thm:interpolation}]
Let $\bar{\pmb{\kappa}}$ be the image of $\pmb{\kappa}$ in ${\textbf{KS}}(\overline{T},\FFc,\PP)$. By Theorem~\ref{thm:KSmain} it follows that $\bar{\pmb{\kappa}}\neq 0$, so it follows by \cite[Theorem 5.2.10(ii)]{mr02} and the commutative diagram (\ref{eqn:diagramKSdeform}) that $\varphi_g(\pmb{\kappa})$ generates the free $\oo_g$-module $\overline{\textbf{KS}}(T_g,\FFc,\PP)$ of rank one.
\end{proof}
\begin{rem}\label{rem:lamdasareunits}
Theorem~\ref{thm:interpolation} states that the \emph{improvements} (by the factors $\lambda_g$) of the Beilinson-Kato Kolyvagin systems  interpolate to give rise to the big Kolyvagin system, rather than the Beilinson-Kato Kolyvagin systems themselves. The Kolyvagin system $\varphi_g(\pmb{\kappa})$ is called an \emph{improvement} to $\pmb{\kappa}_g^{\textup{BK}}$ as the bound (on the relevant Selmer group) obtained using $\varphi_g(\pmb{\kappa})$ improves that obtained using $\pmb{\kappa}_g^{\textup{BK}}=\lambda_g\cdot \varphi_g(\pmb{\kappa})$ by a factor of $\lambda_g$. In particular, when the Kolyvagin system $\pmb{\kappa}_g^{\textup{BK}}$ is itself primitive (in the sense of \cite[Definition 4.5.5]{mr02}, see also Corollary 5.2.13(ii) and Theorem 5.3.10(iii) in loc.cit.) we have $\lambda_g\in \oo_f^\times$. It follows from \cite{skinnerurbanmainconj} that this is indeed the case in a variety of cases.
\end{rem}
\begin{rem}
\label{rem:questionDoeslambdainterpolate}
The most interesting question regarding the interpolation factors $\{\lambda_g\}$ is the following: Does there exist a global regular function $\lambda$ on the universal deformation space $\textup{spec} R$ that takes the value $\lambda_g$ at the $\oo_g$-valued modular point $\varphi_g$, for every modular form $g$ as above? In Sections~\ref{subsec:example2} and \ref{subsub:generalex} below, we tackle (and partially answer) somewhat modest versions (in the sense that they concern smaller deformation spaces) of this question. It is also worth reminding the reader that an affirmative answer to this question has powerful arithmetic consequences, as we have pointed out in Remark~\ref{rem:lamdaareunits}.
\end{rem}
\begin{rem}
\label{rem:universalpadicLfunc} In Theorem~\ref{thm:sharpening} above we explained how a universal Kolyvagin system $\pmb{\kappa}$ produces bounds on the corresponding Selmer groups. Note that:
\begin{itemize}
\item[(i)] The Kolyvagin system $\pmb{\kappa}_g^{\textup{BK}}$ is related to a special value of $L$-function attached to $g$, by the work of Kato~\cite[\S14]{ka1}; 
\item[(ii)] The (classical) modular points are Zariski dense in $\textup{Spec}(R)$ thanks to the main result of \cite{bockleajm2001}.
\end{itemize}
Thus the leading term $\kappa_1$ not only controls the dual Selmer group via Theorem~\ref{thm:sharpening}, it also very much resembles what might be thought of as a ($3$-variable) $p$-adic $L$-function on the universal deformation space. We shall elaborate on this point in Sections \ref{subsec:example2} and \ref{subsub:generalex} below and verify that one may recover the $2$-variable $p$-adic $L$-functions associated to various pieces of the universal deformation space from $\kappa_1$.
\end{rem}

\begin{prop}
\label{prop:BKleadingtermnonzero}
The leading term $\kappa_1$ of universal Kolyvagin system $\pmb{\kappa}$ for $\TT$ is non-zero.
 \end{prop}
\begin{proof}
This follows from \cite[Theorem 12.5(2)]{ka1} considering the canonical map  $$\overline{\mathbf{KS}}(\TT,\FFc,\PP)\stackrel{\varphi_{\LL}}{\lra} \overline{\mathbf{KS}}(T_{\widetilde{f}}\otimes\LL,\FFc,\PP)$$
that is induced from the ring homomorphism $\varphi_{\LL}:R\ra \oo[[\Gamma]]$ (which is deduced from the existence of the deformation $T_{\widetilde{f}}\otimes\LL$).

\end{proof}
  \subsection{Universal Kolyvagin system on Hida's nearly-ordinary deformation space}
 \label{subsec:example2}
 Suppose $f$ and $\TT$ are as in Section~\ref{subsec:interpolateBK}, verifying the hypotheses (\textbf{H1})-(\textbf{H4}), (\textbf{H.Tam}) and (\textbf{H.nA}) simultaneously. We further assume in this section that $f$ is $p$-ordinary and $p$-distinguished. Let $\Gamma^{\textup{w}}=1+p\ZZ_p$. Identify $\Delta=(\ZZ/p\ZZ)^\times$ by $\pmb{\mu}_{p-1}$ via the Teichm\"uller character $\omega$ so that we have $\ZZ_p^\times\cong \Delta\times\Gamma^{\textup{w}}.$ Set $\LL^{\textup{w}}=\ZZ_p[[\Gamma^{\textup{w}}]]$. Let $\frak{h}^{\textup{ord}}$ Hida's
universal ordinary Hecke algebra parametrizing Hida family passing through $f$, which is finite flat over $\LL^{\textup{w}}$ by~\cite[Theorem
1.1]{hida}. We will recall some basic properties of
$\frak{h}^{\textup{ord}}$, for details the reader may
consult~\cite{hida, hidafamily} and \cite[\S2]{emertonpollackweston}
for a survey. 

The eigenform $f$ corresponds to an \emph{arithmetic specialization} $\varphi_f: \frak{h}^{\textup{ord}}\ra \oo\,.$
Decompose $\frak{h}^{\textup{ord}}$ into a direct sum of its
completions at maximal ideals and let
$\frak{h}^{\textup{ord}}_{\mm}$ be the (unique) summand through
which $\varphi_f$ factors. The localization of
$\frak{h}^{\textup{ord}}$ at $\ker(\varphi_f)$ is a discrete
valuation ring~\cite[\S12.7.5]{nek}, and hence there is a unique
minimal prime $\frak{a} \subset \frak{h}^{\textup{ord}}_{\mm}$ such
that $\varphi_f$ factors through the integral domain $\mathcal{R}=\frak{h}^{\textup{ord}}_{\mm}/\frak{a}.$ The $\LL^{\textup{w}}$-algebra $\mathcal{R}$ is called the branch of the Hida
family on which $f$ lives, by duality it corresponds to a family $\mathbb{F}$ of ordinary modular forms. Hida in~\cite{hidafamily} gives a
construction of a  big $G_{\QQ}$-representation $\mathcal{T}$ with
coefficients in $\mathcal{R}$. It follows from  \cite{wiles,taylorwiles} that the ring $\RR$ is Gorenstein of dimension two and that $\mathcal{T}$ is a free $\mathcal{R}$-module of rank two. We set $\frak{R}=\RR\otimes_{\ZZ_p}\LL$ and $\frak{T}=\mathcal{T}\otimes_{\ZZ_p}\LL$, where $\LL=\ZZ_p[[\textup{Gal}(\QQ_\infty/\QQ)]]$ is the cyclotomic Iwasawa algebra as usual and $G_\QQ$ acts on $\frak{T}$ diagonally.  The representation $\frak{T}$ is what Ochiai calls the the \emph{universal ordinary deformation} of $\overline{T}$.

Let $\pmb{\kappa}$ be a fixed universal Kolyvagin system and let $\pmb{\kappa}^{\textup{n.o.}}$ denote its image under the natural map $\overline{\mathbf{KS}}(\TT,\FFc,\PP) \ra \overline{\mathbf{KS}}(\frak{T},\FFc,\PP)$ induced by the universality of $\TT$. We call $\pmb{\kappa}^{\textup{n.o.}}$ the \emph{nearly-ordinary universal Kolyvagin system}. Note that it follows from Theorem~\ref{thm:KSmain} that the $\frak{R}$-module $\overline{\mathbf{KS}}(\frak{T},\FFc,\PP)$ is free of rank one and it is generated by the nearly-ordinary universal Kolyvagin system. This latter fact essentially recovers a result due to Ochiai~\cite{ochiaideform}, where he interpolates Beilinson-Kato Euler systems\footnote{whereas we carry this out in the level of Kolyvagin systems.} along the ordinary locus of the universal deformation space. For any ordinary eigenform $g$ that lives in the branch $\RR$ of the Hida family, let
$$\varphi_g: \RR\lra \oo_g$$
denote the corresponding arithmetic specialization and $T_g=\mathcal{T}\otimes_{\varphi_g}\oo_g$ the associated Galois representation, where $\oo_g$ is the integers of a finite extension of $\QQ_p$. Let
$$\pmb{\kappa}^{\textup{BK}}\in \overline{\mathbf{KS}}(T_g\otimes\LL,\FFc,\PP)$$
be the $\LL$-adic Kolyvagin system for to the cyclotomic deformation $T_g\otimes \LL$ obtained from the Beilinson-Kato Euler system as in \cite[\S6.2]{mr02}. We shall later specify a normalization of the Beilinson-Kato elements befitting our needs; however the following holds true regardless of such choice.  
\begin{thm}
\label{thm:ochiaiKS}
Let $\overline{\mathbf{KS}}(\frak{T},\FFc,\PP)\stackrel{\varphi_g}{\lra} \overline{\mathbf{KS}}(T_g\otimes\LL,\FFc,\PP)$
denote the map induced from the arithmetic specialization $\varphi_g$ above by functoriality. Then there is a $\lambda_g \in \oo_g[[\Gamma]]$ such that
$$\lambda_g\cdot\varphi_g(\pmb{\kappa}^{\textup{n.o.}})=\pmb{\kappa}^{\textup{BK}}.$$
\end{thm}
\begin{proof}
Identical to the proof of Theorem~\ref{thm:interpolation}.
\end{proof}
\begin{cor}
\label{cor:nonvanishingkappa1}
Let  $\kappa_1^{\textup{n.o.}} \in H^1(\QQ,\frak{T})$ denote the leading term of the nearly-ordinary universal Kolyvagin system. Then ${\kappa}_1^{\textup{n.o.}} $ is non-vanishing.
\end{cor}
\begin{proof}
For $g$ as above, Kato proved that ${\kappa}^{\textup{BK}}_1 \in H^1(\QQ,T_g\otimes\LL)$ is non-zero. Corollary follows from Theorem~\ref{thm:ochiaiKS}.
\end{proof}
Note that  one may recover the well-known results of \cite{ochiaitwovarMC} when Theorem~\ref{thm:ochiaiKS} and Corollary~\ref{cor:nonvanishingkappa1} plugged in the Kolyvagin system machinery (Theorem~\ref{thm:sharpening}). One novelty resulting from our approach is that the Kolyvagin systems interpolated along the nearly-ordinary locus are simply obtained by \emph{restriction} from the universal deformation space. 

In what follows, we shall discuss the questions raised in Remarks~\ref{rem:questionDoeslambdainterpolate} and \ref{rem:universalpadicLfunc} over the nearly-ordinary locus $\textup{Spec}(\frak{R})$ of the deformation space. In particular, we show in Theorem~\ref{thm:univKSvsOchiaiKS} that the factors $\lambda_g$ of Theorem~\ref{thm:ochiaiKS} interpolate (as $g$ moves along the Hida family). This allows us to compare the leading term $\kappa_1^{\textup{n.o.}} $ of the nearly-ordinary universal Kolyvagin system to Ochiai's \emph{two-variable} (optimal) Beilinson-Kato element and deduce the desired variational results.

We remark that under our running hypotheses on $\overline{T}$, the representation $\frak{T}$ admits a $G_{\QQ_p}$-stable submodule $F_p^+\,\frak{T}\subset \frak{T}$ that is saturated and free of rank one (as an $\frak{R}$-module). Set $F_p^-\frak{T}=\frak{T}/F_p^+\frak{T}$ and let $\textup{loc}_p^s$ denote the \emph{singular projection}
$$\textup{loc}_p^s: H^1(\QQ,\frak{T})\lra \frac{H^1(\QQ_p,\frak{T})}{H^1(\QQ_p,F_p^+\frak{T})}=:H^1_s(\QQ_p,\frak{T})\,.$$
We also set $F_p^\pm\overline{T}:=F_p^\pm\frak{T}\otimes_{\frak{R}} k$.

Let $\mathcal{L}^{\textup{Kit}}(\mathbb{F}) \in \frak{R}$ denote Kitagawa's two-variable $p$-adic $L$-function associated to the Hida family $\mathbb{F}$, which is normalized in accordance with \cite[Sec. 6.3]{ochiaitwovarMC} (and denoted by $L_{p,b}^{\textup{Ki}}(\mathcal{T})$ in loc.cit). The main theorem of \cite{ochiaiajm} equips us with a \emph{Hida-theoretic Coleman map}
$$\Xi_d: H^1_s(\QQ_p,\frak{T}) \lra \frak{R}$$
where $d$ is a basis of ${D}_{\textup{cris}}(F_p^+\frak{T})$. The map $\Xi_d$ is injective, has pseudo-null cokernel and it verifies 
$$\Xi_d\left(\textup{loc}_p^s(\mathcal{Z}_1^{\textup{BKO}})\right)=\mathcal{L}^{\textup{Kit}}(\mathbb{F})\,.$$
Here $\mathcal{Z}_1^{\textup{BKO}} \in H^1(\QQ,\frak{T})$ is the initial term of Ochiai's optimized Beilinson-Kato Euler system for $\frak{T}$ (denoted by $\mathcal{Z}^{\textup{Ki}}_{b,d}(1)$ in \cite{ochiaitwovarMC}). 
\begin{define}
\label{def:formnonexceptionalmodp}
We say that the cuspidal newform $f=\sum a_nq^n$ is \emph{exceptional mod $p$} if $p-1 \mid \frac{k+j}{2}-1$ and $a_p \equiv 1 (\textup{mod\,} \varpi)$. 
\end{define}
\begin{prop}
\label{prop:nonexceptionalimpliesfree}Suppose that we are in the setting of Theorem~\ref{thm:ochiaiKS} and assume that the form $f$ is non-exceptional mod $p$. Then the natural injection $H^1_s(\QQ_p,\frak{T}) \hookrightarrow H^1(\QQ_p,F_p^-\,\frak{T})$ is an isomorphism of free $\frak{R}$-modules of rank one. Furthermore, $\Xi_d$ is an isomorphism as well.
\end{prop}
\begin{proof}
Let $\alpha_{f,p}$ denote the unramified character of $G_{\QQ_p}$ sending the arithmetic Frobenius to $a_p$. The running assumptions imply that the $k$-valued character $\beta_{f,p}:=(\alpha_{f,p}\cdot\omega^{1-\frac{k+j}{2}} \mod \varpi)$ is non-trivial. It follows from \cite[Proposition 2.4.1]{how2} that $G_{\QQ_p}$ acts on the one-dimensional $k$-vector space $F_p^-\overline{{T}}$ by the character $\beta_{f,p}$ and we conclude that $H^0({\QQ_p},F_p^-\overline{T})=0$. This shows by local duality and Nakayama's lemma that $H^2(\QQ_p, F_p^+\frak{T})=0$ (recall that we are working with the self-dual twist of the Hida family) and we conclude with the proof that the injection $H^1_s(\QQ_p,\frak{T}) \hookrightarrow H^1(\QQ_p,F_p^-\,\frak{T})$ is a surjection as well. It also follows from our hypothesis (\textbf{H.nA}) that $H^0({\QQ_p},F_p^+\overline{T})=0$ and by local duality along with Nakayama's lemma that $H^2(\QQ_p,F_p^-\frak{T})=0$. Similarly we may prove that $H^2(\QQ_p,F_p^-\frak{T}^{D})=0$ where $\frak{T}^D:=\textup{Hom}_\frak{R}(\frak{T},\frak{R})(1)$. As explained in \cite[Remark 2.8]{kbbesrankr}, we may now deduce that the $\frak{R}$-module $H^1(\QQ_p,F_p^-\frak{T})$ is free, since the ring $\frak{R}$ is Gorenstein. On the other hand, as $\Xi_d$ injects $H^1_s(\QQ_p,\frak{T})$ into a cyclic $\frak{R}$-module, $H^1(\QQ_p,F_p^-\frak{T})$ is in fact free of rank one and the first part of the proposition is now proved.

Note now that $\Xi_d\left(H^1_s(\QQ_p,\frak{T})\right)$ is a free submodule of $\frak{R}$ and Ochiai has proved that the quotient $\frak{R}/\Xi_d\left(H^1_s(\QQ_p,\frak{T})\right)$ is pseudo-null and therefore trivial, verifying the second assertion.
\end{proof}

\begin{thm}
\label{thm:univKSvsOchiaiKS}Suppose that we are in the setting of Proposition~\ref{prop:nonexceptionalimpliesfree}. There exists an element $\lambda=\lambda(\mathbb{F}) \in \frak{R}$ such that for every $g$ as above we have $\lambda_g=\lambda(g)$, where $\lambda(g)$ denotes the image $\varphi_g(\lambda)$ of $\lambda$ under the specialization map $\varphi_g:\frak{R}\ra \oo_g[[\Gamma]]$.
\end{thm}
\begin{proof}
Set $\frak{f}=\Xi_d\left(\textup{loc}_p^s(\kappa_1^{\textup{n.o.}})\right)$ and $\frak{g}=\Xi_d\left(\textup{loc}_p^s(\mathcal{Z}_1^{\textup{BKO}})\right)$. The assertion amounts to the statement that $\frak{f}\mid\frak{g}$, which is what we verify now. If $f \nmid g$, one may use \cite[Section 6.3]{fouquetcompositio} to find an $S$-valued point $\pi_S:\frak{R}\ra S$ of $\textup{Spec}(\frak{R})$ (where $S$ is a discrete evaluation ring) with the following properties:
\begin{itemize}
\item $\pi_S(f) \neq 0$, 
\item $\pi_S(f) \nmid \pi_S(g)$\,.
\end{itemize}
Set $T_S=\frak{T}\otimes_{\pi_S}S$ and let $\pmb{\kappa}^S=\pi_S(\pmb{\kappa}^{\textup{n.o.}}) \in \overline{\mathbf{KS}}(T_S,\FFc,\PP)$. Let $\frak{m}_S$ denote the maximal ideal of $S$ and $\overline{T}_S=\frak{T}\otimes S/\frak{m}_S$. Using the fact that the reduction of $\pmb{\kappa}^{\textup{n.o.}}$ modulo the maximal ideal of $\frak{R}$ is non-vanishing, it follows that the image of  $\pmb{\kappa}^S$ in ${\mathbf{KS}}(\overline{T}_S,\FFc,\PP)$ is non-trivial as well. Theorem 5.2.10 of \cite{mr02} now shows that $\pmb{\kappa}^S$ generates the module $\overline{\mathbf{KS}}(T_S,\FFc,\PP)$.  The Beilinson-Kato-Ochiai Euler system for $\frak{T}$ gives rise to an Euler system for $T_S$. Let 
$$\pmb{\kappa}^{\textup{O,S}} \in \textup{im}\left(\textup{ES}(T_S)\lra \overline{\mathbf{KS}}(T_S,\FFc,\PP)\right)$$ denote the image of this Euler system under the Euler systems to Kolyvagin systems map of \cite{mr02}. In particular, we have ${\kappa}^{\textup{O,S}}_1=\pi_S(\mathcal{Z}_1^{\textup{BKO}}) \in H^1_{\FFc}(\QQ,T_S)$ for the initial term. The discussion above regarding $\pmb{\kappa}^{S}$ shows that there exists $\lambda_S\in S$ with $\pmb{\kappa}^{\textup{O,S}}=\lambda_S\cdot\pmb{\kappa}^{S}$, so that
\be\label{eqn:BKOtoNOKS}\pi_S(\mathcal{Z}_1^{\textup{BKO}})={\kappa}^{\textup{O,S}}_1=\lambda_S\cdot \kappa^S_1=\lambda_S\cdot\pi_S(\kappa^{\textup{n.o.}}_1)\ee
where $\kappa_1^S \in H^1_{\FFc}(\QQ,T_S)$ as usual stands for the leading term of the Kolyvagin system $\pmb{\kappa}^{S}$. Set $F^\pm T_S:=F^\pm\frak{T}\otimes_{\pi_S} S$. The following equalities lead to a contradiction with the choice of $S$ and conclude the proof of the theorem: 
\begin{align}
\label{eqn:eqn1}\pi_S(g)&=\pi_S\left(\textup{Fitt}\left(H^1(\QQ_p,F^-\frak{T})/\frak{R}\cdot\textup{loc}_p^s(\mathcal{Z}_1^{\textup{BKO}})\right)\right)\\
\label{eqn:eqn2}&=\lambda_S\cdot\textup{Fitt}\left(H^1(\QQ_p,F^-{T}_S)/S\cdot\pi_S\left(\textup{loc}_p^s(\kappa_1^{\textup{n.o.}})\right)\right)\\
\label{eqn:eqn3}&=\lambda_S\cdot \pi_S\left(\textup{Fitt}\left(H^1(\QQ_p,F^-\frak{T})/\frak{R}\cdot\textup{loc}_p^s(\kappa_1^{\textup{n.o.}})\right)\right)\\
\label{eqn:eqn4}&=\lambda_S\cdot \pi_S(f)\,.
\end{align}
We explain these equalities. The ingredients that go in the proof of Proposition~\ref{prop:nonexceptionalimpliesfree} show that  
\be\label{eqn:eqn5}H^1(\QQ_p,F^-\frak{T})\otimes_{\pi_S}S\stackrel{\sim}{\lra}H^1(\QQ_p,F^-{T}_S)\,.\ee
The equality (\ref{eqn:eqn3}) follows from (\ref{eqn:eqn5}) whereas (\ref{eqn:eqn2}) using (\ref{eqn:BKOtoNOKS}) and (\ref{eqn:eqn5}). The equalities (\ref{eqn:eqn1}) and (\ref{eqn:eqn4}) are immediate by Proposition~\ref{prop:nonexceptionalimpliesfree}.  
\end{proof}
\begin{cor}
\label{cor:nearlyordkappa1andkitagawapadicLfunc}
For $\lambda(\mathbb{F}) \in \frak{R}$ as in the statement of Theorem~\ref{thm:univKSvsOchiaiKS}, we have 
$$\lambda(\mathbb{F})\cdot\textup{char}_{\frak{R}}\left(H^1(\QQ_p,F_p^-\,\frak{T})/\frak{R}\cdot\textup{loc}_p^s(\kappa_1^{\textup{n.o.}})\right)=\frak{R}\cdot\mathcal{L}^{\textup{Kit}}(\mathbb{F})\,\,.$$
\end{cor}
\begin{proof}
This is immediate after Theorem~\ref{thm:univKSvsOchiaiKS} and the main theorem of \cite[Corollary 6.17]{ochiaitwovarMC}.
\end{proof}
Note that Skinner in \cite{skinnersplitcyclo} has recently made use of a variational argument (essentially encoded in the equivalence of (iii) and (iv) in Corollary~\ref{cor:twovarmainconjMTTmainconj} below) so as to deduce Mazur's main conjecture for a modular form at primes $p$ of multiplicative reduction, by moving to a form where $p$ is a good ordinary prime (and where the desired result is known thanks to \cite{skinnerurbanmainconj}). Although the equivalence of (iii) and (iv) in Corollary~\ref{cor:twovarmainconjMTTmainconj} is well-known, we still would like to record it here along with the view point offered by our universal Kolyvagin system with the hope that its variants (such as those discussed in Section~\ref{subsub:generalex} below) might be useful in other contexts where the technology used in Skinner's work is no longer available. 
\begin{cor}
\label{cor:twovarmainconjMTTmainconj}
The following assertions are equivalent:
\begin{itemize}
\item[(i)] $\lambda_g \in \oo_g[[\Gamma]]^\times$ for a single arithmetic member $g$ of the family $\mathbb{F}$.
\item[(ii)] $\lambda(\mathbb{F}) \in \frak{R}^\times$ and $\lambda_g \in \oo_g[[\Gamma]]^\times$ for every arithmetic member $g$ of the family $\mathbb{F}$.
\item[(iii)] Conjecture 2.4 (the two variable main conjecture) of \cite{ochiaitwovarMC} holds true.
\item[(iv)] The cyclotomic main conjecture holds true for a single (for every) eigenform $g$ in the family $\mathbb{F}$. 
\end{itemize}
\end{cor}
\begin{proof}
The implication (ii)$\implies$(iii) follows from Poitou-Tate Global duality, used along with the proof of Theorem~\ref{thm:sharpening} and Corollary~\ref{cor:nearlyordkappa1andkitagawapadicLfunc}. Assuming the truth of (iv), \cite[Theorem 3.23]{kbb} shows that $\pmb{\kappa}^{\textup{BK}}$ is a generator of the cyclic $\LL$-module $\overline{\mathbf{KS}}(T_g\otimes\LL,\FFc,\PP)$ and (i) follows. The remaining implications are clear.
\end{proof}
 \subsection{The sheaf of universal $p$-adic $L$-function on the eigencurve}
\label{subsub:generalex}

Before giving a proof of the main technical result of this article, we present in this section applications of our universal Kolyvagin system towards a \emph{main conjecture} on the eigencurve and a related discussion on variational problems (in the flavor of those that were treated over the nearly-ordinary locus in the previous section). 

Suppose we are in the setting of Section~\ref{subsec:example1}. In particular $\overline{\rho}$ is the residual representation of the critical-twist $\widetilde{f}$ of the normalized cuspidal elliptic modular newform $f$ of even weight. Assume in addition that $\widetilde{f}$ is of finite slope, namely that the action of $\varphi$ on the potentially semi-stable Diuedonn\'e module verifies ${D}_{\textup{pst}}(V_{\widetilde{f}})^{\varphi=\alpha}\neq0$ for some $\alpha$. Here $V_{\widetilde{f}}$ is as usual Deligne's Galois representation attached to $\widetilde{f}$.

\begin{define}
\label{def:pseudogeom}
Let $E/\QQ_p$ be a finite extension and let $\oo_E$ be its ring of integers. An $E$-valued \emph{pseudo-geometric specialization} is a ring homomorphism $\psi: R\ra \oo_E$ such that the $G_{\QQ_p}$-representation is $\TT\otimes_{\psi}E$ is potentially semi-stable with distinct Hodge-Tate weights. 
\end{define}

We recall some definitions we gave in the introduction. Let $X=\textup{Spec}(R)$ denote universal deformation space. Recall that we are working under the assumption (\textbf{H.nOb}) that the deformation problem is unobstructed. Let $\frak{X}$ denote the (Berthelot) generic fiber of $\textup{Spf}(R)$ and let $R^\dagger:=\Gamma(\frak{X},\oo_{\textup{Spf}R})$. 
Let $\LL_E=\oo_E[[\Gamma]]$ be the cyclotomic Iwasawa algebra, $\frak{I}_E$ be Berthelot's analytic generic fiber of $\textup{Spf}\,\LL$ and $\LL^\dagger_E=\Gamma(\frak{I}_E,\oo_{\textup{Spf}\,\LL_E})$. Note that our $\LL^\dagger_E$  is denoted by $\LL_\infty$ in \cite{pottharstcyclo}. 
\begin{rem}
\label{rem:fromdaggertopsi} Let $\psi^\dagger : R^\dagger \rightarrow E$ be an $E$-valued point. By continuity, we have an induced $\oo_E$-valued point $R\ra \oo_E$ that we denote by $\psi$. We say that $\psi^{\dagger}$ is pseudo-geometric if $\psi$ is in the sense of Definition~\ref{def:pseudogeom}. We denote the $G_{\QQ,\Sigma}$-representation $\TT\otimes_{\psi} E$ by $V_{\psi^\dagger}$. 
\end{rem} 
\begin{define}
\label{def:pstcalledfiniteslope}
A pseudo-geometric specialization $\psi^\dagger$ is called finite-slope if $D_{\textup{pst}}(V_{\psi^\dagger})^{\varphi=\alpha}$ is non-zero for some $\alpha$.
\end{define}
\begin{define}\label{def:cyclodeformpsi} Given $\psi^\dagger \in \frak{X}(E)$ and $\psi \in X(\oo_E)$ as above, the deformation $\left(\TT\otimes_{\psi}\oo_E\right)\otimes \LL_E$ of $\overline{\rho}$ to $\LL_E$ induces a ring homomorphism 
$\psi_\LL:R\ra\LL_E$, which in turn induces a map $\psi^\dagger_{\LL}: R^\dagger \ra \LL_E^\dagger$
on the analytic global sections. Set $\widetilde{V}_{\psi^\dagger}:=\TT^\dagger\otimes_{\psi^\dagger_\LL} \LL_E^\dagger$, the \emph{cyclotomic deformation} of $V_{\psi^\dagger}$.
\end{define}

Let $\mathcal{C}(\overline{\rho})$ denote the Coleman-Mazur $\overline{\rho}$-eigencurve which admits a Zariski-analytic dense subset whose $\mathbb{C}_p$-valued points correspond bijectively to $(g,\alpha)$ where $g$ is a cuspidal eigenform (of tame level $N$) of finite-slope whose residual Galois representation has the same semi-simplification as $\overline{\rho}$. Since $\overline{\rho}$ is fixed throughout this section, we will denote $\mathcal{C}(\overline{\rho})$ simply by ${\mathcal{C}}$. Let $\lambda \in \oo(\mathcal{C})^\times$ be the $U_p$-eigenvalue function and $\kappa \in \oo(\mathcal{C})$ the weight function. Finally, let $\TT_{\mathcal{C}}$ denote the pullback of the universal deformation, which is locally free coherent sheaf on $\mathcal{C}$ of rank $2$ which equipped with a continuous $\oo(\mathcal{C})$-linear Galois action.

\begin{rem}
\label{rem:finiteslopeclassical}
One may give an alternative description of the eigencurve using a theorem of Kisin~\cite{kisinFM}: $\mathcal{C}$ is the Zariski-analytic closure of $\mathcal{C}_{\textup{pst-fs}}$, where $\mathcal{C}_{\textup{pst-fs}} \subset \frak{X}\times\mathbb{G}_m$ is the set of points $r=(\psi^\dagger,\lambda(r))$ such that $\psi^\dagger$ is finite-slope with Hodge-Tate weights $0, \kappa(r)-1$ with $\kappa(r)\in \ZZ_{\geq 2}$. 
\end{rem}

\begin{define}
\label{define:senspace}
Let $P_{\mathcal{S}}(T)=T^2+aT+b \in \mathcal{O(\frak{X})}[T]$ denote the Sen polynomial and let $\frak{X}_0$ denote Sen's null locus, which is the closed subspace cut by $b=0$, which has codimension one everywhere. 
For a point $r=(x,\alpha) \in \mathcal{C}(\overline{\QQ}_p)$, the Sen polynomial of $\rho_r$ is $P_{\mathcal{S}}(x)=T^2+a(x)T+b(x)$. It follows from a theorem of Faltings and Jordan that if $r$ is a refined modular point (in the sense of Coleman and Mazur) then $b(x)=0$ and therefore that $\mathcal{C}\subset\frak{X}_0\times \mathbb{G}_m$. 
\end{define}

\begin{define}
\label{define:singularlocus}
We say that a point $r\in \mathcal{C}$ is $p$-\emph{exceptional} if $\kappa(r)=0$ and the corresponding Galois representation $V_r$ is crystalline at $p$ with $\dim D_{\textup{cris}}(V_r)^{\varphi=\lambda(r)}=2$. We say that $r$ is \emph{unsaturated} either 
\begin{itemize}
\item[(i)] $r$ is $p$-exceptional or; 
\item[(i)] $r$ is not $p$-exceptional, $\kappa(r)$ is a positive integer and $v_p(\lambda(r))>\kappa(r)$ or; 
\item[(ii)] $r$ is not $p$-exceptional, $V_r$ has a rank one sub-$G_{\QQ_p}$-representation $V_{r}^\prime$ which is crystalline with Hodge-Tate weight $\kappa(r)$,
\end{itemize}
and otherwise we call $r$ a saturated point.
\end{define}

Let $\mathcal{D}_{\textup{rig}}^{\dagger}(\TT_{\mathcal{C}})$ denote the sheafification of the $(\varphi,\Gamma)$-module functor on the weak $G$-topology of $\mathcal{C}$ given as in \cite[Definition 1.2.4]{rliu} and $\mathcal{C}_0$ be the set of saturated points.
\begin{thm}[Liu, Kedlaya-Pottharst-Xiao]
\label{thm:lkpx}
There is a coherent subsheaf $\mathcal{F}$  of the sheaf $\mathcal{D}_{\textup{rig}}^{\dagger}(\TT_{\mathcal{C}})$ which is locally free of rank one away from exceptional points, and saturated away from the unsaturated points of $\mathcal{C}$ and restricts on this locus to a triangulation of the $(\varphi,\Gamma)$-module $D_{\textup{rig}}^{\dagger}(V_r)$ associated to $V_r$.
\end{thm}
 \begin{define}
 \label{def:universalpadicLfunction}
The quasicoherent sheaf of \emph{universal $p$-adic $L$-function} is the invertible sheaf
$$\Xi_{\overline{\rho}}:=\textup{im}\left(R\cdot\textup{loc}_p(\kappa_1) \lra H^1_{{\psi}}(\mathcal{D}_{\textup{rig}}^\dagger(\TT_{\mathcal{C}}))/H^1_{\psi}(\mathcal{F})\right)\,.$$
Here $\kappa_1$ is the initial term of \emph{any} universal Kolyvagin system (the definition of $\Xi_{\overline{\rho}}$ clearly does not depend on the choice of this  Kolyvagin system), $ H^1_{\psi}(*)$ is the Iwasawa cohomology sheaf of \cite{kpx} (where $\psi$ is a left inverse of the Frobenius operator $\varphi$ in the context of $(\varphi,\Gamma)$-modules) and the arrow is obtained as follows: Sen's theory and Tate twisting yields a morphism $\tau:\, \textup{Sp}(\LL^\dagger)\times\mathcal{C}\ra \frak{X}\times\mathbb{G}_m$ which gives by pullback a map $\TT^\dagger\ra  \TT_{\mathcal{C}}\,\widehat{\otimes}\,\LL^\dagger$. The desired arrow comes using \cite[Theorem 1.9]{pottharst} and the identification of \cite[Corollary 4.4.11]{kpx}.
\end{define} 

 We would like to think of the sheaf $\Xi_{\overline{\rho}}$ as a generalization of Perrin-Riou's \cite{pr-ast} module of algebraic $p$-adic $L$-function, whose definition she gives for the cyclotomic deformation of a motive. The interpolation property we prove in Theorem~\ref{prop:singularkatointerpolation} partially justifies why  $\Xi_{\overline{\rho}}$ deserves to be called a ``$p$-adic $L$-function''.
 \begin{define}
 \label{def:pottharstscoleman} Let $g$ be a cuspidal new eigenform of weight $k$ which has finite-slope at $p$ and let $E/\QQ_p$ be a finite extension that contains all Fourier coefficients of $g$. Let $D$ denote the $(\varphi,\Gamma)$-module associated to $g$. Let $F\subset D$ be the rank one $(\varphi,\Gamma)$-submodule of $D$ chosen as in \cite[Section 5]{pottharstcyclo} (which is characterized by the property that its potentially semi-stable Dieudonn\'e module $F_{\textup{pst}}$ is spanned by $e_i$, where $\{e_i,e_{i^\prime}\}$ is a distinguished $E$-basis of $D_{\textup{pst}}$ described fully in loc.cit). Let $\{e_i^*,e_{i^\prime}^*\} \in D_{\textup{pst}}^*$ be the dual basis of $\{e_i,e_{i^\prime}\}$. Given these choices there exists a \emph{big logarithm} $H^1_{\textup{Iw}}(\QQ_p,D/F)\stackrel{\textup{Log}_{D/F}}{\lra} D_{\textup{pcris}}\otimes\LL^\dagger$, where $D_{\textup{pcris}}$ is the potentially crystalline Dieudonn\'e module (c.f. \cite[Section 4]{pottharstcyclo}) and the cohomology group is Pottharst's (analytic) Iwasawa cohomology group. This map composed with $e_{i^\prime}^*$ yields a homomorphism
 $$\mathcal{L}_{D/F}=e_{i^\prime}^*\circ \textup{Log}_{D/F}: H^1_{\textup{Iw}}(\QQ_p,D/F)\lra \LL^\dagger_E\,.$$ 
 \end{define} 
 \begin{define}
\label{define:gammafactors}
For $g$ as above, we define following \cite{pr, pottharstcyclo} 
$$\Gamma_{j}=\ell_1^{-1}\ell_2^{-1}\cdots\ell_{j}^{-1} \in \textup{Frac}(\LL^\dagger)$$
where $\displaystyle{\ell_i=\frac{\log(\gamma\,\chi_{\textup{cyc}}(\gamma)^{-i})}{\log(\chi_{\textup{cyc}}(\gamma))}}$. We also set $\delta_i=\textup{char}_{\LL^\dagger}\, H^i_{\textup{Iw}}(\QQ_p,D/F)_{\textup{tors}}$ and $\delta_g=\delta_1^{-1}\delta_2$. As explained in \cite[Section 5]{pottharstcyclo}, the ideals $\delta_i$ are non-trivial only when $g$ has weight $2$, has semistable reduction at $p$  and the restriction of the central character to $G_{\QQ_p}$ factors through $\Gamma$. 
\end{define}
 
\begin{define}
\label{define:restrictionoftriangulation}
For a saturated point $x \in  \mathcal{C}_{\textup{cl-fs}}$, we let $F_x\subset D_x$ denote the triangulation of the associated $(\varphi,\Gamma)$-module. Notice that this is the restriction of the global triangulation $\mathcal{F}\subset \mathcal{D}_{\textup{rig}}^\dagger(\TT_{\mathcal{C}})$ by Theorem~\ref{thm:lkpx}. 
\end{define}
  \begin{thm}
  \label{prop:singularkatointerpolation}
  For an $E$-valued saturated point $x=(\psi^\dagger,\lambda(x)) \in  \mathcal{C}_{\textup{cl-fs}}$ $($where $E$ is a finite extension of $\QQ_p$$)$ let $f_{x}$ denote the corresponding (classical) eigenform with a distinguished $U_p$-eigenvalue $\lambda(x)$. For some $\frak{t}_{x} \in \LL_E$\,, the following equality of invertible ideals of $\LL_E^\dagger$ holds true:
  $$\frak{t}_{x}\cdot\al_{D_{x}/F_{x}}\circ\psi^\dagger_\LL\left(\Xi_{\overline{\rho}}\right)=\Gamma_{\kappa(x)-1}\cdot\delta^{-1}_x\al_{p,\lambda(x)^c}(f_x^c)\cdot \LL^\dagger_E\,,$$
  where $\al_{p,\lambda(x)^c}(f_{\psi^\dagger}^c)$ is the $p$-adic $L$-function attached to the complex conjugate of $f_{x}$ for the $p$-stabilization determined by $\lambda(x)^c$.
  \end{thm}
  \begin{proof}
  For $\psi :R \ra \oo_E$ as in Remark~\ref{rem:fromdaggertopsi}, let $\widetilde{T}_\psi:=\TT\otimes_{\psi}\LL_E$ and let $\pmb{\kappa}^\psi \in \overline{\mathbf{KS}}(\widetilde{T}_\psi,\FFc,\PP)$ be the image of (any) universal Kolyvagin system $\pmb{\kappa}$.  It follows from Kato's explicit reciprocity law and the (unpublished) work of Kato-Kurihara-Tsuji that there is a Beilinson-Kato Kolyvagin system  $\pmb{\kappa}^{\textup{BK},\psi} \in \overline{\mathbf{KS}}(\widetilde{T}_\psi,\FFc,\PP)$ attached to the modular form $f_x$  whose initial term verifies 
$$\al_{D_{x}/F_{x}}(\textup{loc}_p({\kappa}^{\textup{BK},\psi}_1))=\Gamma_{\kappa(x)-1}\cdot\delta^{-1}_x\al_{p,\lambda(x)^c}(f_x^c)\cdot \LL^\dagger_E\,.$$  
As the $\LL$-module  $\overline{\mathbf{KS}}(\widetilde{T}_\psi,\FFc,\PP)$ is free of rank one and generated by $\pmb{\kappa}^\psi$, there is $\frak{t}_{x} \in \LL$ such that $\pmb{\kappa}^{\textup{Kato},\psi}=\frak{t}_{x}\cdot \pmb{\kappa}^{\psi}$. 

  \end{proof}
  Inspired by the conclusion of Corollary~\ref{cor:nearlyordkappa1andkitagawapadicLfunc} we raise the following two questions. It would be very desirable to have affirmative answers to these questions as they would have interesting consequences regarding the variational behavior of arithmetic data in families, similar to those over the nearly-ordinary locus indicated in Section~\ref{subsec:example2} above. 
\begin{question}
\label{question:txinterpolate1}
 Is it possible to interpolate $\frak{t}_x$ over an affinoid subdomain $\textup{Sp}(S)$ of the eigencurve that contains no unsaturated points to an element $\frak{t}_S \in S\,\widehat{\otimes}\,\LL^\dagger$?
 If yes, is it then possible to patch the $\frak{t}_S$ to a quasicoherent sheaf $\frak{T}$ over $\mathcal{C}_0$?
\end{question}
In relation to these questions, we propose the following Conjecture\footnote{As far as the author understands, much of this conjecture has been settled by David Hansen.} that should be of independent interest. Its first part predicts that the Coleman-Perrin-Riou maps interpolate over sufficiently nice affinoid subdomains $\textup{Sp}(A)$ of the eigencurve. Its second and third parts have to do with the behavior of the Beilinson-Kato elements over $\textup{Sp}(A)$, in line with the construction of the sheaf $\Xi_{\overline{\rho}}$ over $\mathcal{C}$. Theorem~\ref{thm:firstquestionhasaffrimativeanswerifconjholdstrue} below shows that first of the questions raised above holds true assuming the truth of this conjecture. See also Remark~\ref{rem:globalcolemantrivialization} regarding the second question.
\begin{conj}
\label{conj:bigColeman} Let $x_0\in \mathcal{C}_{\textup{cl-fs}}$ be a classical saturated point on the eigencurve. There exists an admissible affinoid subdomain neighborhood $\textup{Sp}(A)\subset\mathcal{C}$ of $x_0$ that contains no unsaturated points and an element $\alpha,\beta,\theta \in A$ with the following properties:
\begin{itemize}
\item[(i)] There exists a \emph{Coleman-trivialization} 
$$\frak{Log}_A:\,H^1_{{\psi}}(\mathcal{D}_{\textup{rig}}^\dagger(\TT_{A}))/H^1_{\psi}(\mathcal{F}(A)) \lra A\,\widehat{\otimes}\,\LL^\dagger$$
such that for every $E$-valued point $x \in  \mathcal{C}_{\textup{cl-fs}} \cap \textup{Sp}(A)$ the following diagram commutes:
$$\xymatrix{H^1_{{\psi}}(\mathcal{D}_{\textup{rig}}^\dagger(\TT_{A}))/H^1_{\psi}(\mathcal{F}(A)) \ar[d]\ar[r]^(.7){\frak{Log}_A}&A\,\widehat{\otimes}\,\LL^\dagger\ar[d]\\
H^1_{\textup{Iw}}(\QQ_p,D_x/F_x)\ar[r]^(.7){\alpha(x)\mathcal{L}_{D_x/F_x}}& \LL^\dagger_{E}}$$
\item[(ii)] There exist an element $\mathcal{Z}_1^{\textup{BK},A} \in H^1_{\textup{Iw}}(\QQ,\TT_A)$ (of Pottharst's analytic Iwasawa cohomology) that restricts for every $E$-vauled point $x=(\psi^\dagger,\lambda(x)) \in  \mathcal{C}_{\textup{cl-fs}} \cap \textup{Sp}(A)$ to $\beta(x)\cdot\kappa_1^{\BK,\psi}$ (where $\kappa_1^{\BK,\psi}$ is given as in the proof of Theorem~\ref{prop:singularkatointerpolation}).
 \item[(iii)] For every $E$-valued point $y=(\phi^\dagger,\lambda(y)) \in  \textup{Sp}(A)$, let $T_y$ denote a $G_\QQ$-stable $\oo_E$-lattice contained in the $E$-vector space $V_y$ and let $\mathcal{Z}_1^{\textup{BK},A}(y) \in H^1_{\textup{Iw}}(\QQ,V_y)$ denote the restriction of $\mathcal{Z}_1^{\textup{BK},A}$ to $y$. There exists an Euler system $\mathbf{c}^{\BK,y}$ for $T_y$ whose initial term $\mathbf{c}^{\BK,y}_\QQ \in H^1(\QQ,T_y\otimes\LL)$ along the cyclotomic tower identifies with $\mathcal{Z}_1^{\textup{BK},A}(y)$ under the natural analytification morphism.
\end{itemize}
\end{conj}

\begin{thm}
\label{thm:firstquestionhasaffrimativeanswerifconjholdstrue}
Let $x_0\in \mathcal{C}_{\textup{cl-fs}}$ be a classical saturated point on the eigencurve and suppose that there exists an admissible affinoid subdomain neighborhood $\textup{Sp}(A)\subset\mathcal{C}$ of $x_0$ that contains no unsaturated points and which verifies Conjecture~\ref{conj:bigColeman}. Then the first part of Question~\ref{question:txinterpolate1} has a positive answer on $A$.
\end{thm}

\begin{proof}
This follows from the existence of a two-variable $p$-adic $L$-function (\cite{bellaicheL}) and a slight modification of the arguments in the proof of Theorem~\ref{thm:univKSvsOchiaiKS} (where the corresponding assertions (over the ordinary locus) to those of Conjecture~\ref{conj:bigColeman} have already been verified by T. Ochiai).
\end{proof}
\begin{rem}
\label{rem:globalcolemantrivialization} The forthcoming work of Hansen we have referred to above also recovers Bella\"ice's two-variable $p$-adic $L$-function $\mathbb{L}_{p,\lambda}(A) \in A\,\widehat{\otimes}\,\LL^\dagger$ from the Beilinson-Kato elements. This together with Theorem~\ref{thm:firstquestionhasaffrimativeanswerifconjholdstrue} shows that the image of the sections $\Xi_{\overline{\rho}}(A)$ of the sheaf of universal $p$-adic $L$-function on $A$ under the Coleman-trivialization lies in the module generated by $\mathbb{L}_{p,\lambda}(A)$. The following two features of this picture is noteworthy:
\begin{itemize}
\item The sheaf $\Xi_{\overline{\rho}}$ is constructed out of an integral element $\kappa_1$. This fact should be useful in applications towards the proof of Main Conjectures over affinoid subdomains of the eigencurve (through the form of main conjecture we deduced in Section~\ref{subsec:weakleo} above over the full deformation space), which (given the technology of the day) would necessarily rely on the Euler-Kolyvagin system machinery that is only available at integral level. We note that one missing ingredient is a well-behaved descent/control mechanism for Selmer complexes in this context and we hope to address this matter in a future work.
\item The modules  $\Xi_{\overline{\rho}}(A)$ (which, by Hansen's work, are closely related to $\mathbb{L}_{p,\lambda}(A)$) readily patch together along the eigencurve.
\end{itemize}
\end{rem}
\section{Core vertices and deforming Kolyvagin systems}
\label{sec:coreverticesandKS}
Our goal in this section is to give a proof of Theorem~\ref{thm:KSmain} (aside from Theorem~\ref{thm:corevertexn} below, which will be proved later in Section~\ref{sec:coreverticesexist}), assuming (\textbf{H1})-(\textbf{H4}), (\textbf{H.Tam}) and (\textbf{H.nA}).

Let $\bar{\frak{n}}=(r,u,v,w)\in(\ZZ_{>0})^4$ and $\bar{\frak{s}}=(r,u,v)\in(\ZZ_{>0})^3$. Assume throughout this section that $\chi(\TT)=\chi(\frak{T})=1$. 
\subsection{Core vertices}
Suppose $T$ is one of $\TT_{\bar{\frak{n}}}$, $\frak{T}_{\bar{\frak{s}}}$ or $\frak{T}/\mm$ and $S$ is the corresponding quotient ring $R_{\bar{\frak{n}}}$, $\frak{R}_{\bar{\frak{s}}}$ or $\mathtt{k}$.

\begin{define}
\label{def:coreverticesforcan}
The integer $n \in \NN_{\bar{\frak{n}}}$ (resp., $n \in \NN_{\bar{\frak{s}}}$) is called a \emph{core vertex} for the Selmer structure $\FFc$ on $T$ if
\begin{itemize}
\item[(i)] $H^1_{\FFc(n)^*}(\QQ,T^*)=0$,
\item[(ii)] $H^1_{\FFc(n)}(\QQ,T)$ is a free $S$-module of rank one.
\end{itemize}
\end{define}

Suppose that the hypotheses (\textbf{H1})-(\textbf{H4}), (\textbf{H.Tam}), and (\textbf{H.nA}) hold true. The following theorem is fundamental in proving the existence of Kolyvagin systems.
\begin{thm}
\label{thm:corevertexn}
 Let $n\in\NN_{\bar{\frak{n}}}$ \textup{(}resp., $n \in \NN_{\bar{\frak{s}}}$$\textup{)}$  be a core vertex for the Selmer structure $\FFc$ on the residual representation $\overline{T}$. Then $n$ is a core vertex for the Selmer structure $\FFc$ on $T$ as well.
\end{thm}
Theorem~\ref{thm:corevertexn} is proved in \S\ref{sec:coreverticesexist}. In this section we show how it may be used to prove the existence of Kolyvagin systems for $\TT$ and $\frak{T}$, given as in Definitions~\ref{selmer sheaf} and \ref{def:KSbigrings}.

\begin{thm}
\label{thm:KSoverartinianringS}
The $S$-module  $\textbf{\textup{KS}}(T,\FFc,\PP)$ is free of rank one.
\end{thm}
Theorem~\ref{thm:KSoverartinianringS} is proved in two steps. As the first step, we prove:
\begin{thm}
\label{thm:KSsurjective}
Suppose $n\in\NN$ is any core vertex for the Selmer structure $\FFc$ on $T$. Then the natural map
$$\textbf{\textup{KS}}(T,\FFc,\PP) \lra H^1_{\FFc(n)}(\QQ,T)\otimes G_n$$
$($given by $\pmb{\kappa}\mapsto \kappa_n$$)$ is surjective.
\end{thm}
\begin{proof}The arguments utilized in the proof of \cite[Theorem 3.11]{kbb} may be modified without much difficulty in order to prove Theorem~\ref{thm:KSsurjective}. The main point is that we have Theorem~\ref{thm:corevertexn} here in place of \cite[Theorem 2.27]{kbb}.
\end{proof}
Define a subgraph $\XX^{0}=\XX^{0}(\PP)$ of $\XX$ whose vertices are the core vertices of $\XX$ and whose edges are defined as follows: We join $n$ and $n\ell$ by an edge in $\XX^{0}$ if and only if the localization map
$$H^1_{\FFc(n)}(\QQ,\overline{T}) \lra H^1_{f}(\QQ_{\ell},\overline{T})$$
is \emph{non-zero}. We define the sheaf $\mathcal{H}^0$ on $\XX^{0}$ as the restriction of the Selmer sheaf $\mathcal{H}$ to $\XX^{0}$.

\begin{lemma}
\label{lem:x0connected}
The graph $\XX^{0}$ is connected.
\end{lemma}
\begin{proof}
The edges of $\XX^0$ are defined in terms of $\overline{T}$ (and not $T$ itself) so the arguments that go into the proof of \cite[Theorem 4.3.12]{mr02} apply.
\end{proof}

The following Theorem, combined with Theorem~\ref{thm:KSsurjective} completes the proof of Theorem~\ref{thm:KSoverartinianringS}.
\begin{thm}
\label{thm:KSinjective}
Suppose $n\in\NN$ is any core vertex for the Selmer structure $\FFc$ on $T$. Then the natural map
$$\textbf{\textup{KS}}(T,\FFc,\PP) \lra H^1_{\FFc(n)}(\QQ,T)\otimes G_n$$
 is is injective.
\end{thm}
\begin{proof}
One may argue as in the proof of~\cite[Theorem 3.12]{kbb} in order to prove this assertion. The essential input is the fact that the graph $\XX^0$ is connected (Lemma~\ref{lem:x0connected}).
\end{proof}
\begin{lemma}
\label{lem:j}
Fix a quadruple $\bar{\frak{n}}$ 
as above. For any $\bar{\frak{i}} \succ \bar{\frak{n}}$,
the natural restriction map
$$\mathbf{KS}(\TT_{\bar{\frak{n}}},\FFc, \PP_{\bar{\frak{n}}}) \lra \mathbf{KS}(\TT_{\bar{\frak{n}}},\FFc, \PP_{\bar{\frak{i}}})$$
 is an isomorphism.
\end{lemma}
\begin{proof}
Theorems~\ref{thm:KSsurjective} and \ref{thm:KSinjective} applied with a core vertex $n \in \NN_{\bar{\frak{i}}}$, we have isomorphisms
$$\mathbf{KS}(\TT_{\bar{\frak{n}}},\FFc, \PP_{\bar{\frak{n}}}) \stackrel{\sim}{\lra} H^1_{\FFc(n)}(\QQ,\TT_{\bar{\frak{n}}}) \stackrel{\sim}{\longleftarrow}\mathbf{KS}(\TT_{\bar{\frak{n}}}, \FFc,\PP_{\bar{\frak{i}}})$$
 compatible with the restriction map
 $$\mathbf{KS}(\TT_{\bar{\frak{n}}}, \FFc,\PP_{\bar{\frak{n}}}) \lra \mathbf{KS}(\TT_{\bar{\frak{n}}},\FFc, \PP_{\bar{\frak{i}}}).$$
 Note that $n\in \NN_{\bar{\frak{i}}}$ as above exists by  \cite[Corollary 4.1.9]{mr02} and Theorem~\ref{thm:corevertexn} above.
 \end{proof}

\begin{lemma}
\label{lem:surj0}
Let $\bar{\frak{n}}^{\prime}\prec \bar{\frak{n}}$ and let $n \in \NN_{\bar{\frak{n}}}$ be a core vertex. The map
$$H^1_{\FFc(n)}(\QQ,\TT_{\bar{\frak{n}}})\lra  H^1_{\FFc(n)}(\QQ,\TT_{\bar{\frak{n}}^\prime})$$
 is surjective.
\end{lemma}
\begin{proof}
We verify the assertion of the Lemma for $\bar{\frak{n}}^\prime=(r,u,v,w)$ and $\bar{\frak{n}}=(r+1,u,v,w)$. The proof of the general case follows by applying this argument (or where necessary, its slightly modified form) repeatedly.

 We have the following commutative diagram, where the vertical isomorphism is obtained from (a slight variation of) Proposition~\ref{prop:upperbound}(iv) below:
 $$\xymatrix @C=.9in @R=.4in{
H^1_{\FFc(n)}(\QQ,\TT_{r+1,u,v,w}) \ar[r]^(.55){\textup{reduction}} \ar[rd]^(.52){\varpi}&  H^1_{\FFc(n)}(\QQ,\TT_{r,u,v,w})\ar[d]^(.45){[\varpi]}_{\cong}\\
& H^1_{\FFc(n)}(\QQ,T_{r+1,u,v,w})[\varpi^{r}]
}$$
Since $n\in \NN_{\bar{\frak{n}}}$ is a core vertex (and therefore $H^1_{\FFc(n)}(\QQ,\TT_{\bar{\frak{n}}})$ is a free $R_{\bar{\frak{n}}}$-module of rank one), the map on the diagonal is surjective. This proves that the horizontal map is surjective as well.
\end{proof}
\begin{lemma}
\label{lem:surj}
The map
$$\mathbf{KS}(\TT_{\bar{\frak{n}}},\FFc, \PP_{\bar{\frak{n}}}) \lra \mathbf{KS}(\TT_{\bar{\frak{n}}^\prime},\FFc, \PP_{\bar{\frak{n}}})$$ is surjective for $\bar{\frak{n}}^{\prime}\prec \bar{\frak{n}}$.
\end{lemma}
\begin{proof}
 By Theorem~\ref{thm:KSsurjective}, Theorem~\ref{thm:KSinjective} and Lemma~\ref{lem:j} applied with a core vertex $n\in \NN_{\bar{\frak{n}}}$ to both $\TT_{\bar{\frak{n}}}$ and $\TT_{\bar{\frak{n}}^\prime}$, we obtain the following commutative diagram with vertical isomorphisms:
$$\xymatrix{
\mathbf{KS}(\TT_{\bar{\frak{n}}},\FFc, \PP_{\bar{\frak{n}}})  \ar[r]\ar[d]_{\cong}& \mathbf{KS}(\TT_{\bar{\frak{n}}^\prime},\FFc, \PP_{\bar{\frak{n}}}) \ar[d]^{\cong}\\
 H^1_{\FFc(n)}(\QQ,\TT_{\bar{\frak{n}}}) \otimes G_{n} \ar @{->>}[r] & H^1_{\FFc(n)}(\QQ,\TT_{\bar{\frak{n}}^\prime}) \otimes G_{n}
}$$
where the surjection in the second row is Lemma~\ref{lem:surj0}. It follows at once that the upper horizontal map in the diagram is surjective as well.
\end{proof}

\begin{proof}[Proof of Theorem~\ref{thm:KSmain}]
Lemma~\ref{lem:j} shows that
$$\varinjlim_{\bar{\frak{i}}} \mathbf{KS}(\TT_{\bar{\frak{n}}}, \FFc,\PP_{\bar{\frak{i}}})=\mathbf{KS}(\TT_{\bar{\frak{n}}},\FFc, \PP_{\bar{\frak{n}}}).$$
 The proof of (i) now follows by Theorem~\ref{thm:KSoverartinianringS} and Lemma~\ref{lem:surj}. (ii) is proved similarly, by appropriately modifying the ingredients that go into the proof of (i).
\end{proof}

\subsection{The existence of core vertices}
\label{sec:coreverticesexist}
In this subsection we verify the truth of Theorem~\ref{thm:corevertexn}.
\subsubsection{Cartesian properties}
\label{subsec:cartesianprop}
Let $\textup{Quot}(\TT)$ denote the collection $\{\TT_{\bar{\frak{n}}}: \bar{\frak{n}} \in (\ZZ^+)^4\}$ of quotients of $\TT$ and similarly, let $\textup{Quot}(\frak{T})=\{\frak{T}_{\bar{\frak{m}}}: \bar{\frak{m}} \in (\ZZ^+)^3\}\cup\{\frak{T}/\mm\}$. 

Given $\bar{\alpha}=(\alpha_1,\cdots,\alpha_d)\in (\ZZ^+)^d$ and $1\leq i \leq d$, we define $\bar{\alpha}_{+,i} \in (\ZZ^+)^d$ to be the $d$-tuple whose $j$th coordinate is $\alpha_i+\delta_{ij}$. Here $\delta_{ij}$ is Kronecker's delta.
\begin{define}
\label{iwasawa-cartesian}
A local condition $\FF$ at a prime $\ell$ is said to be \emph{cartesian} on the collection $\textup{Quot}(\TT)$ if it satisfies the following conditions:

For $\bar{\alpha} \in (\ZZ^+)^4$,
\begin{enumerate}
\item[\textbf{(C1)}] (Weak functoriality) 
\begin{itemize}
\item[\textbf{(C1.a)}] $H^1_{\FF}(\QQ_{\ell},\TT_{\bar{\alpha}})$ is the exact image of $H^1_{\FF}(\QQ_{\ell},\TT_{\bar{\alpha}_{+,1}})$ under the canonical map
    $$H^1(\QQ_{\ell},\TT_{\bar{\alpha}_{+,1}}) \lra H^1(\QQ_{\ell},\TT_{\bar{\alpha}}),$$
\item[\textbf{(C1.b)}] For $i\in \{2,3,4\}$,  $H^1_{\FF}(\QQ_{\ell},\TT_{\bar{\alpha}})$ lies in the image of  $H^1_{\FF}(\QQ_{\ell},\TT_{\bar{\alpha}_{+.i}})$ under the map
    $$H^1(\QQ_{\ell},\TT_{\bar{\alpha}_{+.i}}) \lra H^1(\QQ_{\ell},\TT_{\bar{\alpha}}).$$

    \end{itemize}
\item[\textbf{(C2)}] (Cartesian property) For $1\leq i \leq 4$,
$$\displaystyle H^1_{\FF}(\QQ_{\ell},\TT_{\bar{\alpha}})=\ker\left(H^1(\QQ_{\ell},\TT_{\bar{\alpha}})\lra \frac{H^1(\QQ_{\ell},\TT_{\bar{\alpha}_{+,i}})}{H^1_{\FF}(\QQ_{\ell},\TT_{\bar{\alpha}_{+,i}})}\right).$$
   If $i>1$, the arrow here is induced from the injection $\TT_{\bar{\alpha}}\stackrel{[X_i]}{\lra}\TT_{\bar{\alpha}_{+,i}}$, where $[X_i]$ stands for multiplication by $X_i$.
    When $i=1$, the arrow is induced from the injection $\TT_{\bar{\alpha}} \stackrel{[\varpi]}{\lra} \TT_{\bar{\alpha}_{+,1}}.$
\end{enumerate}
\end{define}

\begin{define}
\label{iwasawa-cartesian}
A local condition $\FF$ at a prime $\ell$ is said to be \emph{cartesian} on the collection $\textup{Quot}(\frak{T})$ if it satisfies the following conditions:

For $\bar{\alpha} \in (\ZZ^+)^3$,
\begin{enumerate}
\item[\textbf{(D1.a)}] $H^1_{\FF}(\QQ_{\ell},\frak{T}_{\bar{\alpha}})$ is the exact image of $H^1_{\FF}(\QQ_{\ell},\frak{T}_{\bar{\alpha}_{+,1}})$ under the canonical map
    $$H^1(\QQ_{\ell},\frak{T}_{\bar{\alpha}_{+,1}}) \lra H^1(\QQ_{\ell},\frak{T}_{\bar{\alpha}}),$$
\item[\textbf{(D1.b)}] For $i\in \{2,3\}$,  $H^1_{\FF}(\QQ_{\ell},\frak{T}_{\bar{\alpha}})$ lies in the image of  $H^1_{\FF}(\QQ_{\ell},\frak{T}_{\bar{\alpha}_{+.i}})$ under the map
    $$H^1(\QQ_{\ell},\frak{T}_{\bar{\alpha}_{+.i}}) \lra H^1(\QQ_{\ell},\frak{T}_{\bar{\alpha}}).$$
\item[\textbf{(D2)}]  For $1\leq i \leq 3$,
$$\displaystyle H^1_{\FF}(\QQ_{\ell},\frak{T}_{\bar{\alpha}})=\ker\left(H^1(\QQ_{\ell},\frak{T}_{\bar{\alpha}})\lra \frac{H^1(\QQ_{\ell},\frak{T}_{\bar{\alpha}_{+,i}})}{H^1_{\FF}(\QQ_{\ell},\frak{T}_{\bar{\alpha}_{+,i}})}\right).$$

\item[\textbf{(D3)}]$\displaystyle H^1_{\FF}(\QQ_{\ell},\frak{T}/\mm)=\ker\left(H^1(\QQ_{\ell},\frak{T}/\mm)\lra \frac{H^1(\QQ_{\ell},\frak{T}_{1,1,1})}{H^1_{\FF}(\QQ_{\ell},\frak{T}_{1,1,1})}\right),$ where the arrow is induced from the injection $\frak{R}/\mm=\mathtt{k}\stackrel{\sim}{\ra}\RR_0[\mm]\hookrightarrow \RR_0.$
\end{enumerate}
\end{define}
\subsubsection{Cartesian properties at $p$}
\begin{prop}
\label{propcartatptt} Assuming $($\textbf{\upshape{H.nA}}$)$, the local condition at $p$ given by $\FFc$ on the collection $\textup{Quot}(\TT)$ $($resp., on the collection $\textup{Quot}(\frak{T})$$)$ is cartesian.
\label{prop:cartatp}
\begin{proof}
This is obvious thanks to Proposition~\ref{prop:loccondatp}.
\end{proof}
\end{prop}

\subsubsection{Cartesian properties primes $\ell \neq p$ for the coefficient ring $R$}
Throughout this section the hypothesis (\textbf{H.Tam}) is in force.
\begin{lemma}
\label{lem:comparefiniteunr}$\,$
\begin{itemize}
\item[(i)] $H^1_{f}(\QQ_\ell,\mathcal{A}_{u,v,w})= H^1_{\textup{ur}}(\QQ_\ell,\mathcal{A}_{u,v,w}).$
\item[(ii)] $H^1_{f}(\QQ_\ell,\TT_{u,v,w})= H^1_{\textup{ur}}(\QQ_\ell,\TT_{u,v,w}).$
\item[(iii)]$H^1_{\FFc}(\QQ_\ell,\TT_{r,u,v,w})= H^1_{\textup{ur}}(\QQ_\ell,\TT_{r,u,v,w}).$
\item[(iv)] The following sequence is exact:
$$\xymatrix{
0\ar[r]& H^1_{\FFc}(\QQ_\ell,\TT_{r,u,v,w})\ar[r]& H^1(\QQ_\ell,\TT_{r,u,v,w})\ar[r]&\frac{H^1(\QQ_\ell,\mathcal{A}_{u,v,w})}{H^1_{f}(\QQ_\ell,\mathcal{A}_{u,v,w})}
}$$
\end{itemize}
\end{lemma}
\begin{proof}
(iv) follows from (\ref{eqn:caneqfin}).

By \cite[Lemma 1.3.5]{r00} we have the following two exact sequences:
\be\label{eqn:seqA}0\lra H^1_{f}(\QQ_\ell,\mathcal{A}_{u,v,w}) \lra H^1_{\textup{ur}}(\QQ_\ell,\mathcal{A}_{u,v,w})\lra \mathcal{W}/(\textup{Fr}_\ell-1)\mathcal{W}\lra 0\ee
\be\label{eqn:seqT}0\lra H^1_{\textup{ur}}(\QQ_\ell,\TT_{u,v,w})\lra H^1_{f}(\QQ_\ell,\TT_{u,v,w}) \lra  \mathcal{W}^{\textup{Fr}_\ell=1}\lra0\ee
where $\mathcal{W}=\mathcal{A}_{u,v,w}^{I_v}/(\mathcal{A}_{u,v,w}^{I_\ell})_{\textup{div}}$. In Lemma~\ref{lem:technicalhelp}, we check under the assumption (\textbf{H.Tam}) that $\mathcal{W}^{\textup{Fr}_\ell=1}=0$. Since $\mathcal{W}$ is a finite module, the exact sequence
$$0\lra\mathcal{W}^{\textup{Fr}_\ell=1}\lra \mathcal{W} \stackrel{\textup{Fr}_\ell-1}{\lra} \mathcal{W} \lra \mathcal{W}/(\textup{Fr}_\ell-1)\mathcal{W}$$
shows that $\mathcal{W}/(\textup{Fr}_\ell-1)\mathcal{W}=0$. This proves that
\be\label{eqn:fequr} H^1_{f}(\QQ_\ell,\mathcal{A}_{u,v,w})= H^1_{\textup{ur}}(\QQ_\ell,\mathcal{A}_{u,v,w}) \,\,\,,\,\,\, H^1_{\textup{ur}}(\QQ_\ell,\TT_{u,v,w})= H^1_{f}(\QQ_\ell,\TT_{u,v,w}).\ee
This proves (i) and (ii). By (\ref{eqn:caneqfin}) and (\ref{eqn:fequr}) it now follows that
$$ H^1_{\FFc}(\QQ_\ell,\TT_{r,u,v,w})=\textup{im}\left(H^1_{\textup{ur}}(\QQ_\ell,\TT_{u,v,w})\ra H^1(\QQ_\ell,\TT_{r,u,v,w})\right)\subset H^1_{\textup{ur}}(\QQ_\ell,\TT_{k, u,v,w})$$
$$H^1_{\FFc}(\QQ_\ell,\TT_{r,u,v,w})=\ker\left(H^1(\QQ_\ell,\TT_{r,u,v,w}) \lra  \frac{H^1(\QQ_\ell,\mathcal{A}_{u,v,w})}{H^1_{\textup{ur}}(\QQ_\ell,\mathcal{A}_{u,v,w})}\right)\supset H^1_{\textup{ur}}(\QQ_\ell,\TT_{k, u,v,w})$$
and the proof of (iii) follows.
\end{proof}
\begin{lemma}
\label{lem:technicalhelp}
$\mathcal{W}^{\textup{Fr}_\ell=1}=0$.
\end{lemma}
\begin{proof}
As we have $\mathcal{A}_{1,1,1}[\varpi]=\overline{T}$, it follows that
$H^0(\QQ_\ell,\mathcal{A}_{1,1,1}[\varpi])=0$ since we assume \textbf{H.Tam}, hence also that
$$H^0(\QQ_\ell,\mathcal{A}_{1,1,1})=0.$$
Using the $G_{\QQ_\ell}$-cohomology of the exact sequences
$$0\lra \mathcal{A}_{1,1,w}\stackrel{[X_3]}{\lra} \mathcal{A}_{1,1,w+1}\lra \mathcal{A}_{1,1,1}\lra0
$$
$$0\lra \mathcal{A}_{1,v,w}\stackrel{[X_2]}{\lra} \mathcal{A}_{1,v+1,w}\lra \mathcal{A}_{1,1,w}\lra0,
$$
$$0\lra \mathcal{A}_{u,v,w}\stackrel{[X_1]}{\lra} \mathcal{A}_{u+1,v,w}\lra \mathcal{A}_{1,v,w}\lra0,
$$

it follows by induction that
\be\label{eqn:vanishingzero}H^0(\QQ_\ell,\mathcal{A}_{u,v,w})=0.\ee
Taking the $G_{\QQ_\ell}/I_\ell$-invariance of the short exact sequence
$$0\lra(\mathcal{A}_{u,v,w}^{I_\ell})_{\textup{div}}\lra\mathcal{A}_{u,v,w}^{I_\ell}\lra \mathcal{W}\lra0$$
we see by (\ref{eqn:vanishingzero}) that
$$\mathcal{W}^{\textup{Fr}_\ell=1}\hookrightarrow H^1(G_{\QQ_\ell}/I_\ell,(\mathcal{A}_{u,v,w}^{I_\ell})_{\textup{div}})\cong (\mathcal{A}_{u,v,w}^{I_\ell})_{\textup{div}}/(\textup{Fr}_\ell-1).$$
To conclude with the proof, it therefore suffices to show that
$$(\mathcal{A}_{u,v,w}^{I_\ell})_{\textup{div}}/(\textup{Fr}_\ell-1)=0.$$
For any $\alpha \in \ZZ^+$, (\ref{eqn:vanishingzero}) shows
\be\label{eqn:usefulvanish}H^0(G_{\QQ_\ell}/I_\ell, (\mathcal{A}_{u,v,w}^{I_\ell})_{\textup{div}}[\varpi^\alpha])=0.\ee
 The exact sequence
$$\left((\mathcal{A}_{u,v,w}^{I_\ell})_{\textup{div}}[\varpi^\alpha]\right)^{\textup{Fr}_\ell=1}\ra (\mathcal{A}_{u,v,w}^{I_\ell})_{\textup{div}}[\varpi^\alpha] \stackrel{\textup{Fr}_\ell-1}{\lra}(\mathcal{A}_{u,v,w}^{I_\ell})_{\textup{div}}[\varpi^\alpha]\ra \mathcal{A}_{u,v,w}^{I_\ell}[\varpi^\alpha]/(\textup{Fr}_\ell-1)\ra 0$$
and (\ref{eqn:usefulvanish}) shows that $\mathcal{A}_{u,v,w}^{I_\ell}[\varpi^\alpha]/(\textup{Fr}_\ell-1)=0$. Passing to direct limit the Lemma follows.

\end{proof}
By Lemma~\ref{lem:comparefiniteunr}(iv) we have the following commutative diagram with exact rows:
$$\xymatrix{0\ar[r]& H^1_{\FFc}(\QQ_\ell,\TT_{r,u,v,w})\ar[r]\ar[d]& H^1(\QQ_\ell,\TT_{r,u,v,w})\ar[r]\ar[d]&\frac{H^1(\QQ_\ell,\mathcal{A}_{u,v,w})}{H^1_{f}(\QQ_\ell,\mathcal{A}_{u,v,w})}\ar[d]^{\alpha}\\
0\ar[r]& H^1_{\FFc}(\QQ_\ell,\TT_{r,u+1,v,w})\ar[r]& H^1(\QQ_\ell,\TT_{r,u+1,v,w})\ar[r]&\frac{H^1(\QQ_\ell,\mathcal{A}_{u+1,v,w})}{H^1_{f}(\QQ_\ell,\mathcal{A}_{u+1,v,w})} }$$
\begin{lemma}
\label{lem:alphainj}
The map $\alpha$ is injective if
$$\beta:  H^1(I_\ell,\mathcal{A}_{u,v,w})^{\textup{Fr}_\ell=1}\lra  H^1(I_\ell,\mathcal{A}_{u+1,v,w})^{\textup{Fr}_\ell=1}$$
is injective.
\end{lemma}
\begin{proof}
This follows from the commutative diagram
$$\xymatrix{0\lra \frac{H^1(\QQ_\ell,\mathcal{A}_{u,v,w})}{H^1_f(\QQ_\ell,\mathcal{A}_{u,v,w})} \ar[r]\ar[d]_\alpha & H^1(I_\ell,\mathcal{A}_{u,v,w})^{\textup{Fr}_\ell=1}\ar[d]^\beta\\
0\lra \frac{H^1(\QQ_\ell,\mathcal{A}_{u+1,v,w})}{H^1_f(\QQ_\ell,\mathcal{A}_{u+1,v,w})}\ar[r] & H^1(I_\ell,\mathcal{A}_{u+1,v,w})^{\textup{Fr}_\ell=1}}$$
whose exact rows come from the Hochschild-Serre spectral sequence and the fact that
$$H^1_f(\QQ_\ell,\mathcal{A}_{u,v,w})=H^1_{\textup{ur}}(\QQ_\ell,\mathcal{A}_{u,v,w}):=\ker\left(H^1(\QQ_\ell,\mathcal{A}_{u,v,w})\lra H^1(I_\ell,\mathcal{A}_{u,v,w})^{\textup{Fr}_\ell=1}\right),$$
where the first equality is Lemma~\ref{lem:comparefiniteunr}(i).
\end{proof}
Consider the short exact sequence
$$0\lra \mathcal{A}_{u,v,w}\stackrel{[X_1]}{\lra} \mathcal{A}_{u+1,v,w}\lra \mathcal{A}_{1,v,w}\lra0
$$
The $I_\ell$-cohomology of this sequence gives
\be\label{eqn:reductionseq1}
0\lra \mathcal{A}_{u+1,v,w}^{I_\ell}/\mathcal{A}_{u,v,w}^{I_\ell}\lra \mathcal{A}_{1,v,w}^{I_\ell}\lra H^1(I_\ell,\mathcal{A}_{u,v,w}) \lra  H^1(I_\ell,\mathcal{A}_{u+1,v,w})\ee
To ease the notation set
$$\mathcal{K}_{v,w}=\mathcal{A}_{u+1,v,w}^{I_\ell}/\mathcal{A}_{u,v,w}^{I_\ell},$$
so that the sequence (\ref{eqn:reductionseq1}) may be rewritten as
\be\label{eqn:reductionseq2}
0\lra \mathcal{A}_{1,v,w}^{I_\ell}/\mathcal{K}_{v,w}\lra H^1(I_\ell,\mathcal{A}_{u,v,w}) \lra  H^1(I_\ell,\mathcal{A}_{u+1,v,w})
\ee
Taking $G_{\QQ_\ell}/I_\ell$-invariance in (\ref{eqn:reductionseq2}), we conclude that
\begin{lemma}$\ker(\beta)\cong H^0(G_{\QQ_\ell}/I_\ell, \mathcal{A}_{1,v,w}^{I_\ell}/\mathcal{K}_{v,w})$.
\end{lemma}
\begin{lemma}
\label{lem:redatell1}
Under the assumption that \textup{(\textbf{H.Tam})} holds true,
\begin{itemize}
\item[(i)] $H^0(\QQ_\ell,\mathcal{A}_{1,v,w})=0,$
\item[(ii)] $H^0(\QQ_\ell,\mathcal{K}_{v,w})=0.$
\end{itemize}
\end{lemma}
\begin{proof}
Noting that $\overline{T}\cong \mathcal{A}_{1,1,1}$, Hypothesis (\textbf{H.Tam}) shows that
$$H^0(\QQ_\ell,\mathcal{A}_{1,1,1}[\varpi])=0$$
and also that $H^0(\QQ_\ell,\mathcal{A}_{1,1,1})=0$.
The $G_{\QQ_\ell}$-invariance of the sequence
$$0\lra \mathcal{A}_{1,1,w-1}\stackrel{[X_3]}{\lra} \mathcal{A}_{1,1,w} \lra \mathcal{A}_{1,1,1}\lra 0$$
shows by induction that $H^0(\QQ_\ell,\mathcal{A}_{1,1,w})=0$ for all $w\in \ZZ_{\geq2}$.
Using similarly the exact sequence
$$0\lra \mathcal{A}_{1,v-1,w}\stackrel{[X_2]}{\lra} \mathcal{A}_{1,v,w}\lra \mathcal{A}_{1,1,w}\lra0$$
we conclude with the proof of (i). (ii) follows from (i) as $\mathcal{K}_{v,w}$ is a submodule of $\mathcal{A}_{1,v,w}$.
\end{proof}
\begin{prop}
\label{prop:betainj}
$\ker(\beta)=0$.
\end{prop}
\begin{proof}
Taking the $G_{\QQ_\ell}/I_\ell$-invariance of the short exact sequence
$$0\lra \mathcal{K}_{v,w}\lra \mathcal{A}_{1,v,w}^{I_\ell} \lra \mathcal{A}_{1,v,w}^{I_\ell}/ \mathcal{K}_{v,w}\lra 0,$$
we conclude using Lemma~\ref{lem:redatell1} that
\be\label{eqn:controlkerbeta}\ker(\beta) \hookrightarrow H^1(G_{\QQ_\ell}/I_\ell,\mathcal{K}_{v,w})\cong \mathcal{K}_{v,w}/(\textup{Fr}_\ell-1) \mathcal{K}_{v,w}.\ee
Lemma~\ref{lem:redatell1}(ii) yields (using the fact that $ \mathcal{K}_{v,w}$ is $\varpi^\infty$-torsion) an exact sequence
 $$0\lra \mathcal{K}_{v,w}[\varpi^\alpha] \stackrel{\textup{Fr}_\ell-1}{\lra}  \mathcal{K}_{v,w}[\varpi^\alpha]\lra  \mathcal{K}_{v,w}[\varpi^\alpha]/(\textup{Fr}_\ell-1)\lra 0$$
for every $\alpha\in\ZZ^+$. Noting that the module $ \mathcal{K}_{v,w}[\varpi^\alpha]$ has finite cardinality, it follows now that $$\mathcal{K}_{v,w}[\varpi^\alpha]/(\textup{Fr}_\ell-1)=0.$$
 Passing to direct limit, Proposition follows by (\ref{eqn:controlkerbeta}).
\end{proof}
\begin{prop}
\label{prop:cartatelltt}
The local condition at a prime $\ell\neq p$, given by $\FFc$ on the collection $\textup{Quot}(\TT)$ is cartesian.
\end{prop}
\begin{proof}
(\textbf{C1.a}) holds true by definition and (\textbf{C1.b}) by Lemma~\ref{lem:comparefiniteunr}(iii). The parts of (\textbf{C2}) concerning the cases $2\leq i\leq 4$ follow from Lemma~\ref{lem:alphainj} and Proposition~\ref{prop:betainj}; the part concerning the case $i=1$ from \cite[Lemma 3.7.1]{mr02}.
\end{proof}
\subsubsection{Cartesian properties at $\ell \neq p$ for the coefficient ring $\frak{R}$}
Assume throughout this section that (\textbf{H.Tam}) holds true. Recall that $\frak{R}=\mathcal{R}[[\Gamma]]$ where $\mathcal{R}$ is a Gorenstein $\oo$-algebra of dimension $2$ with maximal ideal $\mm_{\RR}$.
\begin{lemma}
\label{lem:cartpfrakprepare1}$\,$
\begin{itemize}
\item[(i)] $H^1_{\FFc}(\QQ_p,\frak{T}_{r,u,v})=H^1_{\textup{ur}}(\QQ_p,\frak{T}_{r,u,v})$.
\item[(ii)] The sequence
$$0\lra H^1_{\FFc}(\QQ_\ell,\frak{T}_{r,u,v})\lra H^1(\QQ_\ell,\frak{T}_{r,u,v})\lra \frac{ H^1(\QQ_\ell,\frak{A}_{u,v})}{ H^1_{\textup{ur}}(\QQ_\ell,\frak{A}_{u,v})}$$
is exact.
\end{itemize}
\end{lemma}
\begin{proof}
The proof of Lemma~\ref{lem:comparefiniteunr} above works verbatim.
\end{proof}
\begin{lemma}
\label{lem:cartpfrakprepare2}
In the commutative diagram
$$\xymatrix{0\ar[r]& H^1_{\FFc}(\QQ_\ell,\frak{T}_{r,u,v})\ar[r]\ar[d]& H^1(\QQ_\ell,\frak{T}_{r,u,v})\ar[r]\ar[d]&\frac{H^1(\QQ_\ell,\frak{A}_{u,v})}{H^1_{\textup{ur}}(\QQ_\ell,\frak{A}_{u,v})}\ar[d]^{\alpha}\\
0\ar[r]& H^1_{\FFc}(\QQ_\ell,\frak{T}_{r,u+1,v})\ar[r]& H^1(\QQ_\ell,\frak{T}_{r,u+1,v})\ar[r]&\frac{H^1(\QQ_\ell,\frak{A}_{u+1,v})}{H^1_{\textup{ur}}(\QQ_\ell,\frak{A}_{u+1,v})} }$$
the map $\alpha$ is injective.

\end{lemma}
\begin{proof}
Identical to the proof of Proposition~\ref{prop:betainj}.
\end{proof}
Recall the ring $\frak{O}$ and the module ${T}_\frak{O}$ from \S\ref{subsec:notationhypo}.
\begin{prop}
\label{prop:D4}$\,$
\begin{itemize}
\item[(i)] $H^1_f(\QQ_\ell,{T}_\frak{O})=H^1_{\textup{ur}}(\QQ_\ell,{T}_{\frak{O}})$, where $$H^1_f(\QQ_\ell,{T}_\frak{O})=\ker\left(H^1(\QQ_\ell,{T}_\frak{O})\ra H^1(I_\ell,{T}_\frak{O}\otimes\QQ_p)\right).$$
\item[(ii)] $H^1_{\FFc}(\QQ_\ell,\frak{T}_{1,1})=H^1_{\textup{ur}}(\QQ_\ell,\frak{T}_{1,1})$.
\item[(iii)]$H^1_{\FFc}(\QQ_\ell,\frak{T}_{1,1,1})=H^1_{\textup{ur}}(\QQ_\ell,\frak{T}_{1,1,1})$.
\item[(iv)] $H^1_{\FFc}(\QQ_\ell,\overline{T})=\textup{im}\left(H^1_{\textup{ur}}(\QQ_\ell,{T}_{\frak{O}})\ra H^1(\QQ_\ell,\overline{T})\right)=H^1_{\textup{ur}}(\QQ_\ell,\overline{T})$.
\item[(v)] $H^1_{\textup{ur}}(\QQ_\ell,\overline{T})$ is the inverse image of $H^1_{\textup{ur}}(\QQ_\ell,\frak{T}_{1,1})[\mm]$ under the map induced from \textup{(\ref{eqn:gorzeromodm})}.
\end{itemize}
\end{prop}

\begin{proof}
(i) and (ii) follows from \cite[Lemma 1.3.5]{r00} since we assumed (\textbf{H.Tam}), and (iii) follows mimicking the proof of Lemma~\ref{lem:comparefiniteunr}(iii). We next verify (iv). By the very definition of $H^1_{\FFc}(\QQ_\ell,\overline{T})$ (see the beginning of \S\ref{sec:selmerstr}),
$$H^1_{\FFc}(\QQ_\ell,\overline{T})=\textup{im}\left(H^1_{\FFc}(\QQ_\ell,\frak{T}_{1,1})\ra H^1(\QQ_\ell,\overline{T})\right)=\textup{im}\left(H^1_{\textup{ur}}(\QQ_\ell,\frak{T}_{1,1})\ra H^1(\QQ_\ell,\overline{T})\right),$$
where the second equality is thanks to (i). Thus, the assertion (iv) amounts to the statements
\be\label{eqn:unrunrbar1} \textup{im}\left(H^1_{\textup{ur}}(\QQ_\ell,\frak{T}_{1,1})\ra H^1(\QQ_\ell,\overline{T})\right)=H^1_{\textup{ur}}(\QQ_\ell,\overline{T}),\ee
\be\label{eqn:unrunrbar2} \textup{im}\left(H^1_{\textup{ur}}(\QQ_\ell,{T}_{\frak{O}})\ra H^1(\QQ_\ell,\overline{T})\right)=H^1_{\textup{ur}}(\QQ_\ell,\overline{T}). \ee

In order to verify (\ref{eqn:unrunrbar1}), it suffices to check that we have a surjection
$$H^0(I_\ell,\frak{T}_{1,1})\twoheadrightarrow H^0(I_\ell,\overline{T})$$
as $G_{\QQ_\ell}/I_\ell$ has cohomological dimension 1. Taking the $I_\ell$-invariance of the exact sequence
\be\label{eqn:mmexactbar}0\lra \mm_\RR \frak{T}_{1,1}\lra \frak{T}_{1,1}\lra \overline{T}\lra 0\ee
we see that
$$\textup{coker}\left(H^0(I_\ell,\frak{T}_{1,1})\lra H^0(I_\ell,\overline{T})\right)\hookrightarrow H^1(I_\ell,\mm_{\RR}\frak{T}_{1,1}).$$
As the module $H^0(I_\ell,\overline{T})$ is of finite order, the image of the injection above lands in the $\ZZ_p$-torsion submodule $H^1(I_\ell,\mm_{\RR}\frak{T}_{1,1})_{\textup{tors}}$ of $H^1(I_\ell,\mm_{\RR}\frak{T}_{1,1})$. On the other hand,
$$H^1(I_\ell,\mm_{\RR}\frak{T}_{1,1})_{\textup{tors}}\cong (\mm\frak{T}_{1,1}\otimes\QQ_p/\ZZ_p)^{I_{\ell}}/\textup{div}=A^{I_\ell}/\textup{div}=0$$
where
\begin{itemize}
\item $M/\textup{div}$ is short for $M/M_{\textup{div}}$;
\item the second equality is obtained tensoring the exact sequence (\ref{eqn:mmexactbar}) by $\QQ_p/\ZZ_p$ and noting that the exactness is preserved as $\mm_{\RR}\frak{T}_{1,1}$ is $\ZZ_p$-torsion free, and that $\overline{T}\otimes\QQ_p/\ZZ_p=0$;
\item the last equality is (\textbf{H.Tam}).
\end{itemize}
This shows that
$$\textup{coker}\left(H^0(I_\ell,\frak{T}_{1,1})\lra H^0(I_\ell,\overline{T})\right)=0$$
as desired and (\ref{eqn:unrunrbar1}) is verified.

To verify (\ref{eqn:unrunrbar2}), it again suffices to check that
$$\textup{coker}\left(H^0(I_\ell,T_{\frak{O}})\lra H^0(I_\ell,\overline{T})\right).$$
Considering the $I_\ell$-invariance of the exact sequence
$$0\lra T_\frak{O}\stackrel{\pi_\frak{O}}\lra T_{\frak{O}}\lra \overline{T}\lra 0$$
we see that
$$\textup{coker}\left(H^0(I_\ell,T_{\frak{O}})\lra H^0(I_\ell,\overline{T})\right)\hookrightarrow H^1(I_\ell,{T}_{\frak{O}})_{\textup{tors}}.$$
As above, $H^1(I_\ell,{T}_{\frak{O}})_{\textup{tors}}\cong A^{I_\ell}/\textup{div}=0$ and this completes the proof of (iv).

We now prove (v). Consider the sequence
\be\label{eqn:exactbarTQ}0\lra \overline{T}\lra \frak{T}_{1,1} \lra \mathcal{Q}\lra 0\ee
where the arrow $\overline{T}\ra \frak{T}_{1,1}$ is obtained from (\ref{eqn:gorzeromodm}) and $\mathcal{Q}$ is defined by the exactness of this sequence. Taking the $I_\ell$-invariance of the sequence (\ref{eqn:exactbarTQ}), we obtain another exact sequence
$$0\lra \mathcal{Q}_0 \lra H^1(I_\ell,\overline{T}) \lra H^1(I_\ell,\frak{T}_{1,1})$$
where $\mathcal{Q}_0:=\mathcal{Q}^{I_\ell}\Big{/}\frak{T}_{1,1}^{I_\ell}/\overline{T}^{I_\ell}$. Taking the $G_{\QQ_\ell}/I_\ell$-invariance of the final exact sequence, we conclude that
$$\ker\left(H^1(I_\ell,\overline{T})^{\textup{Fr}_\ell=1} \lra H^1(I_\ell,\frak{T}_{1,1})^{\textup{Fr}_\ell=1}\right)=\mathcal{Q}_0^{\textup{Fr}_\ell=1}$$
hence by Lemma~\ref{lem:Q0Ielinv} below that
\be\label{eqn:usefulinjIell1} \ker\left(H^1(I_\ell,\overline{T})^{\textup{Fr}_\ell=1} \lra H^1(I_\ell,\frak{T}_{1,1})^{\textup{Fr}_\ell=1}\right)=0.\ee

Consider now the commutative diagram
$$\xymatrix{0\ar[r]& H^1_{\textup{ur}}(\QQ_\ell,\overline{T})\ar[r]\ar[d]&H^1(\QQ_\ell,\overline{T})\ar[r]\ar[d]&H^1(I_\ell,\overline{T})^{\textup{Fr}_\ell=1}\ar[r]\ar[d]^{\varphi}&0\\
0\ar[r]& H^1_{\textup{ur}}(\QQ_\ell,\frak{T}_{1,1})\ar[r]&H^1(\QQ_\ell,\frak{T}_{1,1})\ar[r]&H^1(I_\ell,\frak{T}_{1,1})^{\textup{Fr}_\ell=1}\ar[r]&0
}$$
(\ref{eqn:usefulinjIell1}) shows that $\varphi$ is injective, and a simple diagram chase yields
$$ H^1_{\textup{ur}}(\QQ_\ell,\overline{T})=\ker\left( H^1(\QQ_\ell,\overline{T})\lra H^1(\QQ_\ell,\frak{T}_{1,1})/H^1_{\textup{ur}}(\QQ_\ell,\frak{T}_{1,1})\right)$$
which is a restatement of (v).
\end{proof}

\begin{lemma}
\label{lem:Q0Ielinv}
$\mathcal{Q}_0^{\textup{Fr}_\ell=1}=0$
\end{lemma}

\begin{proof}
As $\overline{T}^{G_{\QQ_\ell}}=0$, it follows by the proof of \cite[Lemma 2.1.4]{mr02} that $S^{G_{\QQ_\ell}}=0$ for any subquotient $S$ of $\frak{T}$, in particular for $S=\mathcal{Q}_0$.
\end{proof}
\begin{prop}
\label{prop:cartatell}
The local condition at a prime $\ell\neq p$, given by $\FFc$ on the collection $\textup{Quot}(\frak{T})$ is cartesian.
\end{prop}
\begin{proof}
One verifies (\textbf{D1}) using Lemma~\ref{lem:cartpfrakprepare1}, (\textbf{D2}) using Lemma~\ref{lem:cartpfrakprepare2} and  \cite[Lemma 3.7.1]{mr02}. (\textbf{D3}) follows from Proposition~\ref{prop:D4}(i) and Proposition~\ref{prop:D4}(iv).
\end{proof}
\subsubsection{Cartesian properties for the transverse condition} Recall the partial order $\prec$ from Definition~\ref{def:partialorder} on the quadruples (resp., on the triples) of positive integers.
\begin{prop}
\label{prop:carttransverse}
For $\bar{\frak{n}}_0=(r_0,u_0,v_0,w_0)$, suppose $\ell \in \PP_{\bar{\frak{n}}}$ is a Kolyvagin prime in the sense of Definition~\ref{def:kolyprimes}. Then the transverse local condition at $\ell$ is Cartesian on the family  $\{\TT_{\bar{\frak{n}}}\}_{_{\bar{\frak{n}}\prec\bar{\frak{n}}_{0}}}$
$($resp., on the family $\{\frak{T}_{\bar{\frak{n}}}\}_{_{\bar{\frak{n}}\prec \bar{\frak{n}}_0}}\cup\{\frak{T}/\mm\}$$)$.
\end{prop}
\begin{proof}
Suppose $\bar{\frak{n}}=(r,u,v,w)$ and $\bar{\frak{n}}^\prime=(r^ \prime,u^ \prime,v^ \prime,w^ \prime)$ are such that $\bar{\frak{n}}\prec \bar{\frak{n}}^\prime\prec\bar{\frak{n}}_0$. Then we have the following commutative diagram whose rows are exact by Lemma~\ref{lem:transverseproperties}:
$$\xymatrix{0\ar[r]& H^1_{\textup{tr}}(\QQ_\ell,\TT_{\bar{\frak{n}}^\prime})\ar[r]&  H^1(\QQ_\ell,\TT_{\bar{\frak{n}}^\prime})\ar[r]\ar[d]&  H^1_{f}(\QQ_\ell,\TT_{\bar{\frak{n}}^\prime})\ar[r]\ar[d]& 0\\
0\ar[r]& H^1_{\textup{tr}}(\QQ_\ell,\TT_{ \bar{\frak{n}}})\ar[r]& H^1(\QQ_\ell,\TT_{ \bar{\frak{n}}})\ar[r]& H^1_f(\QQ_\ell,\TT_{ \bar{\frak{n}}})\ar[r]& 0
}$$
Here the vertical arrows are induced from the natural surjection $\TT_{ \bar{\frak{n}}^\prime}\twoheadrightarrow \TT_{ \bar{\frak{n}}}$. This shows that  $H^1_{\textup{tr}}(\QQ_\ell,\TT_{ \bar{\frak{n}}^\prime})$ is mapped into $H^1_{\textup{tr}}(\QQ_\ell,\TT_{ \bar{\frak{n}}})$. Furthermore, as the $R_{\bar{\frak{n}}^\prime}$-module $\TT_{ \bar{\frak{n}}^\prime}^{\textup{Fr}_\ell=1}$ (resp., the $R_{ \bar{\frak{n}}}$-module $\TT_{ \bar{\frak{n}}}^{\textup{Fr}_\ell=1}$) is free of rank one, it follows by Lemma~\ref{lem:transverseproperties} and Proposition~\ref{prop:transversefurtherprop}(i) that
$$H^1_{\textup{tr}}(\QQ_\ell,\TT_{ \bar{\frak{n}}^\prime})\twoheadrightarrow H^1_{\textup{tr}}(\QQ_\ell,\TT_{ \bar{\frak{n}}}),$$
i.e., the transverse local condition on the quotients $\TT_{ \bar{\frak{n}}}$ is the same as the propagation of the local condition $H^1_{\textup{tr}}(\QQ_\ell,\TT_{\bar{\frak{n}}_0})$.  This verifies (even a stronger form of) (\textbf{C1}).

 As the quotient
$$H^1(\QQ_\ell,\TT_{\bar{\frak{n}}_0})/H^1_{\textup{tr}}(\QQ_\ell,\TT_{\bar{\frak{n}}_0})\cong H^1_f(\QQ_\ell,\TT_{\bar{\frak{n}}_0})$$
is a free $R_{\bar{\frak{n}}_0}$-module of rank one, (\textbf{C2}) follows from the proof of \cite[Lemma 3.7.1(i)]{mr02}, using the argument in loc.cit. for the multiplication by $[X_1]$, $[X_2]$, $[X_3]$ and $[\varpi]$ maps separately. 	

One verifies (\textbf{D1}) and (\textbf{D2}) for the collection $\{\frak{T}_{\bar{\frak{n}}}\}_{_{\bar{\frak{n}}\prec \bar{\frak{n}}_0}}\cup\{\frak{T}/\mm\}$ in an identical way. It remains to verify ($\textbf{D3}$). To settle that, consider the commutative diagram with exact rows:
$$\xymatrix{0\ar[r]& H^1_{\textup{tr}}(\QQ_\ell,\frak{T}/\mm)\ar[r]\ar@{-->}[d]&  H^1(\QQ_\ell,\frak{T}/\mm)\ar[r]\ar[d]&  H^1_{f}(\QQ_\ell,\frak{T}/\mm)\ar[r]\ar[d]& 0\\
0\ar[r]& H^1_{\textup{tr}}(\QQ_\ell,\frak{T}_{1,1,1})[\frak{m}_{\RR}]\ar[r]& H^1(\QQ_\ell,\frak{T}_{1,1,1})[\mm_{\RR}]\ar[r]& H^1_f(\QQ_\ell,\frak{T}_{1,1,1})[\mm_{\RR}]\ar[r]& 0
}$$
As the $\RR_0$-module $H^1_f(\QQ_\ell,\frak{T}_{1,1,1})$ (resp., the $\texttt{k}$-vector space  $H^1_{f}(\QQ_\ell,\frak{T}/\mm)$) is free of rank one (resp., is one-dimensional), it follows that the right-most arrow is injective and by chasing the diagram it follows that
$$H^1_{\textup{tr}}(\QQ_\ell,\frak{T}/\mm)=\ker\left(H^1(\QQ_\ell,\frak{T}/\mm)\lra \frac{H^1(\QQ_\ell,\frak{T}_{1,1,1})}{H^1_{\textup{tr}}(\QQ_\ell,\frak{T}_{1,1,1})}\right),$$
which is exactly the statement of (\textbf{D4}).
\end{proof}

\subsection{Controlling the Selmer sheaf}
\label{subsec:controlSelmersheaf}
Assume throughout this section that $\chi(\TT)=\chi(\frak{T})=1$ in addition to the running hypotheses. Let $\bar{\frak{n}}=(r,u,v,w)\in(\ZZ_{>0})^4$ and $\bar{\frak{s}}=(r,u,v)\in(\ZZ_{>0})^3$. Define the quotients $R_{\bar{\frak{n}}}=R/(\varpi^r,X_1^u,X_2^v,X_3^w)$ and $\frak{R}_{\bar{\frak{s}}}=\frak{R}/(\varpi^r,X^u,(\gamma-1)^v)$.

 \subsubsection{The upper bound}
 \begin{prop}
 \label{prop:upperbound}We have the following isomorphisms:
 \begin{itemize}
 \item[(i)] $H^1(\QQ_{\Sigma(\FFc)}/\QQ,\overline{T})\stackrel{\sim}{\lra} H^1(\QQ_{\Sigma(\FFc)}/\QQ,\TT_{\bar{\frak{n}}})[\mathcal{M}]$,\\
 \item[(ii)] $H^1(\QQ_{\Sigma(\FFc)}/\QQ,\frak{T}/\mm)\stackrel{\sim}{\lra} H^1(\QQ_{\Sigma(\FFc)}/\QQ,\frak{T}_{1,1,1})[\frak{m}]$,\\
 \item[(iii)] $H^1(\QQ_{\Sigma(\FFc)}/\QQ,\frak{T}_{1,1,1})\stackrel{\sim}{\lra} H^1(\QQ_{\Sigma(\FFc)}/\QQ,\frak{T}_{\bar{\frak{s}}})[(\varpi,X,\gamma-1)]$,\\
\item[(iv)] $H^1_{\FFc(n)}(\QQ,\overline{T}) \stackrel{\sim}{\lra} H^1_{\FFc(n)}(\QQ,{\TT}_{\bar{\frak{n}}})[\mathcal{M}]$,\\
\item[(v)] $H^1_{\FFc(n)}(\QQ,\frak{T}/\mm) \stackrel{\sim}{\lra} H^1_{\FFc(n)}(\QQ,{\frak{T}}_{\bar{\frak{s}}})[\frak{m}]$.
 \end{itemize}
 \end{prop}
 \begin{proof}
 (i), (ii) and (iii) follows from the proof of \cite[Lemma 3.5.2]{mr02}; see in particular the displayed equation (7) in loc.cit. (iv) is now verified using (i) and Propositions~\ref{propcartatptt}, \ref{prop:cartatelltt} and \ref{prop:carttransverse}. (v) follows from (ii), (iii) and Propositions~\ref{propcartatptt}, \ref{prop:cartatell} and \ref{prop:carttransverse}.
 \end{proof}
 
 \begin{cor}
 \label{cor:upperbound}
 Let $n\in\NN_{\bar{\frak{n}}}$ \textup{(}resp., $n \in \NN_{\bar{\frak{s}}}$$\textup{)}$  be a core vertex for the Selmer structure $\FFc$ on $\overline{T}$ $($in the sense of Definition~\ref{def:coreverticesforcan}$)$. Then,
 \begin{itemize}
 \item[(i)] the $R_{\bar{\frak{n}}}$-module $\textup{Hom}\left(H^1_{\FFc(n)}(\QQ,\TT_{\bar{\frak{n}}}),\Phi/\oo\right)$ and, 
 \item[(ii)] the $\frak{R}_{\bar{\frak{s}}}$-module $\textup{Hom}\left(H^1_{\FFc(n)}(\QQ,\frak{T}_{\bar{\frak{s}}}),\Phi/\oo\right)$
 \end{itemize}
 are both cyclic.
 \end{cor}
 \begin{proof}
 By Proposition~\ref{prop:upperbound}(iii), it follows that
 $$\textup{Hom}\left(H^1_{\FFc(n)}(\QQ,\TT_{\bar{\frak{n}}}),\Phi/\oo\right)\Big{/}\mathcal{M}\cong \textup{Hom}\left(H^1_{\FFc(n)}(\QQ,\overline{T}),\Phi/\oo\right).$$
 Since the \texttt{k}-vector space $\textup{Hom}\left(H^1_{\FFc(n)}(\QQ,\overline{T}),\Phi/\oo\right)$ is one-dimensional (thanks to our assumption that $n$ is a core vertex and that $\chi(\TT)=1$), it follows  $\textup{Hom}\left(H^1_{\FFc(n)}(\QQ,\TT_{\bar{\frak{n}}}),\Phi/\oo\right)$ is a cyclic $R_{\bar{\frak{n}}}$-module by Nakayama's Lemma. The statement for $\frak{T}_{\bar{\frak{s}}}$ is proved in an identical fashion, using Proposition~\ref{prop:upperbound}(iv) instead of Proposition~\ref{prop:upperbound}(iii).
 \end{proof}

 \begin{rem}
 There exists infinitely many $n$ as in the statement of Corollary~\ref{cor:upperbound} thanks to \cite[\S4.1]{mr02}.
 \end{rem}
 \subsubsection{The lower bound} As above, let $\bar{\frak{n}}=(r,u,v,w)\in(\ZZ_{>0})^4$ and $\bar{\frak{s}}=(r,u,v)\in(\ZZ_{>0})^3$.

 \begin{prop}
 \label{prop:lowerbound}
  For $n\in\NN_{\bar{\frak{n}}}$  we have,
  $$\textup{length}_{\oo}\left(H^1_{\FFc(n)}(\QQ,\TT_{\bar{\frak{n}}})\right)-\textup{length}_{\oo}\left(H^1_{\FFc(n)^*}(\QQ,\TT_{\bar{\frak{n}}}^*)\right)=\textup{length}_{\oo}(R_{\bar{\frak{n}}}).$$
Similarly for $n \in \NN_{\bar{\frak{s}}}$,  
$$\textup{length}_{\oo}\left(H^1_{\FFc(n)}(\QQ,\frak{T}_{\bar{\frak{s}}})\right)-\textup{length}_{\oo}\left(H^1_{\FFc(n)^*}(\QQ,\frak{T}_{\bar{\frak{s}}}^*)\right)=\textup{length}_{\oo}(\frak{R}_{\bar{\frak{s}}}).$$
  \end{prop}

  \begin{proof}
  By \cite[Corollary 2.3.6]{mr02} it suffices to verify the assertions of the proposition only when $n=1$.

  Let $\TT_{u,v,w}$ be as in \S\ref{subsec:notationhypo}, so that $\TT_{u,v,w}$ is a free $\oo$-module of rank $uvw$. 
  Theorem 4.1.13 of \cite{mr02} (applied with the $\oo[[G_\QQ]]$-representation $T=\TT_{u,v,w}$ and its quotient $\TT_{\bar{\frak{n}}}=\TT_{u,v,w}/\varpi^r
  )$ shows that
  $$\textup{length}_{\oo}\left(H^1_{\FFc}(\QQ,\TT_{\bar{\frak{n}}})\right)-\textup{length}_{\oo}\left(H^1_{\FFc^*}(\QQ,\TT_{\bar{\frak{n}}}^*)\right)=ruvw\cdot\chi(\TT)=\textup{length}_{\oo}(R_{\bar{\frak{n}}}),$$
  as desired. Similarly, repeating the arguments above for the free $\oo$-module $\frak{T}_{u,v}$ (of rank $uv\cdot\dim_{\texttt{k}}(\RR_0)$), we conclude with the second assertion.
  \end{proof}
  \begin{cor}
\label{cor:lowerbound}
For $n$ as in Proposition~\ref{prop:lowerbound},
\begin{itemize}
\item[(i)] $\textup{length}_{\oo}\left(H^1_{\FFc(n)}(\QQ,\TT_{\bar{\frak{n}}})\right) \geq \textup{length}_{\oo}(R_{\bar{\frak{n}}})$,
\item[(ii)] $\textup{length}_{\oo}\left(H^1_{\FFc(n)}(\QQ,\frak{T}_{\bar{\frak{s}}})\right)\geq \textup{length}_{\oo}(\frak{R}_{\bar{\frak{s}}}).$
\end{itemize}
  \end{cor}
  We are now ready to prove Theorem~\ref{thm:corevertexn}: 
  \begin{cor}
  \label{cor:corevertexn}$\,$
  \begin{itemize} 
\item[(i)] Let $n\in\NN_{\bar{\frak{n}}}$  be a core vertex for the Selmer structure $\FFc$ on the residual representation $\overline{T}$. Then, the $R_{\bar{\frak{n}}}$-module $H^1_{\FFc(n)}(\QQ,\TT_{\bar{\frak{n}}})$  is free of rank one and $H^1_{\FFc(n)^*}(\QQ,\TT_{\bar{\frak{n}}}^*)=0$.
  \item[(ii)]   Let $s \in \NN_{\bar{\frak{s}}}$ $\textup{)}$ be a core vertex for $\FFc$ on $\overline{T}$.Then, the $\frak{R}_{\bar{\frak{s}}}$-module $H^1_{\FFc(s)}(\QQ,\frak{T}_{\bar{\frak{s}}})$ is free of rank one and $H^1_{\FFc(s)^*}(\QQ,\frak{T}_{\bar{\frak{s}}}^*)=0$.
  \end{itemize}
  \end{cor}
  \begin{proof}
  It follows from Corollaries~\ref{cor:upperbound} and~\ref{cor:lowerbound} that $\textup{Hom}\left(H^1_{\FFc(n)}(\QQ,\TT_{\bar{\frak{n}}}),\Phi/\oo\right)$ (resp.,  $\textup{Hom}\left(H^1_{\FFc(s)}(\QQ,\frak{T}_{\bar{\frak{s}}}),\Phi/\oo\right)$) is a free $R_{\bar{\frak{n}}}$-module
  (resp., a free $\frak{R}_{\bar{\frak{s}}}$-module) of rank one. The first halves of (i) and (ii) follow from the Gorenstein property of $R$ and $\frak{R}$, c.f., \cite[Prop. 4.9 and 4.10]{grothendieck}. The point is that $\Phi/\oo$ is an injective hull of $\texttt{k}$ and thus a dualizing module for $R$ and $\frak{R}$.

  The second halves (the vanishing statements of the dual Selmer groups) follow from the first halves and Proposition~\ref{prop:lowerbound}.
  \end{proof}
{\scriptsize
\bibliographystyle{halpha}
\bibliography{references}

\begin{thebibliography}{{Fou}13}

\bibitem[BBL14]{bandiniNY}
Andrea Bandini, Francesc Bars, and Ignazio Longhi.
\newblock Characteristic ideals and {I}wasawa theory.
\newblock {\em New York J. Math.}, 20:759--778, 2014.

\bibitem[Bel12a]{bellaicheL}
Jo\"el Bella{\"\i}che.
\newblock {Critical $p$-adic $L$-functions.}
\newblock {\em Invent. Math.}, 189(1):1--60, 2012.

\bibitem[Bel12b]{bellaicheS}
Jo\"el Bella{\"\i}che.
\newblock {Ranks of Selmer groups in an analytic family.}
\newblock {\em Trans. Amer. Math. Soc.}, 364(9):4735--4761, 2012.

\bibitem[B{\"o}c01]{bockleajm2001}
Gebhard B{\"o}ckle.
\newblock On the density of modular points in universal deformation spaces.
\newblock {\em Amer. J. Math.}, 123(5):985--1007, 2001.

\bibitem[B{\"u}y10]{kbbesrankr}
K{\^a}zim B{\"u}y{\"u}kboduk.
\newblock On {E}uler systems of rank {$r$} and their {K}olyvagin systems.
\newblock {\em Indiana Univ. Math. J.}, 59(4):1277--1332, 2010.

\bibitem[B{\"u}y11]{kbb}
K\^az{\i}m B{\"u}y\"ukboduk.
\newblock {$\Lambda$}-adic {K}olyvagin systems.
\newblock {\em IMRN}, 2011(14):3141--3206, 2011.

\bibitem[B{\"u}y13]{kbbssCM1}
K\^az{\i}m B{\"u}y\"ukboduk.
\newblock O\lowercase{n the }{I}\lowercase{wasawa theory of} {CM}
  \lowercase{fields for supersingular primes}, 2013.
\newblock submitted.

\bibitem[B{\"u}y14]{kbbheegner}
K{\^a}z{\i}m B{\"u}y{\"u}kboduk.
\newblock Big {H}eegner point {K}olyvagin system for a family of modular forms.
\newblock {\em Selecta Math. (N.S.)}, 20(3):787--815, 2014.

\bibitem[B{\"u}y15]{kbbCMabvar}
K{\^a}z{\i}m B{\"u}y{\"u}kboduk.
\newblock Main conjectures for {CM} fields and a {Y}ager-type theorem for
  {R}ubin-{S}tark elements, 2015.
\newblock \emph{IMRN}, to appear. {DOI:10.1093/imrn/rnt140}.

\bibitem[CE56]{cartaneilenberg}
Henri Cartan and Samuel Eilenberg.
\newblock {\em Homological algebra}.
\newblock Princeton University Press, Princeton, N. J., 1956.

\bibitem[dSL97]{lenstradeform}
Bart de~Smit and Hendrik~W. Lenstra, Jr.
\newblock Explicit construction of universal deformation rings.
\newblock In {\em Modular forms and {F}ermat's last theorem ({B}oston, {MA},
  1995)}, pages 313--326. Springer, New York, 1997.

\bibitem[EPW06]{emertonpollackweston}
Matthew Emerton, Robert Pollack, and Tom Weston.
\newblock Variation of {I}wasawa invariants in {H}ida families.
\newblock {\em Invent. Math.}, 163(3):523--580, 2006.

\bibitem[Fla92]{flach}
Matthias Flach.
\newblock A finiteness theorem for the symmetric square of an elliptic curve.
\newblock {\em Invent. Math.}, 109(2):307--327, 1992.

\bibitem[FO12]{fouochiai}
Olivier Fouquet and Tadashi Ochiai.
\newblock {Control theorems for Selmer groups of nearly ordinary deformations.}
\newblock {\em J. Reine Angew. Math.}, 666:163--187, 2012.

\bibitem[{Fou}13]{fouquetcompositio}
Olivier {Fouquet}.
\newblock {Dihedral Iwasawa theory of nearly ordinary quaternionic automorphic
  forms.}
\newblock {\em {Compos. Math.}}, 149(3):356--416, 2013.

\bibitem[FPR94]{FoPR}
Jean-Marc Fontaine and Bernadette Perrin-Riou.
\newblock {About the Bloch and Kato conjectures: Galois cohomology and values
  of $L$-functions. (Autour des conjectures de Bloch et Kato: Cohomologie
  galoisienne et valeurs de fonctions $L$.)}.
\newblock {Providence, RI: American Mathematical Society}, 1994.

\bibitem[Gou01]{gouveadeform}
Fernando~Q. Gouv{\^e}a.
\newblock Deformations of {G}alois representations.
\newblock In {\em Arithmetic algebraic geometry ({P}ark {C}ity, {UT}, 1999)},
  volume~9 of {\em IAS/Park City Math. Ser.}, pages 233--406. Amer. Math. Soc.,
  Providence, RI, 2001.
\newblock Appendix 1 by Mark Dickinson, Appendix 2 by Tom Weston and Appendix 3
  by Matthew Emerton.

\bibitem[Gre94]{greenbergdeform}
Ralph Greenberg.
\newblock Iwasawa theory and {$p$}-adic deformations of motives.
\newblock In {\em Motives ({S}eattle, {WA}, 1991)}, volume~55 of {\em Proc.
  Sympos. Pure Math.}, pages 193--223. Amer. Math. Soc., Providence, RI, 1994.

\bibitem[Gro67]{grothendieck}
A.~Grothendieck.
\newblock {\em {Local cohomology. A seminar given by A. Grothendieck, Harvard
  University, Fall 1961. Notes by R. Hartshorne.}}
\newblock {}, 1967.

\bibitem[Gro71]{sga6}
A.~Grothendieck.
\newblock {\em Th\'eorie des intersections et th\'eor\`eme de
  {R}iemann-{R}och}.
\newblock Lecture Notes in Mathematics, Vol. 225. Springer-Verlag, Berlin,
  1971.
\newblock \uppercase{S}{\'e}minaire de \uppercase{G}{\'e}om{\'e}trie
  \uppercase{A}lg{\'e}brique du Bois-Marie 1966--1967 (SGA 6), Dirig{\'e} par
  P. Berthelot, A. Grothendieck et L. Illusie. Avec la collaboration de D.
  Ferrand, J. P. Jouanolou, O. Jussila, S. Kleiman, M. Raynaud et J. P. Serre.

\bibitem[Hid86a]{hida}
Haruzo Hida.
\newblock Galois representations into {${\rm GL}\sb 2({\bf Z}\sb p[[X]])$}
  attached to ordinary cusp forms.
\newblock {\em Invent. Math.}, 85(3):545--613, 1986.

\bibitem[Hid86b]{hidafamily}
Haruzo Hida.
\newblock {Iwasawa modules attached to congruences of cusp forms.}
\newblock {\em Ann. Sci. \'Ec. Norm. Sup\'er. (4)}, 19(2):231--273, 1986.

\bibitem[How07]{how2}
Benjamin Howard.
\newblock {Variation of Heegner points in Hida families.}
\newblock {\em Invent. Math.}, 167(1):91--128, 2007.

\bibitem[Kat93]{katolectturesonapproach}
Kazuya Kato.
\newblock Lectures on the approach to {I}wasawa theory for {H}asse-{W}eil
  {$L$}-functions via {$B_{\rm dR}$}. {I}.
\newblock In {\em Arithmetic algebraic geometry ({T}rento, 1991)}, volume 1553
  of {\em Lecture Notes in Math.}, pages 50--163. Springer, Berlin, 1993.

\bibitem[Kat04]{ka1}
Kazuya Kato.
\newblock {$p$}-adic {H}odge theory and values of zeta functions of modular
  forms.
\newblock {\em Ast\'erisque}, (295):ix, 117--290, 2004.
\newblock Cohomologies $p$-adiques et applications arithm\'etiques. III.

\bibitem[Kis09a]{kisinserre}
Mark Kisin.
\newblock {Modularity of 2-adic Barsotti-Tate representations.}
\newblock {\em Invent. Math.}, 178(3):587--634, 2009.

\bibitem[Kis09b]{kisinFM}
Mark Kisin.
\newblock {The Fontaine-Mazur conjecture for $\text{GL}_2$.}
\newblock {\em J. Am. Math. Soc.}, 22(3):641--690, 2009.

\bibitem[KPX14]{kpx}
Kiran~S. {Kedlaya}, Jonathan {Pottharst}, and Liang {Xiao}.
\newblock {Cohomology of arithmetic families of $(\varphi ,\Gamma)$-modules.}
\newblock {\em {J. Am. Math. Soc.}}, 27(4):1043--1115, 2014.

\bibitem[KW09]{kharewint}
Chandrashekhar Khare and Jean-Pierre Wintenberger.
\newblock {Serre's modularity conjecture. I.}
\newblock {\em Invent. Math.}, 178(3):485--504, 2009.

\bibitem[Liu12]{rliu}
Ruochuan Liu.
\newblock {T}riangulation of refined families, 2012.
\newblock Preprint, \textup{arXiv:1202.2188}.

\bibitem[Mat89]{matsumura}
Hideyuki Matsumura.
\newblock {\em {Commutative ring theory. Transl. from the Japanese by M. Reid.
  Paperback ed.}}
\newblock {Cambridge etc.: Cambridge University Press}, 1989.

\bibitem[Maz72]{mazur-anom}
Barry Mazur.
\newblock Rational points of abelian varieties with values in towers of number
  fields.
\newblock {\em Invent. Math.}, 18:183--266, 1972.

\bibitem[Maz89]{mazurdeform}
B.~Mazur.
\newblock Deforming {G}alois representations.
\newblock In {\em Galois groups over {${\bf Q}$} ({B}erkeley, {CA}, 1987)},
  volume~16 of {\em Math. Sci. Res. Inst. Publ.}, pages 385--437. Springer, New
  York, 1989.

\bibitem[MR04]{mr02}
Barry Mazur and Karl Rubin.
\newblock Kolyvagin systems.
\newblock {\em Mem. Amer. Math. Soc.}, 168(799):viii+96, 2004.

\bibitem[MR11]{mrdarmon}
Barry Mazur and Karl Rubin.
\newblock Refined class number formulas and {K}olyvagin systems.
\newblock {\em Compos. Math.}, 147(1):56--74, 2011.

\bibitem[Nek06]{nek}
Jan Nekov{\'a}{\v{r}}.
\newblock Selmer complexes.
\newblock {\em Ast\'erisque}, (310):viii+559, 2006.

\bibitem[Och03]{ochiaiajm}
Tadashi Ochiai.
\newblock A generalization of the {C}oleman map for {H}ida deformations.
\newblock {\em Amer. J. Math.}, 125(4):849--892, 2003.

\bibitem[Och05]{ochiaideform}
Tadashi Ochiai.
\newblock Euler system for {G}alois deformations.
\newblock {\em Ann. Inst. Fourier (Grenoble)}, 55(1):113--146, 2005.

\bibitem[Och06]{ochiaitwovarMC}
Tadashi Ochiai.
\newblock On the two-variable {I}wasawa main conjecture.
\newblock {\em Compos. Math.}, 142(5):1157--1200, 2006.

\bibitem[Pot12]{pottharstcyclo}
Jonathan Pottharst.
\newblock {C}yclotomic {I}wasawa theory of motives, 2012.
\newblock Preprint, 25pp.

\bibitem[Pot13]{pottharst}
Jonathan Pottharst.
\newblock Analytic families of finite-slope {S}elmer groups.
\newblock {\em Algebra Number Theory}, 7(7):1571--1612, 2013.

\bibitem[PR94]{pr}
Bernadette Perrin-Riou.
\newblock Th\'eorie d'{I}wasawa des repr\'esentations {$p$}-adiques sur un
  corps local.
\newblock {\em Invent. Math.}, 115(1):81--161, 1994.
\newblock With an appendix by Jean-Marc Fontaine.

\bibitem[PR95]{pr-ast}
Bernadette Perrin-Riou.
\newblock Fonctions {$L$} {$p$}-adiques des repr\'esentations {$p$}-adiques.
\newblock {\em Ast\'erisque}, (229):198, 1995.

\bibitem[Rib85]{ribet85}
Kenneth~A. Ribet.
\newblock On {$l$}-adic representations attached to modular forms. {II}.
\newblock {\em Glasgow Math. J.}, 27:185--194, 1985.

\bibitem[Rub00]{r00}
Karl Rubin.
\newblock {\em Euler systems}, volume 147 of {\em Annals of Mathematics
  Studies}.
\newblock Princeton University Press, Princeton, NJ, 2000.
\newblock Hermann Weyl Lectures. The Institute for Advanced Study.

\bibitem[Ser72]{serreimage}
Jean-Pierre Serre.
\newblock Propri\'et\'es galoisiennes des points d'ordre fini des courbes
  elliptiques.
\newblock {\em Invent. Math.}, 15(4):259--331, 1972.

\bibitem[Ser87]{serreconj}
Jean-Pierre Serre.
\newblock {On the modular representations of degree two of
  $\text{Gal}({\overline {\Bbb Q}}/{\Bbb Q})$. (Sur les repr\'esentations
  modulaires de degr\'e 2 de $\text{Gal}({\overline {\Bbb Q}}/{\Bbb Q})$.)}.
\newblock {\em Duke Math. J.}, 54:179--230, 1987.

\bibitem[Ski14]{skinnersplitcyclo}
Christopher Skinner.
\newblock {Multiplicative reduction and the cyclotomic main conjecture for
  $\mathrm{GL}_2$ }.
\newblock 2014.

\bibitem[SU14]{skinnerurbanmainconj}
Christopher Skinner and Eric Urban.
\newblock The {I}wasawa main conjectures for {$\rm GL_2$}.
\newblock {\em Invent. Math.}, 195(1):1--277, 2014.

\bibitem[Til97]{flttil}
Jacques Tilouine.
\newblock {Hecke algebras and the Gorenstein property.}
\newblock {New York, NY: Springer}, 1997.

\bibitem[TW95]{taylorwiles}
Richard Taylor and Andrew Wiles.
\newblock {Ring-theoretic properties of certain Hecke algebras.}
\newblock {\em Ann. Math. (2)}, 141(3):553--572, 1995.

\bibitem[Wes04]{westonunobstructed}
Tom Weston.
\newblock Unobstructed modular deformation problems.
\newblock {\em Amer. J. Math.}, 126(6):1237--1252, 2004.

\bibitem[Wil95]{wiles}
Andrew Wiles.
\newblock Modular elliptic curves and {F}ermat's last theorem.
\newblock {\em Ann. of Math. (2)}, 141(3):443--551, 1995.

\end{thebibliography}
}
\end{document}